\documentclass[12pt]{amsart}

\usepackage[letterpaper,margin=.6in]{geometry}

\usepackage{url}
\usepackage{color}
\usepackage[all]{xypic}
\usepackage{enumitem}

\usepackage{latexsym}
\usepackage{amssymb}
\usepackage{amsfonts}
\usepackage{amscd}
\usepackage{amsmath,amsthm}
\usepackage{verbatim}

\usepackage{imakeidx}

\usepackage{tikz}
\usetikzlibrary{matrix,arrows,decorations.pathmorphing,fit,cd}
\tikzset{myboxgroup/.style={draw, densely dotted}} % style for the boxed groups

\usepackage[colorlinks,linkcolor=blue,citecolor=black!75!red]{hyperref}
\usepackage{cleveref}
\newtheorem{lemma}{Lemma}[section]
\newtheorem{proposition}[lemma]{Proposition}
\newtheorem{theorem}[lemma]{Theorem}
\newtheorem{corollary}[lemma]{Corollary}

{
}

\theoremstyle{definition}

\newtheorem{example}[lemma]{Example}
\newtheorem{definition}[lemma]{Definition}

\theoremstyle{remark}

\newtheorem{remark}[lemma]{Remark}

\makeatletter
\let\xx@thm\@thm
\AtBeginDocument{\let\@thm\xx@thm}
\makeatother

\crefname{section}{section}{sections}
\Crefname{section}{Section}{Sections}
\crefformat{section}{#2section~#1#3}
\Crefformat{section}{#2Section~#1#3}

\crefformat{subsection}{#2\S#1#3}
\Crefformat{subsection}{#2\S#1#3}
\crefrangeformat{subsection}{\S\S#3#1#4--#5#2#6}
\Crefrangeformat{subsection}{\S\S#3#1#4--#5#2#6}
\crefmultiformat{subsection}{\S\S#2#1#3}{ and~#2#1#3}{, #2#1#3}{ and~#2#1#3}
\Crefmultiformat{subsection}{\S\S#2#1#3}{ and~#2#1#3}{, #2#1#3}{ and~#2#1#3}
\crefrangemultiformat{subsection}{\S\S#3#1#4--#5#2#6}{ and~#3#1#4--#5#2#6}{, #3#1#4--#5#2#6}{ and~#3#1#4--#5#2#6}
\Crefrangemultiformat{subsection}{\S\S#3#1#4--#5#2#6}{ and~#3#1#4--#5#2#6}{, #3#1#4--#5#2#6}{ and~#3#1#4--#5#2#6}

\crefname{definition}{Definition}{Definitions}
\crefformat{definition}{#2Definition~#1#3}
\Crefformat{definition}{#2Definition~#1#3}

\crefname{definitionnodiamond}{Definition}{Definitions}
\crefformat{definitionnodiamond}{#2Definition~#1#3}
\Crefformat{definitionnodiamond}{#2Definition~#1#3}

\crefname{example}{Example}{Examples}
\crefformat{example}{#2Example~#1#3}
\Crefformat{example}{#2Example~#1#3}

\crefname{examplenodiamond}{Example}{Examples}
\crefformat{examplenodiamond}{#2Example~#1#3}
\Crefformat{examplenodiamond}{#2Example~#1#3}

\crefname{remark}{Remark}{Remarks}
\crefformat{remark}{#2Remark~#1#3}
\Crefformat{remark}{#2Remark~#1#3}

\crefname{remarknodiamond}{Remark}{Remarks}
\crefformat{remarknodiamond}{#2Remark~#1#3}
\Crefformat{remarknodiamond}{#2Remark~#1#3}

\crefname{convention}{Convention}{Conventions}
\crefformat{convention}{#2Convention~#1#3}
\Crefformat{convention}{#2Convention~#1#3}

\crefname{notation}{Notation}{Notations}
\crefformat{notation}{#2Notation~#1#3}
\Crefformat{notation}{#2Notation~#1#3}

\crefname{notationnodiamond}{Notation}{Notations}
\crefformat{notationnodiamond}{#2Notation~#1#3}
\Crefformat{notationnodiamond}{#2Notation~#1#3}

\crefname{lemma}{Lemma}{Lemmas}
\crefformat{lemma}{#2Lemma~#1#3}
\Crefformat{lemma}{#2Lemma~#1#3}

\crefname{proposition}{Proposition}{Propositions}
\crefformat{proposition}{#2Proposition~#1#3}
\Crefformat{proposition}{#2Proposition~#1#3}

\crefname{corollary}{Corollary}{Corollaries}
\crefformat{corollary}{#2Corollary~#1#3}
\Crefformat{corollary}{#2Corollary~#1#3}

\crefname{theorem}{Theorem}{Theorems}
\crefformat{theorem}{#2Theorem~#1#3}
\Crefformat{theorem}{#2Theorem~#1#3}

\crefname{assumption}{Assumption}{Assumptions}
\crefformat{assumption}{#2Assumption~#1#3}
\Crefformat{assumption}{#2Assumption~#1#3}

\crefname{enumi}{}{}
\crefformat{enumi}{(#2#1#3)}
\Crefformat{enumi}{(#2#1#3)}

\crefname{equation}{}{}
\crefformat{equation}{(#2#1#3)}
\Crefformat{equation}{(#2#1#3)}

\crefname{align}{}{}
\crefformat{align}{(#2#1#3)}
\Crefformat{align}{(#2#1#3)}

\crefname{proofstep}{Step}{Steps}
\crefformat{proofstep}{#2Step~#1#3}
\Crefformat{proofstep}{#2Step~#1#3}

\crefname{table}{Table}{Tables}
\crefformat{table}{#2Table~#1#3}
\Crefformat{table}{#2Table~#1#3}

\numberwithin{equation}{section}

\newenvironment{pf}{\proof}{\endproof}

\def\CC{{\mathbb C}}

\def\PP{{\mathbb P}}
\def\QQ{{\mathbb Q}}
\def\RR{{\mathbb R}}

\def\ZZ{{\mathbb Z}}

\newcommand{\bbZ}{\mathbb{Z}}
\newcommand{\bbC}{\mathbb{C}}
\newcommand{\bbP}{\mathbb{P}}

\def\0ol{{\bar 0}}
\def\1ol{{\bar 1}}
\def\2ol{{\bar 2}}
\def\ol2{{\bar 2}}
\def\3ol{{\bar 3}}
\def\4ol{{\bar 4}}
\def\5ol{{\bar 5}}
\def\6ol{{\bar 6}}
\def\7ol{{\bar 7}}
\def\8ol{{\bar 8}}
\def\9ol{{\bar 9}}

\def\bold0{{\bf 0}}
\def\bold1{{\bf 1}}
\def\bold2{{\bf 2}} 
\def\bold3{{\bf 3}}
\def\bold4{{\bf 4}}
\def\bold5{{\bf 5}}
\def\bold6{{\bf 6}}
\def\bold7{{\bf 7}}
\def\bold8{{\bf 8}}
\def\bold9{{\bf 9}}

\def\P2Skly{\PP^2_{Skly}}

\def\a{\alpha}
\def\b{\beta}
\def\c{\gamma}
\def\d{\delta}

\def\g{\gamma}

\def\s{\sigma}

\def\ve{\varepsilon}

\def\D{\Delta}
\def\G{\Gamma}
\def\L{\Lambda}

\def\fd{{\mathfrak d}}

\def\fsl{{\mathfrak s}{\mathfrak l}}

\def\sA{{\sf A}}

\def\sD{{\sf D}}

\def\sT{{\sf T}}

\def\sfc{{\sf c}}
\def\sfd{{\sf d}}
\def\sfe{{\sf e}}

\def\sfh{{\sf h}}

\def\sfk{{\sf k}}
\def\sfl{{\sf l}}
\def\sfm{{\sf m}} 
\def\sfn{{\sf n}}

\def\sft{{\sf t}}
\def\sfu{{\sf u}}

\def\sfw{{\sf w}}
\def\sfx{{\sf x}}
\def\sfy{{\sf y}}
\def\sfz{{\sf z}}

\def\sfA{{\sf A}}

\def\sfE{{\sf E}}

\def\sfH{{\sf H}}

\def\sfJ{{\sf J}}

\def\sfM{{\sf M}}
\def\sfN{{\sf N}}

\def\cal{\mathcal}

\def\cE{{\cal E}}
\def\cF{{\cal F}}

\def\cL{{\cal L}}

\def\cM{{\cal M}}

\def\cO{{\cal O}}
\def\cP{{\cal P}}

\def\cV{{\cal V}}

\def\Supp{\operatorname{Supp}}

\def\Ext{\operatorname{Ext}}
\def\Tor{\operatorname{Tor}}

\def\Aut{\operatorname{Aut}}

\def\coh{\operatorname{\mathsf{coh}}}
\def\rank{\operatorname{rank}}

\def\Pic{\operatorname{Pic}}

\def\th{{\rm th}}

\def\dirlim{\mathop{\vtop{\baselineskip -100pt\lineskip -1pt\lineskiplimit 0pt
\setbox0\hbox{lim}\copy0\hbox to \wd0{\rightarrowfill}}}\limits}
\def\invlim{\mathop{\vtop{\baselineskip -100pt\lineskip -1pt\lineskiplimit 0pt
\setbox0\hbox{lim}\copy0\hbox to \wd0{\leftarrowfill}}}\limits}

\def\I11{{1 \kern -0.8pt \! \mbox{l}}}
\def\mumu{{\mu\kern-4.2pt\mu}}
\def\bfmu{{\mu\kern-4.2pt\mu}}
\def\2slash{\backslash \! \backslash}

\newcommand{\NS}{\operatorname{\mathsf{NS}}}
\newcommand{\Div}{\operatorname{Div}}
\newcommand{\CH}{\operatorname{CH}}
\newcommand{\pr}{\operatorname{pr}}
\newcommand{\sfsum}{\operatorname{\mathsf{sum}}}

\renewcommand{\Im}{\operatorname{Im}}
\renewcommand{\colon}{:}

\makeatletter
\def\l@subsection{\@tocline{2}{0pt}{2.75pc}{5pc}{}}
\makeatother

\makeindex

\begin{document}

\title[The characteristic variety for elliptic algebras]{The characteristic variety for\\
Feigin and Odesskii's elliptic algebras}

\author{Alex Chirvasitu, Ryo Kanda, and S. Paul Smith}

\address[Alex Chirvasitu]{Department of Mathematics, University at
  Buffalo, Buffalo, NY 14260-2900, USA.}
\email{achirvas@buffalo.edu}

\address[Ryo Kanda]{Department of Mathematics, Graduate School of
  Science, Osaka City University, 3-3-138, Sugimoto, Sumiyoshi, Osaka, 558-8585, Japan.}
\email{ryo.kanda.math@gmail.com}

\address[S. Paul Smith]{Department of Mathematics, Box 354350,
  University of Washington, Seattle, WA 98195, USA.}
\email{smith@math.washington.edu}

\subjclass[2010]{14A22 (Primary), 16S38, 16W50, 17B37, 14H52 (Secondary)}

\keywords{Elliptic algebra; Sklyanin algebra; characteristic variety; point module; theta functions}

% \thanks{ was supported by the National Science Foundation,
%  Award No.  }

\begin{abstract}
This paper examines an algebraic variety that controls an important part of the structure and representation theory of the algebra $Q_{n,k}(E,\tau)$ introduced by Feigin and Odesskii. The $Q_{n,k}(E,\tau)$'s are a family of quadratic algebras depending on a pair of coprime integers $n>k\ge 1$, an elliptic curve $E$, and a point $\tau\in E$. It is already known that the structure and representation theory of $Q_{n,1}(E,\tau)$ is controlled by the geometry associated to $E$ embedded as a degree $n$ normal curve in the projective space $\mathbb P^{n-1}$, and by the way in which the translation automorphism $z\mapsto z+\tau$ interacts with that geometry. For $k\ge 2$ a similar phenomenon occurs: $(E,\tau)$ is replaced by $(X_{n/k},\sigma)$ where $X_{n/k}\subseteq\mathbb P^{n-1}$ is the characteristic variety of the title and $\sigma$ is an automorphism of it that is determined by the negative continued fraction for $\frac{n}{k}$. There is a surjective morphism $\Phi:E^g \to X_{n/k}$ where $g$ is the length of that continued fraction. The main result in this paper is that $X_{n/k}$ is a quotient of $E^g$ by the action of an explicit finite group. We also prove some assertions made by Feigin and Odesskii. The morphism $\Phi$ is the natural one associated to a particular invertible sheaf $\mathcal L_{n/k}$ on $E^g$. The generalized Fourier-Mukai transform associated to $\mathcal L_{n/k}$ sends the set of isomorphism classes of degree-zero invertible $\mathcal O_E$-modules to the set of isomorphism classes of indecomposable locally free $\mathcal O_E$-modules of rank $k$ and degree $n$. Thus $X_{n/k}$ has an importance independent of the role it plays in relation to $Q_{n,k}(E,\tau)$. The backward $\sigma$-orbit of each point on $X_{n/k}$ determines a point module for $Q_{n,k}(E,\tau)$.
\end{abstract}

\maketitle

\tableofcontents

%%%%%%%%%%%%%%%%%%%%%%%%%%%%%%%%%%%%%%%%%%%%%%%%%%%%%%%%%%%%%%%%%%%%%%%%%%%%%%%%
%%%%%%%%%%%%%%%%%%%%%%%%%%%%%%%%%%%%%%%%%%%%%%%%%%%%%%%%%%%%%%%%%%%%%%%%%%%%%%%%
\section{Introduction}
%%%%%%%%%%%%%%%%%%%%%%%%%%%%%%%%%%%%%%%%%%%%%%%%%%%%%%%%%%%%%%%%%%%%%%%%%%%%%%%%
%%%%%%%%%%%%%%%%%%%%%%%%%%%%%%%%%%%%%%%%%%%%%%%%%%%%%%%%%%%%%%%%%%%%%%%%%%%%%%%%

This is the second of several papers we are writing about the elliptic algebras $Q_{n,k}(E,\tau)$ defined by Feigin and Odesskii in 1989 \cite{FO89}. The first of them, \cite{CKS1}, focused on their definition in terms of generators and relations and established some immediate consequences of that definition.  The present paper concerns an algebraic variety that controls a large part of its structure and representation theory, its characteristic variety. Another, \cite{CKS3}, will provide evidence of this ``control'' by showing
that certain twisted homogeneous coordinate rings for the characteristic variety are quotients of $Q_{n,k}(E,\tau)$.

%%%%%%%%%%%%%%%%%%%%%%%%%%%%%%%%%%%%%%%%%%%%%%%%%%%%%%%%%%%%%%%%%%%%%%%%%%%%%%%%
\subsection{The algebras $Q_{n,k}(E,\tau)$}
%%%%%%%%%%%%%%%%%%%%%%%%%%%%%%%%%%%%%%%%%%%%%%%%%%%%%%%%%%%%%%%%%%%%%%%%%%%%%%%%

These algebras depend on a pair of relatively prime integers $n >k \ge 1$\index{n@$n$}\index{k@$k$}, a complex elliptic curve $E$, and a translation automorphism $\tau:E \to E$, or, what is almost the same thing, a point $\tau \in E$\index{tau@$\tau$}. This notation will apply throughout the paper. For fixed $(E,\tau)$, the algebras form a family parametrized by the rational numbers $>1$; this parametrization is realized by writing each rational number 
$>1$ as $\frac{n}{k}$ with $n>k\ge 1$ being a pair of relatively prime positive integers.

Once and for all, we fix a lattice $\L$ in $\CC$ such that $E=\CC/\L$ and use the symbol $\tau$ for a point in $E$ and for a fixed preimage of it in $\CC$.  In the introduction we assume $\tau$ is not in $\frac{1}{n}\L$.\footnote{We discussed the case $\tau \in \frac{1}{n}\Lambda$ in \cite{CKS1}.}

By definition, $Q_{n,k}(E,\tau)$\index{Q_n,k(E,tau)@$Q_{n,k}(E,\tau)$} is the free algebra $\CC\langle x_0,\ldots,x_{n-1}\rangle$\index{x_i@$x_{i}$} modulo the $n^2$ relations\footnote{Feigin
and Odesskii's original definition uses $x_{k(j-r)}x_{k(i+r)}$ instead of $x_{j-r}x_{i+r}$ but, 
as remarked in \cite[\S3.1.1]{CKS1}, this is simply a change of variables.}
\begin{equation}
\label{the-relns-1}
r_{ij}\; =\; \sum_{r \in \ZZ_n} \frac{\theta_{j-i+(k-1)r}(0)}{\theta_{j-i-r}(-\tau)\theta_{kr}(\tau)} \,\, x_{j-r}x_{i+r}	
\end{equation}
where the indices $i$ and $j$ belong to $\ZZ_n=\ZZ/n\ZZ$ and $\theta_0,\ldots,\theta_{n-1}$ are the theta functions defined in \cref{prop.official.theta.basis} below.  When $\tau \in \frac{1}{n}\L$, $\theta_{kr}(\tau)=0$ for some $r$ so the relations no longer make sense; nevertheless, as explained in \cite[\S3.3]{CKS1}, the definition can be extended to all $\tau$ in $\CC$.

Tate and Van den Bergh showed that $Q_{n,1}(E,\tau)$ is a noetherian ring whose homogeneous components have the same dimensions as the homogeneous components of the polynomial ring on $n$ variables \cite{TvdB96}.  They also proved that $Q_{n,1}(E,\tau)$ has excellent homological properties: it has all the good homological properties that the polynomial ring has.

The limit of $Q_{n,k}(E,\tau)$ as $\tau \to 0$ is the polynomial ring on $n$ variables. In \cite{CKS4} we will extend Tate and Van den Bergh's result to show that, when $\tau$ is not a torsion point in $E$, the algebras have the same Hilbert series as a polynomial ring on $n$ variables.

The algebras $Q_{3,1}(E,\tau)$ and $Q_{4,1}(E,\tau)$, known as the 3- and 4-dimensional Sklyanin algebras, respectively, are particularly well understood \cite{ATV2}, \cite{SS92}, \cite{LS93}. Their structure and representation theory is intimately related to and, indeed, controlled by, the geometry associated to the translation automorphism $x \mapsto x+\tau$  on $E$ embedded as a cubic (respectively, quartic) curve in $\PP^2$ (respectively, $\PP^3$) via the theta functions $\theta_i$, $i \in \ZZ_3$ (respectively, $i \in \ZZ_4$). These curves in the ambient projective spaces are the characteristic varieties for $Q_{3,1}(E,\tau)$ and $Q_{4,1}(E,\tau)$. For all $n \ge 3$, the characteristic variety for $Q_{n,1}(E,\tau)$, is a copy of $E$ embedded as a degree-$n$ elliptic normal curve in $\PP^{n-1}$.

The way in which the characteristic variety for $Q_{n,1}(E,\tau)$ controls its representation theory is illustrated by the results in \cite[Thm.~1.4]{TvdB96} and \cite[\S5]{Stan} on the classification of, and relationships between, 
{\it linear modules} for $Q_{n,1}(E,\tau)$. Linear modules are somewhat like Verma modules, though they are not induced from smaller subalgebras. The results and assertions of Feigin and Odesskii suggest that the characteristic variety for $Q_{n,k}(E,\tau)$ controls a large part of {\it its} representation theory.  The simplest linear modules are the {\it point modules} defined in \Cref{ssect.pt.mods} below. The characteristic variety parametrizes most of them.

%%%%%%%%%%%%%%%%%%%%%%%%%%%%%%%%%%%%%%%%%%%%%%%%%%%%%%%%%%%%%%%%%%%%%%%%%%%%%%%%
\subsection{The characteristic variety}
%%%%%%%%%%%%%%%%%%%%%%%%%%%%%%%%%%%%%%%%%%%%%%%%%%%%%%%%%%%%%%%%%%%%%%%%%%%%%%%%

The {\sf characteristic variety} for $Q_{n,k}(E,\tau)$, which we denote by $X_{n/k}$, is defined to be the image of the morphism $\Phi_{|D_{n/k}|}\colon E^g \to \PP^{n-1}$ associated to the complete linear system $|D_{n/k}|$ where $D_{n/k}$ is the divisor defined in \Cref{subse.lnk}.  We write $\cL_{n/k}$ for $\cO_{E^g}(D_{n/k})$.

The integer $g$ in the previous paragraph is the length of the ``negative continued fraction'' expansion for the rational number $n/k$: there is a unique sequence of integers $n_1,\ldots,n_g$, all $\ge 2$, such that
$$
\frac{n}{k} \;=\; n_1 - \frac{1}{n_2-\frac{1}{\cdots \, -\, \frac{1}{n_g}}}.
$$
The divisor $D_{n/k}$, the sheaf $\cL_{n/k}$, and hence the characteristic variety, are defined in terms of $n_1,\ldots,n_g$:
$$
\cL_{n/k} \; := \;  \big(\cL^{n_1} \boxtimes \cdots \boxtimes \cL^{n_g} \big) \otimes \left( \bigotimes_{j=1}^{g-1} \pr_{j,j+1}^*\cP
\right)
$$
where $\cL^r=\cO_E((0))^{\otimes r}$, $\pr_{j,j+1}:E^g \to E^2$ denotes the projection to the $j^{\rm th}$ and $(j+1)^{\rm th}$
factors of $E^g$, and $\cP$ is the Poincar\'e bundle on $E^2$.

Although the definition of $\cL_{n/k}$ (see \cite[\S3.3]{FO89}) appears mysterious at first, its naturality and significance 
is illustrated by the following fact.

\begin{proposition}
[\Cref{prop.FM.transform}]
The generalized Fourier-Mukai transform ${\bf R}\pr_{1*}(\cL_{n/k} \otimes^{\bf L} \pr_g^*(\, \cdot \,))$ is an auto-equivalence of the 
bounded derived category 
$\sD^b(\coh E)$ and provides a bijection from the set of isomorphism classes of invertible $\cO_E$-modules of degree zero to 
the set of isomorphism classes of  indecomposable locally free $\cO_E$-modules of rank $k$ and 
degree $n$. 
\end{proposition}

We make no claim to originality for this observation (cf., \cite[Rmk.~5.6]{HP2}). 

\subsection{The characteristic variety is a quotient of $E^g$} 
The main result in this paper is 

\begin{theorem}
[\cref{prop.factor.Phi}]
\label{thm.main}
The characteristic variety $X_{n/k}$ is isomorphic to the quotient  $E^g/\Sigma_{n/k}$. 
\end{theorem}

We now explain the notation $\Sigma_{n/k}$ and its action on $E^g$. 
First, recall the natural permutation action of the symmetric group 
$\Sigma_{g+1}$ of order $(g+1)!$ on the product $E^{g+1}$. 
The summation map $E^{g+1} \to E$ is constant on $\Sigma_{g+1}$-orbits
and its kernel is isomorphic to $E^g$ as an algebraic group in many ways;
by choosing one such isomorphism one obtains an action of $\Sigma_{g+1}$ on $E^g$. The group $\Sigma_{n/k}$ is the 
subgroup of $\Sigma_{g+1} \subseteq \Aut(E^g)$ generated by the ``simple reflections'' $s_i$ for which the integer $n_i$ is equal to 2. 
See \Cref{sect.symmetric.group.action,sect.Sigma.n/k} for the definition of $s_i$ and other details. We note that $\Sigma_{g+1}$ 
acts as automorphisms of $E^g$ as an algebraic group.

The proof that  $X_{n/k}$ is isomorphic to $E^g/\Sigma_{n/k}$ occupies most of \cref{se.char-var-2}. The first step towards that is
to show that $\cL_{n/k}$ has a $\Sigma_{n/k}$-equivariant structure; this is a simple consequence of the fact that the divisor $D_{n/k}$ is
stable under the action of $\Sigma_{n/k}$ (see \Cref{prop.lnk.equiv.str}). To prove  \cref{thm.main} we show that the morphism $\Phi_{|D_{n/k}|}:E^g \to \PP^{n-1}$ factors through $E^g/\Sigma_{n/k}$, and that the induced factor morphism $E^g/\Sigma_{n/k} \to \PP^{n-1}$ 
is an isomorphism onto its image. The second of these steps requires an intricate examination of divisors on $E^g$ that are linearly equivalent to $D_{n/k}$. The intermediate results involved in doing this will be useful in other situations.

It is not simply $X_{n/k}$ in isolation that is important, but the pair $(X_{n/k},\s)$ where $\s$ is the automorphism of $E^g/\Sigma_{n/k}$ 
defined in \cref{sect.aut.sigma}. The automorphism $\s$ is induced from an affine automorphism of $\CC^g$, defined in terms of
the continued fraction for $\frac{n}{k}$, that sends $\L^{g}$ to itself and so descends to $E^g$; that automorphism of $E^g$ commutes with the 
action of $\Sigma_{n/k}$. 

\subsubsection{Description of the characteristic variety as a bundle}
The integers $n_1,\ldots,n_g$ determine a partition $J \sqcup I_1 \sqcup \cdots \sqcup I_s=\{1,\ldots,g+1\}$ 
(see  \cref{sect.Xn/k.again} for details). 

\begin{proposition}
[\cref{prop.E^g/Sigma}]
\label{prop.Xn/k.again}
Let $i_\a=|I_\a|$.
\begin{enumerate}
  \item 
  If there is an $i$ such that $n_i$ and $n_{i+1}$ are $\ge 3$, then 
  $$
  X_{n/k} \; \cong \; E^{|J|-1} \times S^{i_1}E \times \cdots \times S^{i_s}E 
  $$
  where $S^rE$ denotes the  symmetric power of $E$ of dimension $r$.
    \item 
For all $(n,k)$, $X_{n/k}$ is a bundle over $E^{|J|+s-1}$ with fibers  
isomorphic to $\PP^{i_1-1} \times \cdots \times  \PP^{i_s-1}$.
\end{enumerate}
\end{proposition}

After finding this description of $E^g/\Sigma_{n/k}$ we learned that it follows from \cite[Prop.~2.9]{AA18}.

In \cref{cor.E^g/Sigma=S^gE}, we use this result, and a result due to Kov\'acs, to 
determine all pairs $(n,k)$ for which $X_{n/k}$ is isomorphic to a symmetric power of  $E$.
It seems likely that similar arguments would allow one to determine those $(n,k)$ for which $X_{n/k}$ has a 
particular bundle structure.

\subsubsection{An \'etale cover of the characteristic variety}
In \cite{CKS5} we give yet another description of the characteristic variety. Let $E[r]$ denote the $r$-torsion subgroup of $E$. 

\begin{theorem}
Let $i_\a=|I_\a|$, $d=\gcd\{i_1,\ldots,i_s\}$, and let $\theta:E[i_1] \times \cdots \times E[i_s] \to E[d]$ be the map 
$\theta(z_1,\ldots,z_s)=(i_1/d)z_1 + \cdots + (i_s/d)z_s$.
There is an \'etale cover
$$
E^{|J|+s-1}  \times \PP^{i_1-1} \times \cdots \times \PP^{i_s-1} \, \longrightarrow \, X_{n/k}
$$
with Galois group $E[i_1] \times \cdots \times E[i_s]$ when $J \ne \varnothing$ and 
$\ker\theta$ when $J = \varnothing$.
\end{theorem}

The automorphism $\s:X_{n/k} \to X_{n/k}$ lifts to an automorphism of this \'etale cover \cite{CKS5}.

\subsection{The point variety for $Q_{n,k}(E,\tau)$}
\label{ssect.pt.mods}
Let $V=\operatorname{span}\{x_0,\ldots,x_{n-1}\}=Q_{n,k}(E,\tau)_1$\index{V@$V$}. If $p \in \PP(V^*)$ we write $p^\perp$ for the subspace of $V$ consisting of the linear forms that vanish at $p$. 

A {\sf point module} for $Q_{n,k}(E,\tau)$ is a graded left $Q_{n,k}(E,\tau)$-module $M=M_0\oplus M_1 \oplus \cdots$ that is generated 
by $M_0$ and has the property that $\dim_\CC M_i=1$ for all $i \ge 0$. Each point module determines a point $p \in \PP(V^*)$,
namely the unique $p$ for which $p^\perp M_0=0$. 
The set of all $p$ that arise in this way is called the {\sf point variety} for $Q_{n,k}(E,\tau)$
even though we  don't know whether it is an algebraic variety when $k \ge 2$; it is when $k=1$.

In \Cref{sect.pt.mods}, we show that every $p \in X_{n/k}$ belongs to the point variety.
To make such a statement we must reconcile the fact that $X_{n/k}$ is defined as a subvariety of $\PP(H^0(E^g,\cL_{n/k})^*)$ whereas 
$p \in \PP(V^*)$. Two steps are required to do this. First, we show there is a canonical isomorphism between 
$H^0(E^g,\cL_{n/k})$ and the space of theta functions $\Theta_{n/k}(\L)$ on $g$ variables defined in \Cref{sect.Theta.n/k}.
Second, we identify each basis element $x_i \in V$ with a particular  $w_i \in \Theta_{n/k}(\L)$. The $w_i$'s are defined in 
\Cref{sssect.wi.basis}. The space $\Theta_{n/k}(\L)$ is an irreducible representation for the finite 
Heisenberg group $H_n$ of order $n^3$, and the $w_i$'s are defined in terms of the $H_n$-action. 
In particular, we obtain an explicit description of the composition $\CC^g \to E^g \stackrel{\Phi}{\longrightarrow} \PP^{n-1}$ as  $\sfz \; \mapsto \; (w_0(\sfz),\ldots,w_{n-1}(\sfz))$.

\subsubsection{}
Suppose $k=1$. Then $g=1$ and $\cL_{n/1}=\cO_E(n(0))$ so the characteristic variety for $Q_{n,1}(E,\tau)$ is the image of  the 
canonical map $E \to \PP\big(H^0(E,\cL_{n/1})^*\big)$. Implicit in this statement is an identification $V=H^0(E,\cL_{n/1})$. 
We identify $E$ with its image in $\PP(V^*)$. 
For all $n \ge 3$, except for $n=4$, the point variety for $Q_{n,1}(E,\tau)$ is $E$  \cite{SPS-pt-mods}. 
The point variety for $Q_{4,1}(E,\tau)$ is the union of $E$ and four additional points, those additional points 
being the vertices of the four singular quadrics that contain $E$ \cite{SS92}.

\subsubsection{}
When $k>1$, the point variety for $Q_{n,k}(E,\tau)$ is not known. In  \cite{FO89}, Feigin and Odesskii showed it 
contains the characteristic variety when $\tau$ is close to 0 and they speak as if the two coincide: see, in particular, paragraph (d) on the
first and second pages of \cite{FO98}. This is not true for $Q_{4,1}(E,\tau)$
nor for $Q_{8,3}(E,\tau)$.

\subsection{The contents of the paper}

In \cref{sect.prelims} we set up notation and collect some results, old and new, that will be used in this and our
subsequent papers. The subgroup $\Sigma_{n/k} \subseteq \Aut(E^g)$ is defined in \Cref{sect.Sigma.n/k}. 
There
we examine the integers $n_1,\ldots,n_g$ appearing in the continued fraction expression for $\frac{n}{k}$, 
and define and examine two related sequences $\sfk=(k_1,\ldots,k_g)$, and $\sfl=(l_1,\ldots,l_g)$. A distinguished 
automorphism $\s:E^g \to E^g$, depending on the data $(\sfk,\sfl,\tau)$, is defined in \cref{sect.aut.sigma}. We show that $\s$
commutes with the action of $\Sigma_{n/k}$ and therefore descends to an automorphism of $X_{n/k}$.
\Cref{sect.theta.fns.1.var,sect.Theta.n/k} record some facts and notation regarding 
a space of theta functions in one variable, $\Theta_{n,c}(\L)$, and a space of theta functions in $g$ variables,  $\Theta_{n/k}(\L)$.
The space $\Theta_{n/k}(\L)$  is naturally isomorphic to $H^0(E^g,\cL_{n/k})$ (\cref{sect.morphism.Phi.n.k}).

The results in \cref{se.char-var-1,se.char-var-2} depend only on the data $(n,k,E)$. They concern the sheaf $\cL_{n/k}$ on $E^g$,
the divisor $D_{n/k}$ defined in \cref{eq:d}, and the characteristic variety $X_{n/k}$.

Because $\cL_{n/k}$ is $\Sigma_{n/k}$-equivariant there is an action of $\Sigma_{n/k}$ on $H^0(E^g,\cL_{n/k})$ and hence
an action of it on $\Theta_{n/k}(\L)$; the latter action can be defined directly. The Heisenberg group of order $n^3$ also acts naturally on 
$\Theta_{n/k}(\L)$ making it an irreducible representation. \Cref{sect.theta.several.var} addresses these matters. We also identify there 
a basis $\{w_i \;| \; i \in \ZZ_n\}$ for $\Theta_{n/k}(\L)$ that behaves nicely with respect to the action of the Heisenberg group. We explain in \cref{ssect.Q_1} how the degree-one component 
of $Q_{n/k}(E,\tau)$ can be identified with $\Theta_{n/k}(\L)$ in such a way that the $w_i$'s are identified with the $x_i$'s in \cref{the-relns-1}; the proof of this relies on a beautiful theta-function identity discovered by Odesskii {\cite[p.~1153]{Od-survey}} (\cref{09847525}).
We provide a detailed proof of Odesskii's identity in \Cref{sec.prf.th.id}.

The last subsection of the paper considers the point modules for $Q_{n,k}(E,\tau)$. 
We show there is a point module associated to the backward $\s$-orbit of each point on $X_{n/k}$; 
these modules were first identified by Feigin and Odesskii. They are {\it not} all the point modules in general: for example, as noted at 
\cite[Rmk.~(ii), p.~270]{SS92}, $Q_{4,1}(E,\tau)$ has four additional point modules for generic $\tau$. At the end of \cref{sect.pt.mods},
 we observe that $Q_{8,3}(E,\tau)$ (for generic $\tau$) has additional point modules corresponding to the points on four lines in 
 $\PP^7$, none of which lies on its characteristic variety which is $E^2 \subseteq \PP^7$.   

An index of notation appears just before the bibliography.

%%%%%%%%%%%%%%%%%%%%%%%%%%%%%%%%%%%%%%%%%%%%%%%%%%%%%%%%%%%%%%%%%%%%%%%%%%%%%%%%
\subsection{Acknowledgements}
%%%%%%%%%%%%%%%%%%%%%%%%%%%%%%%%%%%%%%%%%%%%%%%%%%%%%%%%%%%%%%%%%%%%%%%%%%%%%%%%
A.C. acknowledges partial support through NSF grant DMS-1801011. 

R.K. was a JSPS Overseas Research Fellow, and supported by JSPS KAKENHI Grant Numbers JP16H06337, JP17K14164, and JP20K14288, Leading Initiative for Excellent Young Researchers, MEXT, Japan, and Osaka City University Advanced Mathematical Institute (MEXT Joint Usage/Research Center on Mathematics and Theoretical Physics JPMXP0619217849). R.K. would like to express his deep gratitude to Paul Smith for his hospitality as a host researcher during R.K.'s visit to the University of Washington.

S.P.S. thanks his colleagues Jarod Alper, S\'andor Kov\'acs, Max Lieblich, Amos Turchet, and Bianca Viray for many useful conversations. 

%%%%%%%%%%%%%%%%%%%%%%%%%%%%%%%%%%%%%%%%%%%%%%%%%%%%%%%%%%%%%%%%%%%%
%%%%%%%%%%%%%%%%%%%%%%%%%%%%%%%%%%%%%%%%%%%%%%%%%%%%%%%%%%%%%%%%%%%%
\section{Preliminaries}
\label{sect.prelims}
%%%%%%%%%%%%%%%%%%%%%%%%%%%%%%%%%%%%%%%%%%%%%%%%%%%%%%%%%%%%%%%%%%%%
%%%%%%%%%%%%%%%%%%%%%%%%%%%%%%%%%%%%%%%%%%%%%%%%%%%%%%%%%%%%%%%%%%%%

\subsection{The elliptic curve $E$}
Fix $\eta \in \CC$\index{eta@$\eta$} lying in the upper half-plane. Let $\L=\ZZ+\ZZ\eta$\index{Lambda@$\Lambda$} and $E=\CC/\L$\index{E@$E$}. 
We write $0$ for the zero element in $\CC$ and for its image in $E$. We give $E$ the group structure inherited from $\CC$. 
Thus $0 \to \L \to \CC \to E \to 0$ is an exact sequence of abelian groups.
The $g$-fold Cartesian product of this, namely $0 \to \L^g \to \CC^g \to E^g \to 0$, is also exact.

Let $\sfsum\colon\Div(E) \to E$\index{sum@$\sfsum$} be the map 
$\sfsum\!\big(\sum_{i=1}^r a_i(z_i)\big)=a_1z_1+\cdots+a_rz_r$.

We write $E[n]$\index{E[n]@$E[n]$} for the $n$-torsion subgroup of $E$. 
It equals $\frac{1}{n}\L/\L$ and is isomorphic to $\ZZ_n \times \ZZ_n$, where $\ZZ_{n}=\ZZ/n\ZZ$.

%%%%%%%%%%%%%%%%%%%%%%%%%%%%%%%%%%%%%%%%%%%%%%%%%%%%%%%%%%%%%%%%%%%%%%%%%%%%%%%%
\subsection{The action of $\Sigma_{g+1}$ on $E^g$}
\label{sect.symmetric.group.action}
%%%%%%%%%%%%%%%%%%%%%%%%%%%%%%%%%%%%%%%%%%%%%%%%%%%%%%%%%%%%%%%%%%%%%%%%%%%%%%%%
Always, $\Sigma_r$\index{Sigma_r@$\Sigma_r$} denotes the symmetric group on $r$ letters. 

We write $S^rE$\index{S^r E@$S^{r}E$} for its $r^{\rm th}$ symmetric power, i.e., 
$S^rE$ is the quotient  $E^r/\Sigma_r$ with respect to the natural action of the symmetric group $\Sigma_r$ on the $r$-fold product $E^r$.

A particular action of $\Sigma_{g+1}$ on $E^g$  will play an important role. Its action is analogous to
the action of the Weyl group of type $A_g$ on ``the'' Cartan subalgebra of $\fsl_{g+1}$,
i.e., the natural action of $\Sigma_{g+1}$ on the diagonal $(g+1) \times (g+1)$ matrices of trace zero. 
The analogy becomes apparent if we first define the action of $\Sigma_{g+1}$ on $\ZZ^{g}$ and then 
obtain the  action of $\Sigma_{g+1}$ on $E^g$ by realizing $E^g$ as $\ZZ^g \otimes_\ZZ E$.

If $(z_{1},\ldots,z_{g}) \in E^g$ we define $z_{0}=0$ and $z_{g+1}=0$ (see \cref{defn.s_j} below).

\begin{proposition}
\label{pr.equiv}
Let $\ve :\ZZ^g \to \ZZ^{g+1}$\index{epsilon@$\ve$}, $\omega :\ZZ^{g+1} \to \ZZ^g$\index{omega@$\omega$}, $\sfsum:\ZZ^{g+1} \to \ZZ$\index{sum@$\sfsum$}, and 
$s_i:\ZZ^g \to \ZZ^g$\index{s_i@$s_{i}$}, $1\le i \le g$,  be the $\ZZ$-linear maps
\begin{align}  
\ve(z_1,\ldots,z_g) & \;:=\; (z_1,z_2-z_1,\ldots,z_{g}-z_{g-1},-z_g),
\label{defn.varepsilon}
\\
\omega(z_1,\ldots,z_{g+1}) & \; :=\; \Big(z_1,z_1+z_2,\ldots, \sum_{i=1}^{g}z_i\Big),
\\
\sfsum(z_1,\ldots,z_{g+1}) & \; :=\; z_1+\cdots +z_{g+1},
\\
s_i(z_1,\ldots,z_g) & \; :=\; (z_1,\ldots, z_{i-1}, z_{i-1}-z_i+z_{i+1},z_{i+1},\ldots,z_g).
\label{defn.s_j}
\end{align}

Let $(i,i+1):\ZZ^{g+1} \to \ZZ^{g+1}$, $1 \le i \le g$, be the transposition that switches the $i^{\rm th}$ and $(i+1)^{\rm st}$ coordinates.
There is a split exact sequence and a commuting rectangle of abelian groups
$$
\xymatrix{
0 \ar[r] & \ZZ^g \ar[rr]_\ve  \ar[d]_{s_i} && \ar@/_1pc/[ll]_{\omega} \ZZ^{g+1}   \ar[d]^{(i,i+1)} \ar[rr]^\sfsum   &&  \ZZ \ar[r] & 0
\\
 & \ZZ^g \ar[rr]_\ve    &&  \ZZ^{g+1}. 
}
$$
The same formulas define  a split exact sequence and a commuting rectangle of abelian varieties
$$
\xymatrix{
0 \ar[r] & E^g \ar[rr]_\ve  \ar[d]_{s_i} && \ar@/_1pc/[ll]_{\omega} E^{g+1}   \ar[d]^{(i,i+1)} \ar[rr]^\sfsum   &&  E \ar[r] & 0
\\
 & E^g \ar[rr]_\ve    &&  E^{g+1}. 
}
$$
In particular,
\begin{enumerate}
  \item 
  the  map $(i,i+1) \mapsto s_i$ extends to an isomorphism $\Sigma_{g+1} \to \langle s_i \; | \; 1 \le i \le g\rangle \subseteq \Aut(E^g)$, and
  \item 
the morphism $\ve:E^g \to E^{g+1}$ is equivariant for the actions of $\Sigma_{g+1}$ on $E^g$ and $E^{g+1}$.
\end{enumerate}
\end{proposition}
\begin{pf}
The properties of the first diagram are verified by simple calculations. 
Applying the functor $-\otimes_\ZZ E$ to it produces another split exact sequence and a commuting rectangle.
\end{pf}

Since $s_i:\ZZ^g \to \ZZ^g$, $s_i:E^g \to E^g$ is an automorphism of $E^g$ as an algebraic group.

%%%%%%%%%%%%%%%%%%%%%%%%%%%%%%%%%%%%%%%%%%%%%%%%%%%%%%%%%%%%%%%
%%%%%%%%%%%%%%%%%%%%%%%%%%%%%%%%%%%%%%%%%%%%%%%%%%%%%%%%%%%%%%% 
\subsection{The integers $n_1,\ldots,n_g$ and the subgroup $\Sigma_{n/k} \subseteq \Aut(E^g)$}
\label{sect.Sigma.n/k}
%%%%%%%%%%%%%%%%%%%%%%%%%%%%%%%%%%%%%%%%%%%%%%%%%%%%%%%%%%%%%%%
%%%%%%%%%%%%%%%%%%%%%%%%%%%%%%%%%%%%%%%%%%%%%%%%%%%%%%%%%%%%%%%

If $n_1,\ldots,n_g$ are positive integers, we use the notation\index{[n_1,...,n_g]@$[n_1,\ldots,n_g]$}
\begin{equation}
\label{contd.frac}
[n_1,\ldots,n_g] \;=\; n_{1}-\frac{1}{n_{2}-\frac{1}{\ddots n_{g-1}-\frac{1}{n_{g}}}}
\end{equation}
throughout the paper (when division by zero does not occur in the right-hand side).  

Given the relatively prime integers $n>k\ge 1$ in the definition of $Q_{n,k}(E,\tau)$,
there are unique integers $g\ge 1$\index{g@$g$} and $n_1,\ldots,n_g$\index{n_i@$n_{i}$}, all $\ge 2$, such that 
\begin{equation}
\label{eq:cont}
	\frac{n}{k}\;=\; [n_1,\ldots,n_g].
\end{equation}

Let $s_i \in \Aut(E^g)$ be the automorphism defined by (\ref{defn.s_j}). Define\index{Sigma_n/k@$\Sigma_{n/k}$}
$$
\Sigma_{n/k}  \; :=\; \text{the subgroup generated by $\{s_i \; | \; n_i=2, \; 1 \le i \le g\}$.}
$$
Thus $\Sigma_{n/k} \cong \Sigma_{t_{1}+1}\times\cdots\times \Sigma_{t_{r}+1}$
where $t_{1},\ldots,t_{r}$ are the lengths of the maximal sequences of consecutive $2$'s in $(n_{1},\ldots,n_{g})$. For example,
if $(n_1,\ldots,n_g)=(3,2,4,2,2,2,5)$, then $\Sigma_{n/k}\cong \Sigma_2 \times \Sigma_4$. Since $n/(n-1)=[2,\ldots,2]$, where the number of $2$'s is $n-1$, 
$\Sigma_{n/n-1} \cong \Sigma_{g+1}=\Sigma_{n}$. We have seen that $Q_{n,n-1}(E,\tau)$ is a polynomial ring on $n$ variables for all $\tau$ (\cite[Prop.~5.5]{CKS1}).

\cref{prop.factor.Phi} shows that the characteristic variety for $Q_{n,k}(E,\tau)$ is isomorphic to $E^g/\Sigma_{n/k}$.

%%%%%%%%%%%%%%%%%%%%%%%%%%%%%%%%%%%%%%%%%%%%%%%%%%%%%%%%%%%%%%%%%%%%%%%%%%%%%%%%
\subsection{The matrix $\sD(n_{1},\ldots,n_{g})$}
\label{sect.negative.contd.fracs}
%%%%%%%%%%%%%%%%%%%%%%%%%%%%%%%%%%%%%%%%%%%%%%%%%%%%%%%%%%%%%%%%%%%%%%%%%%%%%%%%

Given a sequence $n_{1},\ldots,n_{g}\in\ZZ$, we define the $g \times g$ matrix\index{D(n_1,...,n_g)@$\sD(n_{1},\ldots,n_{g})$}
\begin{equation}
\label{defn.D.matrix}
\sD(n_{1},\ldots,n_{g}) \; := \; 
\begin{pmatrix}
		n_{1}	& -1	& 			& 			&		\\
		-1		& n_{2}	& -1		& 			&		\\
				& -1	& \ddots	& \ddots	& 		\\
				&		& \ddots	& n_{g-1}	& -1	\\
				&		&			& -1		& n_{g}
	\end{pmatrix}
\end{equation}
and the integer\index{d(n_1,...,n_g)@$d(n_{1},\ldots,n_{g})$}
\begin{equation*}
  d(n_{1},\ldots,n_{g}) \; := \;  \det\sD(n_1,\ldots,n_g).
\end{equation*}
Note that $d(n_{1},\ldots,n_{g})=d(n_{g},\ldots,n_{1})$ because $\sD(n_{g},\ldots,n_{1}) $ is obtained from $\sD(n_{1},\ldots,n_{g}) $ by conjugation by the permutation matrix that reverses the order of $1,2,\ldots,g$. It is often convenient to have $d(n_{1},\ldots,n_{g})$ defined for $g=0$ and $g=-1$ so we adopt the conventions
\begin{align*}
d(n_{1},\ldots,n_{0}) &  \; := \; 1,
\\
d(n_{1},\ldots,n_{-1}) & \; := \;0.
\end{align*}
 The cofactor expansions of $\sD(n_{1},\ldots,n_{g})$ along the first and last columns lead to the equalities 
\begin{align}
  d(n_{2},\ldots,n_{g})n_{1}& \;=\; d(n_{1},\ldots,n_{g})+d(n_{3},\ldots,n_{g}),\label{4293857}\\
  d(n_{1},\ldots,n_{g-1})n_{g}& \;=\;   d(n_{1},\ldots,n_{g})+d(n_{1},\ldots,n_{g-2}).\label{4398529}
\end{align}
These equalities also hold when $g=1$ and $g=2$.

\begin{lemma}\label{le.nkk}
  Let $g$ be a positive integer and $n_{1},\ldots,n_{g}\in\bbZ$. Then
  \begin{equation}
    \label{eq:nkk}
   d(n_1,\ldots,n_{g-1})d(n_2,\ldots,n_{g}) \;=\;  d(n_2,\ldots,n_{g-1})d(n_1,\ldots,n_{g})+1.
  \end{equation}
\end{lemma}
\begin{proof}
  We prove this  by induction on $g$.
  When $g=1$ both sides of (\ref{eq:nkk}) equal $1$.  Suppose $g\geq 2$. Applying \cref{4293857} to $d(n_1,\ldots,n_{g})$ and $d(n_1,\ldots,n_{g-1})$, we obtain
	\begin{align*}
          &d(n_{1},\ldots,n_{g-1})d(n_{2},\ldots,n_{g})\; -\; d(n_{2},\ldots,n_{g-1})d(n_{1},\ldots,n_{g})\\
          &\;=\; \big(n_{1}d(n_{2},\ldots,n_{g-1})-d(n_{3},\ldots,n_{g-1})\big)d(n_{2},\ldots,n_{g})\\
          &\phantom{\;=\;xxxxxxxxxxxxxx}-d(n_{2},\ldots,n_{g-1})\big(n_{1}d(n_{2},\ldots,n_{g})\;-\; d(n_{3},\ldots,n_{g})\big)\\
          &\;=\; d(n_{2},\ldots,n_{g-1})d(n_{3},\ldots,n_{g})\;-\; d(n_{3},\ldots,n_{g-1})d(n_{2},\ldots,n_{g}),
	\end{align*}
	which is equal to $1$ by the induction hypothesis.
\end{proof}

The identity in \cref{eq:nkk} can be stated as
$$
 \det   \begin{pmatrix}
   d(n_1,\ldots,n_{g})    &  -d(n_1,\ldots,n_{g-1})  
   \\
d(n_2,\ldots,n_{g})      &   -d(n_2,\ldots,n_{g-1})  

\end{pmatrix} 
\;=\; 1.
$$

\begin{lemma}\label{lem.d.id}
  Let $g$ be a positive integer and $n_{1},\ldots,n_{g}\in\bbZ$. For each $1\leq i\leq g$,
	\begin{align*}
          d(n_{1},\ldots,n_{g})
          &=n_{i}d(n_{1},\ldots,n_{i-1})d(n_{i+1},\ldots,n_{g})\\
          &\phantom{xxxxx} \;-\; d(n_{1},\ldots,n_{i-2})d(n_{i+1},\ldots,n_{g})-d(n_{1},\ldots,n_{i-1})d(n_{i+2},\ldots,n_{g}).
	\end{align*}
\end{lemma}
\begin{proof}
  Assume $3\leq i\leq g-2$ for simplicity; the other cases are easier. By the cofactor expansion along the $i^{\mathrm{th}}$ column, $d(n_{1},\ldots,n_{g})$ is equal to
	\begin{align*}
          \det
          \begin{pmatrix}
            \ddots	& -1		&		&			&			\\
            -1		& n_{i-2}	& -1	&			&			\\
            &			& -1	& -1		&			\\
            &			&		& n_{i+1}	& -1		\\
            & & & -1 & \ddots
          \end{pmatrix}
                       +n_{i}d(n_{1},\ldots,n_{i-1})d(n_{i+1},\ldots,n_{g})+\det
                       \begin{pmatrix}
                         \ddots	& -1		&		&			&			\\
                         -1		& n_{i-1}	&		&			&			\\
                         & -1		& -1	& 			&			\\
                         &			& -1	& n_{i+2}	& -1		\\
                         & & & -1 & \ddots
                       \end{pmatrix}.
	\end{align*}
	The cofactor expansion down the column with two $(-1)$'s shows that the first summand is equal to
	\begin{align*}
          \det
          \begin{pmatrix}
            \ddots	& -1		&		&			&			\\
            -1		& n_{i-3}	& -1	&			&			\\
            &			&		& -1		&			\\
            &			&		& n_{i+1}	& -1		\\
            & & & -1 & \ddots
		\end{pmatrix}
		-d(n_{1},\ldots,n_{i-2})d(n_{i+1},\ldots,n_{g})
	\end{align*}
	where the first determinant is zero. Therefore
	\begin{align*}
          d(n_{1},\ldots,n_{g})
          &=n_{i}d(n_{1},\ldots,n_{i-1})d(n_{i+1},\ldots,n_{g})\\
          &-d(n_{1},\ldots,n_{i-2})d(n_{i+1},\ldots,n_{g})-d(n_{1},\ldots,n_{i-1})d(n_{i+2},\ldots,n_{g}).\qedhere
	\end{align*}
\end{proof}

\begin{lemma}\label{lem.D.inverse}
	Let $g$ be a positive integer and $n_{1},\ldots,n_{g}\in\bbZ$. Then
	\begin{equation*}
		\sD(n_{1},\ldots,n_{g})^{-1}=\frac{1}{d(n_{1},\ldots,n_{g})}\big(d[i,j]\big)_{i,j}
	\end{equation*}
	where
	\begin{equation*}
		d[i,j]=
		\begin{cases}
			d(n_{1},\ldots,n_{i-1})d(n_{j+1},\ldots,n_{g}) & \text{if $i\leq j$,} \\
			d(n_{1},\ldots,n_{j-1})d(n_{i+1},\ldots,n_{g}) & \text{if $j\leq i$.}
		\end{cases}
	\end{equation*}
\end{lemma}
\begin{proof}
	The $(i,j)$-entry of the symmetric matrix $\sD(n_{1},\ldots,n_{g})\big(d[i,j]\big)_{i,j}$ is
	\begin{equation}\label{eq.D.inverse.d}
		-d[i-1,j]+n_{i}d[i,j]-d[i+1,j].
	\end{equation}
	
	If $i=j$, then \cref{eq.D.inverse.d} equals
	\begin{equation*}
		-\;d(n_{1},\ldots,n_{i-2})d(n_{i+1},\ldots,n_{g})\;+\;n_{i}d(n_{1},\ldots,n_{i-1})d(n_{i+1},\ldots,n_{g})
		\;-\;d(n_{1},\ldots,n_{i})d(n_{i+1},\ldots,n_{g}),
	\end{equation*}
	which is equal to $d(n_{1},\ldots,n_{g})$ by \cref{lem.d.id}.
	
	If $i<j$, then \cref{eq.D.inverse.d} equals
	\begin{equation*}
		\big(-d(n_{1},\ldots,n_{i-2})+n_{i}d(n_{1},\ldots,n_{i-1})-d(n_{1},\ldots,n_{i})\big)\, d(n_{j+1},\ldots,n_{g}),
	\end{equation*}
	which is zero by \cref{4398529}.
	
	Hence $\sD(n_{1},\ldots,n_{g})\big(d[i,j]\big)_{i,j}$ is the identity matrix multiplied by $d(n_{1},\ldots,n_{g})$.
\end{proof}

\begin{lemma}\label{le.combd}
Let $g$ be a positive integer and $n_{1},\ldots,n_{g}\in\bbZ$. Then
  \begin{equation*}
    d(n_{1},\ldots,n_{g})=\sum_{0\le j\le g}\sum_{\mathbf{i}}(-1)^{\frac {g-j}2}n_{i_1}\ldots n_{i_j}
  \end{equation*}
  where
$
    {\bf i} = (i_1<\ldots<i_j)
$
  ranges over all subsequences of $1,\ldots,g$ that alternate in
  parity and such that $i_1$ is odd and $i_j$ has the same parity as $g$. (When $j=0$ the term $n_{i_1}\cdots n_{i_j}$ is defined to be 1.)
\end{lemma}
\begin{proof}
  This is a simple induction on $g$ using \cref{4293857}, left to the reader.
\end{proof}

\subsubsection{The integers $k'$, $k_i$, and $l_i$}
\label{subsubsect.cont.frac}

Let $n$ and $k$ be coprime integers such that  $n>k \ge 1$.

Let $k'$\index{k'@$k'$} be the unique integer such that $kk'\equiv 1$ (mod $n$) and $n>k' \ge 1$.

We define $k_{i}$\index{k_i@$k_{i}$} and $l_{i}$\index{l_i@$l_{i}$} for $0\leq i\leq g+1$ by
\begin{equation*}
	k_{i}:=d(n_{i+1},\ldots,n_{g})\qquad\text{and}\qquad l_{i}:=d(n_{i-1},\ldots,n_{1}).
\end{equation*}
It follows from \cref{4293857} that $k_{i}n_{i}=k_{i-1}+k_{i+1}$ and $l_{i}n_{i}=l_{i-1}+l_{i+1}$ for all $i=1,\ldots,g$.

\begin{proposition}\label{prop.nkl}\leavevmode
\begin{enumerate}
  \item\label{item.prop.nkl.kl} 
  $n=k_{0}>\cdots>k_{g+1}\;=\; 0\; =\; l_{0}<\cdots<l_{g+1}=n$. 
  \item\label{item.prop.nkl.zg} 
$k_{1}=k$ and $l_g=k'$. 
  \item\label{item.prop.nkl.gcd} 
  $\gcd(k_i,k_{i+1}) =\gcd(l_i,l_{i+1})=1$ for all $i=0,\ldots,g$.
  \item\label{item.prop.det.D}
  $n=d(n_1,\ldots,n_g)=\det\sD(n_1,\ldots,n_g)$, $k=d(n_2,\ldots,n_g)$, and $k'=d(n_1,\ldots,n_{g-1})$.
  \item\label{item.prop.nkl.frac}
 For all $1\leq i\leq g$,
	\begin{equation}\label{20987542}
		\frac{k_{i-1}}{k_{i}}\;=\; [n_i,\ldots, n_g] 
		\qquad\text{and}\qquad \frac{l_{i+1}}{l_{i}}\;=\;     [n_i,    \ldots, n_1].
	\end{equation}
\end{enumerate}
\end{proposition}

\begin{proof}
For $i=1,\ldots,g$, let $m_i'$ and $m_i$ be the unique relatively prime positive integers  such that $m'_{i}/m_{i}$ is equal to the right-hand side of the first equality in \cref{20987542}.
We also define $m_0=m'_1$ and $m_{g+1}=0$.
For all $i=1,\ldots,g-1$,
	\begin{equation*}
		\frac{m'_{i}}{m_{i}} \;=\; n_{i}-\frac{1}{m'_{i+1}/m_{i+1}} \;=\; n_{i}-\frac{m_{i+1}}{m'_{i+1}}\, .
	\end{equation*}
	Since $m_{i}$ and $m'_{i}$ are relatively prime, $m_{i}=m'_{i+1}$. Hence for all $i=2,\ldots,g-1$,
	\begin{equation*}
		\frac{m_{i-1}}{m_{i}} \;=\; n_{i}-\frac{m_{i+1}}{m_{i}}\, .
	\end{equation*}
	Therefore $n_{i}m_{i}=m_{i-1}+m_{i+1}$; this equality also holds for $i=1$ and $i=g$.
	As we noted above, we also have $n_{i}k_{i}=k_{i-1}+k_{i+1}$ for all $i=1,\ldots,g$. Since $m_{g}=1=k_{g}$ and $m_{g+1}=0=k_{g+1}$, 
	it follows that $m_{i}=k_{i}$ for all $i=0, \ldots, g+1$. 
	Thus, $k_{i-1}/k_i = m_{i-1}/m_i = m_i'/m_i$; i.e., the first equality in \cref{20987542} holds.
	
	We obtain the second equality in \cref{20987542} from the first by replacing the sequence $n_{i},\ldots,n_{g}$ by $n_{i},\ldots,n_{1}$.
	
	Since $n_{i}\geq 2$, the continued fractions in \cref{20987542} are $>1$. 
	Hence  $k_{i-1}>k_{i}$ and  $l_{i+1}>l_i$.  
	
	Since $n/k=m_1'/m_1=m_0/m_1$, we have $n=k_0$ and $k=k_1$. 
	Also, $k_{1}l_{g} \equiv 1$ (mod $n$) because 
	$$
	l_{g} k_{1} \;= \; d(n_1,\ldots,n_{g-1})d(n_2,\ldots,n_g) \; = \;  d(n_2,\ldots,n_{g-1})d(n_1,\ldots,n_g) +1\;=\; d(n_2,\ldots,n_{g-1})n+1
	$$
	(where the middle equality comes from \cref{le.nkk}). 
	But $k_{1}=k$ and $0<l_{g}<n$ so $l_g=k'$. 
	We also have $l_{g+1}=d(n_g,\ldots,n_1)=d(n_1,\ldots,n_g)=k_0=n$.
	
	Since $k_{i}=m_i=m'_{i+1}$ and $k_{i+1}=m_{i+1}$, $\gcd(k_i,k_{i+1}) =1$. Similarly, replacing $n_1,\ldots,n_g$ by $n_g,\ldots,n_1$, we obtain 
	$\gcd(l_i,l_{i+1}) =1$. 
	
	All the statements in the proposition have now been proved.
\end{proof}

\begin{corollary}
If $n/k=[n_1,\ldots,n_g]$, then $n/k'=[n_g,\ldots,n_1]$.       
\end{corollary}

\subsubsection{Remark} 
The conditions imposed on the sequences $i_1<\cdots<i_j$  in   \Cref{le.combd} are equivalent to requiring that each term $n_{i_1}\ldots n_{i_j}$ in the sum is obtained from $n_1\ldots n_g$ by deleting any number, $0$ up to $\left\lfloor \frac g2\right\rfloor$, of consecutive factors $n_in_{i+1}$.  These conditions imply that $j$ has the same
  parity as $g$, whence $(-1)^{\frac {g-j}2}$ is $\pm 1$ as opposed to   an ambiguous root of unity.

\begin{lemma} 
\label{lem.Zn}
Let $(n_1,\ldots,n_g)$ be any point in $\ZZ^g$. Let $\sD=\sD(n_1,\ldots,n_g)$ and $d:=\det\sD$. 
Let $(A,+)$ be an abelian group and let $A[d]$ and $dA$ denote the kernel and image of the multiplication map $A \stackrel{d}{\longrightarrow} A$, respectively. If $d \ne 0$, there is an exact sequence of abelian groups
\begin{equation}
\label{eq:Tor.seq}
0 \to A[d] \to  A^g \stackrel{\sD}{\longrightarrow} A^g \to A/dA \to 0,
\end{equation}
 where $\sD:A^g \to A^g$ denotes left multiplication by $\sD$.
In particular, 
when $\frac{n}{k}=[n_1,\ldots,n_g]$,  $\ZZ^g/\sD \ZZ^g\cong \ZZ_n$.
\end{lemma}
\begin{pf}
Since $d \ne 0$, 
$\ZZ^g/\sD \ZZ^g$ is a finite group. Its structure is therefore
determined by its invariant factors, which are the same as those of $\sD$.
If $s_1|s_2|\ldots|s_g$ denote the invariant factors of $\sD$, then $s_k=\frac{d_k(\sD)}{d_{k-1}(\sD)}$ where $d_0(\sD)=1$ and $d_k(\sD)$
denotes the greatest common divisor of the determinants of the $k \times k$ minors of $\sD$ 
\cite[Chap.~II, \S\S13--16]{Newman}. 
The $(g-1)\times (g-1)$ minor obtained by deleting the first column and bottom row of $\sD$ is a lower triangular matrix with $-1$'s on the
diagonal so $d_{g-1}(\sD)=1$. The invariant factors for $\sD$, and hence of the group $\ZZ^g/\sD \ZZ^g$, 
are therefore $1,\ldots,1,\det\sD$. Hence $\ZZ^g/\sD \ZZ^g \cong \ZZ_d$.

Since $\Tor^\ZZ_1(A,\ZZ_d)=A[d]$, the exact sequence \cref{eq:Tor.seq} is obtained by applying the functor 
$A\otimes_\ZZ-$ to the exact sequence $0 \to \ZZ^g \stackrel{\sD}{\longrightarrow} \ZZ^g \to \ZZ_d \to 0$.
\end{pf}

\begin{remark}\label{re.inv-fact}
  More generally,  if $s_1|s_2|\ldots|s_g$ 
  are the invariant factors of $D\in M_g(\ZZ)$, then the kernel and cokernel of $D:A^g\to A^g$ are isomorphic to
  \begin{equation*}
    \bigoplus_i A[s_i]\qquad \text{and}\qquad \bigoplus_i A/s_iA,
  \end{equation*}
 respectively.
\end{remark}

%%%%%%%%%%%%%%%%%%%%%%%%%%%%%%%%%%%%%%%%%%%%%%%%%%%%%%%%%%%%%%%
%%%%%%%%%%%%%%%%%%%%%%%%%%%%%%%%%%%%%%%%%%%%%%%%%%%%%%%%%%%%%%% 
\subsection{A distinguished automorphism $\s:E^g \to E^g$}
\label{sect.aut.sigma}
%%%%%%%%%%%%%%%%%%%%%%%%%%%%%%%%%%%%%%%%%%%%%%%%%%%%%%%%%%%%%%%
%%%%%%%%%%%%%%%%%%%%%%%%%%%%%%%%%%%%%%%%%%%%%%%%%%%%%%%%%%%%%%%

Let $\sfk,\sfl,\sfn \in \ZZ^g \subseteq \CC^g$\index{k@$\sfk$}\index{l@$\sfl$}\index{n@$\sfn$} be the points
\begin{align*}
\sfk & \; :=\; (k_1,\ldots,k_g),
\\
\sfl & \;:=\; (l_1,\ldots,l_g),
\\
\sfn & \;:=\; (n,\ldots,n).
\end{align*}
Define $\s:\CC^g \to \CC^g$\index{sigma@$\s$} by
$$
\s(\sfz)\; :=\; \sfz+(\sfk+\sfl-\sfn)\tau.
$$
Thus, $\s(z_1,\ldots,z_g)=(z_1+\tau_1,\ldots,z_g+\tau_g)$ where  $\tau_i =(k_i+l_i-n)\tau$\index{tau_i@$\tau_{i}$}.

Because $(\CC^g,+)$ is an abelian group, all translation automorphisms  of it commute with one another. 
In particular,  $\s$ 
commutes with the translation action of $\L^g$ on $\CC^g$, and therefore induces an automorphism of $E^g=\CC^g/\L^g$  that we  also 
 denote by $\s$.

\begin{proposition}\label{prop.t.sigma.comm} 
Let $\sft=(\tau_1,\ldots,\tau_g)=(\sfk+\sfl-\sfn)\tau \in E^g$.
\begin{enumerate}
\item\label{item.t.sigma.comm.t}
The point $\sft$ is invariant under the action of $\Sigma_{n/k}$.
\item\label{item.t.sigma.comm.sigma}
The action of $\s$ on $E^g$  commutes with the action of $\Sigma_{n/k}$
\item\label{item.t.sigma.comm.des}
The action of $\s$  descends to an automorphism of $E^g/\Sigma_{n/k}$ (that we also denote by $\s$).
\end{enumerate}
\end{proposition}
\begin{proof}
\cref{item.t.sigma.comm.t}
Suppose $n_j=2$. Then $s_j(\sft)$ differs from $\sft$, if it differs at all, only in the $j^{\th}$ coordinate. Set $\tau_{0}:=(k_{0}+l_{0}-n)\tau=0$ and $\tau_{g+1}:=(k_{g+1}+l_{g+1}-n)\tau=0$.
The $j^{\th}$ coordinate of $s_j(\sft)$ is $\tau_{j-1}-\tau_j+\tau_{j+1}$ so  $s_j(\sft)=\sft$ if and only if $2\tau_j=\tau_{j-1}+\tau_{j+1}$;
 i.e.,  if and only if $2(k_j+l_j)=k_{j-1}+l_{j-1}+ k_{j+1}+l_{j+1}$; but 
this equality holds since $n_ik_i=k_{i-1}+k_{i+1}$ and $n_il_i=l_{i-1}+l_{i+1}$ for all $i$.
Hence $s_j(\sft)=\sft$. It follows that $\sft$ is fixed by every element in $\Sigma_{n/k}$.

\cref{item.t.sigma.comm.sigma}
Let $\theta \in \Sigma_{n/k}$. Since $\theta$ is an automorphism of $E^g$ as an algebraic group and $\s$ is translation by an element
$\sft$ such that $\theta(\sft)=\sft$, a calculation shows that $\s$ commutes with $\theta$. Hence $\s$ commutes with the action of $\Sigma_{n/k}$. \cref{item.t.sigma.comm.des} follows immediately. 
\end{proof}

%%%%%%%%%%%%%%%%%%%%%%%%%%%%%%%%%%%%%%%%%%%%%%%%%%%%%%%%%%%%%%%%%%%%%%%%%%%%%%%%
\subsection{Theta functions in one variable}
\label{sect.theta.fns.1.var}
%%%%%%%%%%%%%%%%%%%%%%%%%%%%%%%%%%%%%%%%%%%%%%%%%%%%%%%%%%%%%%%%%%%%%%%%%%%%%%%%

We recall some notation in   \cite{CKS1}.

Given an integer  $n \ge 1$ and a point $c \in \CC$, we write $\Theta_{n,c}(\L)$\index{Theta_n,c(Lambda)@$\Theta_{n,c}(\L)$} for the
space of holomorphic functions $f:\CC\to\CC$ such that
\begin{align*}
f(z+1) & \; = \; f(z) \qquad \text{and}
\\
f(z+\eta) & \; = \; e\big(-nz+ \tfrac{n}{2}+c\big) f(z).
\end{align*}
This is a vector space of dimension $n$. 
In keeping with the notation in the Kiev preprint \cite[p.~32]{FO-Kiev} and in  the first Odesskii-Feigin paper \cite{FO89}, we always 
use the notation\index{Theta_n(Lambda)@$\Theta_{n}(\L)$}
\begin{equation}
\label{eq.basic.theta.fn}
	\Theta_{n}(\L) \;:=\; \Theta_{n,\frac{n-1}{2}}(\L).
\end{equation}

The function\index{theta(z)@$\theta(z)$}
\begin{equation}
\label{defn.Jacobi.theta}
\theta(z)  \;=\;  \sum_{n \in \ZZ} (-1)^n e\big(nz + \tfrac{1}{2}n(n-1)\eta\big)
\end{equation}
is a basis for $\Theta_{1,0}(\L)$. It has a zero of order one at 0 and no other zeros in a fundamental parallelogram for $\Lambda$.

 Once and for all we fix a basis for $\Theta_n(\L)$, namely $\{\theta_0,\ldots,\theta_{n-1}\}$.

\begin{proposition}\label{prop.official.theta.basis} \cite[Prop.~2.6]{CKS1}
	The functions\index{theta_alpha(z)@$\theta_\a(z)$}
	\begin{equation}
	\label{official.theta_alpha}
		\theta_\a(z) \; :=\; e\left(\a z +\tfrac{\a}{2n}+\tfrac{\a(\a-n)}{2n}\eta\right)  \prod_{m=0}^{n-1}  \theta\!\left(z +\tfrac{m}{n}+ \tfrac{\a}{n}\eta\right),
	\end{equation}
	indexed by $\alpha\in\bbZ$, have the following properties:
	\begin{enumerate}
		\item
		$\theta_{\a+n}=\theta_{\a}$,
		\item
		$\{\theta_{0},\ldots,\theta_{n-1}\}$ is a basis for $\Theta_{n}(\L)$,
		\item 
		$\theta_\a(z+\tfrac{1}{n})  \; = \; e\big(\frac{\a}{n}\big) \theta_\a(z)$, 
		\item
		$\theta_\a(z+\tfrac{1}{n}\eta)  \; = \;  e\big(-z-\tfrac{1}{2n}+\tfrac{n-1}{2n}\eta\big)\theta_{\a+1}(z)$,
		\item\label{item.official.theta.basis.minus}
		$\theta_\a(-z)  \; = \;  -  e\big(-nz+\tfrac{\a}{n}\big)\theta_{-\a}(z)$, and
		\item\label{item.official.theta.basis.zeros}
		the zeros of $\theta_{\alpha}$ are $\{\frac{1}{n}(-\a\eta +m) \; | \; 0 \le m \le n-1\}+\Lambda$, which are all simple,
		\item
		\label{eq.torsion.translate}
		for all integers $r \ge 0$,
		$		\theta_{\a}(z+\frac{r}{n}\eta)     
 		 \,=\,                e\big(-rz -\frac{r}{2n}+\frac{rn-r^2}{2n}\eta\big)   \theta_{\a+r}(z)$.
	\end{enumerate}
\end{proposition}

The  {\sf Heisenberg group of order $n^3$} is\index{H_n@$H_{n}$}\index{S@$S$}\index{T@$T$}
$$
H_n \;=\; \langle S,T \; | \; S^n=T^n=[S,T]^n=1, \, [S,[S,T]]=[T,[S,T]]=1\rangle.
$$
It acts on $\Theta_n(\L)$ via the operators
\begin{align*}
(S\cdot f)(z) & \;=\; f\big(z+\tfrac{1}{n}\big)
\\
(T\cdot f)(z) & \;=\; e\big(z+\tfrac{1}{2n}-\tfrac{n-1}{2n}\eta \big)f\big(z+\tfrac{1}{n}\eta\big).
\end{align*}

\begin{lemma}
\label{lem.Hn.action.1}
\cite[Lem.~2.8]{CKS1}
The space $\Theta_n(\L)$ is an irreducible representation of $H_n$.
The action on the basis $\theta_\a$ is 
\begin{align*}
(S\cdot \theta_\a)(z) & \;=\;   e\big(\tfrac{\a}{n} \big) \theta_{\a}(z),
\\
(T\cdot \theta_\a)(z) & \;=\; \theta_{\a+1}(z).
\end{align*}
\end{lemma}

%%%%%%%%%%%%%%%%%%%%%%%%%%%%%%%%%%%%%%%%%%%%%%%%%%%%%%%%%%%%%%%%%%%%%%%%%%%%%%%%
\subsection{Theta functions in $g$ variables: the space $\Theta_{n/k}(\L)$}
\label{sect.Theta.n/k}
%%%%%%%%%%%%%%%%%%%%%%%%%%%%%%%%%%%%%%%%%%%%%%%%%%%%%%%%%%%%%%%%%%%%%%%%%%%%%%%%

Fix a point $\sfc=(c_1,\ldots,c_g) \in \CC^g$.\footnote{Odesskii's survey \cite[p.~1152]{Od-survey} uses 
$\sfc=\frac{1}{2}(n_1,\ldots,,n_g) - (0,1,\ldots, 1)\eta$.}

Define $\Theta_{n/k}(\Lambda)$\index{Theta_n/k(Lambda)@$\Theta_{n/k}(\Lambda)$} to be the $\CC$-vector space consisting of all holomorphic functions $f:\CC^{g}\to \CC$ such that
\begin{equation}\label{eq:od-qp}
	\begin{cases}
		f(z_{1},\ldots,z_{i}+1,\ldots,z_{g})\;=\;f(z_{1},\ldots,z_{g}),\\
		f(z_{1},\ldots,z_{i}+\eta,\ldots,z_{g}) \;=\;e(z_{i-1}-n_{i}z_{i}+ z_{i+1} +c_{i})f(z_{1},\ldots,z_{g})
	\end{cases}
\end{equation}
with the convention that $z_0=z_{g+1}=0$.

\subsubsection{}
\label{25982524}
We write $\sfe_1,\ldots, \sfe_g$\index{e_i@$\sfe_{i}$} for the standard basis for $\ZZ^g$. Thus,
$\sfe_i = (0,\ldots,0,1,0,\ldots,0)$ where the $1$ is in the $i^{\th}$ position.

\begin{proposition}\cite{Od-survey}\label{prop.dim.theta.sp}
	$\dim_{\CC}\Theta_{n/k}(\Lambda)=n$.
\end{proposition}

\begin{proof}
We adopt the convention that $\sfe_{0}=\sfe_{g+1}=0$ and $z_0=z_{g+1}=0$. 

	Suppose $f:\CC^g \to \CC$ is a holomorphic function that is periodic of period $1$ in each variable. There are unique scalars $a_{\alpha}\in\CC$ such that
	\begin{equation*}
		f(z_{1},\ldots,z_{g}) \;= \; \sum_{\alpha\in\ZZ^{g}}a_{\alpha}e(\alpha_{1}z_{1}+\cdots+\alpha_{g}z_{g}).
	\end{equation*}
	Clearly
	\begin{equation*}
		f(z_{1},\ldots,z_{i}+\eta,\ldots,z_{g}) \;=\; \sum_{\alpha\in\ZZ^{g}}a_{\alpha}e(\alpha_{i}\eta)e(\alpha_{1}z_{1}+\cdots+\alpha_{g}z_{g}).
	\end{equation*}
	However, $e(z_{i-1}-n_{i}z_{i}+z_{i+1}+c_{i})f(z_{1},\ldots,z_{g})$ equals
	\begin{align*}
		&\sum_{\alpha\in\ZZ^{g}}a_{\alpha}e(c_{i})e(\alpha_{1}z_{1}+\cdots+(\alpha_{i-1}+1)z_{i-1}+(\alpha_{i}-n_{i})z_{i}+(\alpha_{i+1}+1)z_{i+1}+\cdots+\alpha_{g}z_{g})\\
		&=\sum_{\alpha\in\ZZ^{g}}a_{\alpha+n_{i}\sfe_{i}-\sfe_{i-1}-\sfe_{i+1}}e(c_{i})e(\alpha_{1}z_{1}+\cdots+\alpha_{g}z_{g})
	\end{align*}
	 so $f\in\Theta_{n/k}(\Lambda)$ if and only if 
	\begin{equation}\label{eq.dim.Theta_n/k}
		a_{\alpha}e(\alpha_{i}\eta)\;=\; a_{\alpha-\sfe_{i-1}+n_{i}\sfe_{i}-\sfe_{i+1}}e(c_{i})
	\end{equation}
	for all $\a \in \ZZ^g$.
	
	It follows from \cref{eq.dim.Theta_n/k} that the $a_\a$'s are determined by their values for $\a$ belonging to a set of coset representatives for the subgroup $\sD\ZZ^{g}$ of $\ZZ^g$, where $\sD:=\sD(n_{1},\ldots,n_{g})$ is the matrix defined in \cref{defn.D.matrix}. Hence $\dim\Theta_{n/k}(\Lambda)$ is equal to the cardinality of 
	$\ZZ^{g}/\sD\ZZ^{g}$. 
	By \cref{lem.Zn}, $\ZZ^{g}/\sD\ZZ^{g} \cong \ZZ_n$.
\end{proof}

%%%%%%%%%%%%%%%%%%%%%%%%%%%%%%%%%%%%%%%%%%%%%%%%%%%%%%%%%%%%%%%%%%%%%%%%%%%%%%%%
\subsection{Notation}
\label{sect.notation.1}
%%%%%%%%%%%%%%%%%%%%%%%%%%%%%%%%%%%%%%%%%%%%%%%%%%%%%%%%%%%%%%%%%%%%%%%%%%%%%%%%

Let $X$ be a non-singular complex projective algebraic variety.

If $1 \le i \ne j \le g$ we write $\pr_{ij}\colon X^g \to X \times X$\index{pr_ij@$\pr_{ij}$} for the projection onto the $i^{th}$ and $j^{th}$ components, and
$\pr_{i}:X^g \to X$\index{pr_i@$\pr_{i}$} for the projection onto the $i^{th}$ component.
If $D\subseteq X\times X$  is a closed subscheme  we write $D_{ij}:=\pr_{ij}^{-1}(D)$.

We write $\Delta$\index{Delta@$\Delta$} for the diagonal $\{(x,x) \; | \; x \in X\} \subseteq X\times X$. 
Thus\index{Delta_ij@$\D_{ij}$}
$$
\D_{ij}   \;=\;  \pr_{ij}^{-1}(\D) \;=\; \{(x_1,\ldots,x_g) \in X^g \; | \; x_i=x_j\}.
$$

If $\cL$ is an invertible sheaf on $X$ that is is generated by its global sections, we write 
$$
\Phi_{|\cL|}:X \, \longrightarrow \PP(H^0(X,\cL)^*)
$$
for the morphism $\Phi_{|\cL|}(x) := \{s \in H^0(X,\cL) \; | \; s(x)=0\}$\index{Phi__L_@$\Phi_{|\cL|}$}.

If $D$ is a divisor on $X$ such that  $\cO_X(D)$ is generated by its global sections or, equivalently, 
such that the linear system $|D|$ is  base-point free, we write $\Phi_{|D|}:=\Phi_{|\cO_{X}(D)|}$\index{Phi__D_@$\Phi_{|D|}$}.

\subsubsection{Line bundles and invertible $\cO_X$-modules}
We will use Roman letters like $L$ to denote line bundles and script letters like $\cF$ to denote sheaves of $\cO_X$-modules.

\subsubsection{Preliminaries on abelian varieties}
\label{sect.notation.2}
Let $X$ be a complex abelian variety.

If $x \in X$ we write $T_x:X \to X$\index{T_x@$T_{x}$} for the translation automorphism $a \mapsto a+x$. If $\cF$ is a sheaf of 
$\cO_X$-modules we call $T_x^*\cF$ a {\sf translation} of $\cF$. 

Let $D_{1}$ and $D_{2}$ be divisors on $X$. We write $D_{1}\sim D_{2}$\index{-@$\sim$} if they are linearly equivalent. 
We say that $D_{1}$ and $D_{2}$ are \textsf{numerically equivalent}, denoted $D_1 \equiv D_2$\index{-@$\equiv$}, if $D_1\cdot C =D_2 \cdot C$ for all irreducible curves $C$ on $X$. By \cite[19.3.1]{Fulton-2nd-ed-98} and the fact that $H^{2}(X,\ZZ)$ is torsion-free (since $X$ is a complex torus), $D_1$ and $D_2$ are algebraically equivalent if and only if they are numerically equivalent.
The \textsf{N\'eron-Severi group} of $X$, which we denote by  $\NS(X)$\index{NS(X)@$\NS(X)$}, is the group of divisors on 
$X$ modulo algebraic equivalence, i.e., $\NS(X)=\Div(X)/\!\equiv$. Thus, 
 $D$ is zero in $\NS(X)$ if and only if $\deg_C(i^* \cO_{X}(D))=0$ for all irreducible curves $i:C \hookrightarrow X$.

The defining homomorphism $\Div(X) \to \NS(X)$ factors through the Picard group, $\Pic(X)$. 
We write $\Pic^0(X)$ for the kernel of the 
homomorphism $\Pic(X) \to \NS(X)$. Thus, $\NS(X)=\Pic(X)/\Pic^0(X)$. 

We extend the notation and terminology of algebraic equivalence to  invertible $\cO_X$-modules:  
if $\cL$ and $\cL'$ are algebraically equivalent invertible $\cO_X$-modules we write $\cL \equiv \cL'$\index{-@$\equiv$}.

By \cite[p.~88]{bl}, two ample invertible sheaves on $X$ are algebraically equivalent if and only if one is isomorphic to a translation of the other.  In other words, if $\cL$ is ample and $\cM$ is in $\Pic^0(X)$, then $\cL \otimes \cM \cong T_x^*\cL$ for some $x \in X$ (see \cite[\S8, Thm.~1]{Mum08}).

\begin{proposition}
\label{thm.mumf.p77}
Let $D$ and $D'$ be divisors on $X$. If $D \equiv D'$, then
\begin{enumerate}
\item\label{item.mumf.ample} $D$ is ample if and only if $D'$ is, and
\item\label{item.mumf.veryample} $D$ is very ample if and only if $D'$ is.
\end{enumerate}
\end{proposition}
\begin{proof}  
\cref{item.mumf.ample} This follows from the Nakai-Moishezon Criterion \cite[Thm.~A.5.1, p.~434]{Hart}.

\cref{item.mumf.veryample} Let $\cL=\cO_X(D)$ and $\cL'=\cO_X(D')$. Suppose $D$ is very ample (and therefore ample). By \cref{item.mumf.ample}, $D'$ is also ample.  Thus $\cL$ and $\cL'$ are algebraically equivalent and ample so $\cL' \cong T_x^*\cL$ for some $x \in X$.  Let $\Phi,\Phi':X \to \PP^d$ be the morphisms associated to $\cL$ and $T_x^*\cL$ respectively.  By \cite[Lem~4.6.1]{bl}, there is an automorphism $f:\PP^d\to \PP^d$ such that the diagram
$$
\xymatrix{
X \ar[d]_-{\Phi'}  \ar[r]^-{T_x} & X \ar[d]^-{\Phi}
\\
\PP^d \ar[r]_-{f} & \PP^d
}
$$
commutes. It follows that $\cL'$, and hence $D'$, is also very ample.
\end{proof}

\begin{definition}\label{def.k}
\cite[p.~60]{Mum08} If $\cL$ is an invertible $\cO_X$-module and $D$ a divisor on $X$ we define\index{K(L)@$K(\cL)$}\index{K(D)@$K(D)$}\index{H(D)@$H(D)$}
\begin{align*}
K(\cL) \; & :=\; \{x \in X \; | \; T_x^*\cL\cong \cL\},
\\
K(D)\; &:=\; K(\cO_X(D)),
\\
H(D)\; & :=\; \{x \in X \; | \; T_x^*(D)=D\} \qquad \text{(actual equality of divisors).}
\end{align*}
\end{definition}

\begin{proposition}
\cite[\S6, Appl.~1]{Mum08}
\label{prop.AV.Appl.1}
For an effective divisor $D \in \Div(X)$ the following conditions are equivalent:
\begin{enumerate}
\item $D$ is ample;
\item $K(D)$ is finite;
\item $H(D)$ is finite.  
\end{enumerate}
\end{proposition}

\begin{proposition}
\cite[\S16, Vanishing Theorem]{Mum08}
\label{prop.vanishing}
If $\cL$ is an invertible $\cO_X$-module such that $K(\cL)$ is finite, then there is a unique $i$ such that $H^i(X,\cL) \ne 0$ and
$H^j(X,\cL) = 0$ for all $j \ne i$.
\end{proposition}

%%%%%%%%%%%%%%%%%%%%%%%%%%%%%%%%%%%%%%%%%%%%%%%%%%%%%%%%%%%%%%%%%%%%%%%%%%%%%%%%
\subsection{A surjective morphism $E^{g} \to \PP^{g}$}
\label{sect.map.E^r.toP^r}
%%%%%%%%%%%%%%%%%%%%%%%%%%%%%%%%%%%%%%%%%%%%%%%%%%%%%%%%%%%%%%%%%%%%%%%%%%%%%%%%

The material in this subsection is folklore. Some of it can be found in \cite {Hulek86}. 

Let $r$ be an integer $\ge 3$ and let $g=r-1$.
The divisor $r(0) \in \Div(E)$ is very ample so determines a morphism 
$\Phi_{|r(0)|}:E \to \PP^{g}=\PP(H^0(E,\cO_E(r(0)))^*)$, the image 
of which is an elliptic normal curve of degree $r$. For the rest of this section we identify $E$ with its image under $\Phi_{|r(0)|}$. 

If $H$ is a hyperplane in $\PP^{g}$, then the scheme-theoretic intersection $E\cap H$ is a divisor $(x_1)+\cdots +(x_r)$ 
for which $x_1+\cdots+x_r=0$. Conversely, if $x_1,\ldots,x_r \in E$ are such that $x_1+\cdots+x_r=0$  
there is a unique hyperplane $H \subseteq \PP^{g}$ such that  $E\cap H$ is the divisor $(x_1)+\cdots +(x_r)$. 

Let $(\PP^{g})^\vee$ denote the dual projective space consisting of the hyperplanes in $\PP^{g}$. 

For each $\sfx=(x_1,\ldots,x_{g}) \in E^{g}$ we define $L_\sfx \subseteq \PP^{g}$ to be the unique hyperplane whose scheme-theoretic 
intersection with $E$ is the divisor $(x_1)+\cdots+(x_{g})+(-x_1-\cdots-x_{g})$. The morphism $E^{g} \to (\PP^{g})^\vee$, 
$\sfx \mapsto L_\sfx$, is surjective and factors through the natural map $E^{g} \to S^{g}E$ so giving a surjective 
morphism  
\begin{equation}
\label{map.SrE.toP^r}
S^{g}E \to (\PP^{g})^\vee, \qquad \sfx=(\!(x_1,\ldots,x_{g})\!) \mapsto L_\sfx.
\end{equation}
The degree of this morphism is $g$ because if $\sfx'$ is obtained from $\sfx$ by replacing any $x_i$ by 
$-x_1-\cdots-x_{g}$, then $L_\sfx=L_{\sfx'}$, and these are the only equalities among the $L_\sfx$'s.

Let $\theta \in \Aut(E^{g})$ be the automorphism $\theta(x_1,\ldots,x_{g}) =(x_1,x_2-x_1,\ldots,x_{g}-x_{g-1})$. 

Thus, $L_{\theta(\sfx)}$ is the unique hyperplane whose scheme-theoretic intersection with $E$ is the divisor 
\begin{equation}
\label{deg.n.divisor}
(x_1)\,+\, (x_2-x_1)\,+\, \cdots \,+ \, (x_{g}-x_{g-1})\,+\,(-x_{g}).
\end{equation}

\begin{proposition}
\label{prop.emb.ell}
 The map $\sfx \mapsto L_{\theta(\sfx)}$ is  a surjective  morphism
$$
E^{g} \, \longrightarrow \, (\PP^{g})^\vee
$$
and is equal to $\Phi_{|D|}$ where 
$$
D\;= \;  
(0)  \! \times \! E^{g-1} \,+\, E^{g-1}  \! \times \!  (0)  \, +\, \sum_{j=1}^{g-1} \D_{j,j+1}  \, \in\, \Div(E^{g}).
$$
\end{proposition}
\begin{pf}
Write $f$ for the morphism $\sfx \mapsto  L_{\theta(\sfx)}$. 
Since every hyperplane in $\PP^{g}$ meets $E$ at $r$ points (counted with multiplicity) whose sum is $0$, $f$ is surjective.

If $\sfh$ is a hyperplane in $(\PP^{g})^\vee$ and $f^*\sfh \in \Div(E^{g})$ is its pullback to $E^{g}$, 
then $f=\Phi_{|f^*\sfh|}$.  In particular, if $\sfh$ is the hyperplane consisting of those hyperplanes in 
$\PP^{g}$ that pass through $0\in E$, then $f^*\cO_{\PP^{g}}(\sfh)=\cO_{E^{g}}(f^*\sfh)$.
Clearly, $(x_1,\ldots,x_g) \in f^{-1}(\sfh)$ if and only $0$ lies on the hyperplane whose scheme-theoretic intersection with $E$ is the divisor
in (\ref{deg.n.divisor}).  Thus, $f^{-1}(\sfh) =  \big\{(x_1,\ldots,x_{g}) \; \big\vert \; 0 \in \{x_1,x_2-x_1,\ldots,x_{g}-x_{r-2},-x_{g}\} \big\}$, $f^*\sfh=D$, 
and $f=\Phi_{|D|}$. 
\end{pf}

\begin{corollary} 
\label{cor.maps.Eg.to.Pg}
Let $x\in E$ and let 
$$
D\; := \;
(x)  \! \times \! E^{g-1}\,+\, E^{g-1}  \! \times \!  (x)  \, +\, \sum_{j=1}^{g-1} \D_{j,j+1}  \, \in\, \Div(E^{g}).
$$
The morphism $\Phi_{|D|}\colon E^{g}\to \PP(H^0(E^{g}, \cO_{E^{g}}(D))^*)$ is a surjection $E^{g} \to \PP^{g}$.
\end{corollary}
\begin{proof}
This follows from \cref{prop.emb.ell} because the $D$ in this corollary is a translation of the $D$ in \cref{prop.emb.ell} by $(-x,\ldots,-x)\in E^{g}$.
\end{proof}

%%%%%%%%%%%%%%%%%%%%%%%%%%%%%%%%%%%%%%%%%%%%%%%%%%%%%%%%%%%%%%%%%%%%%%%%%%%%%%%%
\section{The invertible $\cO_{E^g}$-module $\cL_{n/k}$ and the divisor $D_{n/k}$}
\label{se.char-var-1}
%%%%%%%%%%%%%%%%%%%%%%%%%%%%%%%%%%%%%%%%%%%%%%%%%%%%%%%%%%%%%%%%%%%%%%%%%%%%%%%%

In \cref{defn.Lnk} below, following Odesskii and Feigin \cite[\S3.3]{FO89}, we define an invertible sheaf $\cL_{n/k}$ on  $E^g$. 
Although the definition appears mysterious at first, \Cref{prop.FM.transform} explains the significance of 
$\cL_{n/k}$ in terms of vector bundles on $E$. The rest of \cref{se.char-var-1} establishes the basic properties of $\cL_{n/k}$:
for example, \Cref{pr.ample} shows  $\cL_{n/k}$ is ample and \Cref{cor.sect} confirms  the assertion in \cite[(d), pp.~65--66]{FO98} that 
$\dim_\CC H^0(E^g,\cL_{n/k})=n$. These results can be stated in terms of divisors, and in \cref{sss.grph}
we provide alternative proofs of some of these results for a larger class of divisors on $E^g$.

In \Cref{prop.gen.gl.sec} we prove that $\cL_{n/k}$ is generated by its global sections. This too is stated in \cite{FO89}.

\subsection{Definition of $\cL_{n/k}$ and the divisor $D_{n/k}$ on $E^g$}\label{subse.lnk}

Let $\cO_E((0))$ be the degree-one invertible $\cO_E$-module corresponding to the divisor $(0)$, let $\cL^r=\cO_E((0))^{\otimes r}=
\cO_E(r\cdot(0))$, and define\index{L_n/k@$\cL_{n/k}$}
\begin{equation}
\label{defn.Lnk}
\cL_{n/k} \; : =\; \big( \cL^{n_1} \boxtimes \cdots \boxtimes \cL^{n_g} \big) \otimes \left( \bigotimes_{j=1}^{g-1} \pr^*_{j,j+1} \big(\cL^{-1}
\boxtimes \cL^{-1}\big) (\D)\right)
\end{equation}
The term $\cP:=(\cL^{-1} \boxtimes \cL^{-1}) (\D)=\cO_{E^2}(\D-(0) \! \times \! E - E\! \times \! (0))$\index{P@$\cP$}
in \cref{defn.Lnk} is the Poincar\'e bundle on $E \times E$.

We will often use the fact that $\cL_{n/k}=\cO_{E^g}(D_{n/k})$ where\index{D_n/k@$D_{n/k}$}
\begin{equation}\label{eq:d} 
D_{n/k}  \; :=\; \sum_{i=1}^g E^{i-1} \times D_i \times E^{g-i} \, + \, \sum_{j=1}^{g-1} \D_{j,j+1}
\end{equation}
and $D_{i}:=(n_{i}-2+\delta_{i,1}+\delta_{i,g})(0)$. When $g\geq 2$,
\begin{equation}
\label{defn.Di}
D_i \; = \; 
\begin{cases}
(n_i-1)(0) & \text{if $i \in \{1,g\}$}
\\
(n_i-2)(0) & \text{if $2 \le i \le g-1$.}
\end{cases}
\end{equation}

When $k=1$, $g=1$,  $\cL_{n/1}=\cL^n=\cO_E(n(0))$, and $D_{n/1}=n(0)$.

Sometimes we use the notation\index{L_[n_1,...,n_g]@$\cL_{[n_1,\ldots,n_g]}$}\index{D_[n_1,...,n_g]@$D_{[n_1,\ldots,n_g]}$}
\begin{align*}
\cL_{[n_1,\ldots,n_g]} & \; = \; \cL_{n/k},
\\
D_{[n_1,\ldots,n_g]} & \; = \; D_{n/k}.
\end{align*}
We will also use  the notation $\cL_{[n_1,\ldots,n_g]}$ for the sheaf on the right-hand side of \cref{defn.Lnk} when $n_1,\ldots, n_g$
are arbitrary integers. In the same spirit, we define $D_{[n_1,\ldots,n_g]}$ for all $(n_1,\ldots,n_g) \in \ZZ^g$.

\subsubsection{}
Although the next result is known to the experts (see, e.g., \cite[Rmk.~5.6]{HP2}), we could not find a proof in the literature.
Let $\cE(k,n)$\index{E(k,n)@$\cE(k,n)$} denote the full subcategory of the bounded derived category $\sD^b(E):=\sD^{b}(\coh E)$\index{D^b(E)@$\sD^b(E)$} consisting of the locally free 
indecomposable $\cO_E$-modules of rank $k$ and degree $n$ concentrated in homological degree 0.

\begin{proposition}
\label{prop.FM.transform}
The functor $\Phi:= {\bf R} \pr_{1*}\big( \cL_{n/k} \otimes^{\bf L}\pr_g^*(\, \cdot\, )\big)$ is an auto-equivalence of  
$\sD^b(\coh E)$  that sends $\cE(1,0)$ to $\cE(k,n)$.
\end{proposition}
\begin{proof}
We make use of  the well-known fact  (\cite[\S\S11.3 and 14.2]{Polishchuk-book}) that  
$$
  {\bf R} \pr_{1*}\big( \cP \otimes^{\bf L}\pr_2^*(\cE(r,d))\big)  \;  = \; 
 \begin{cases}
 \cE(d,-r) & \text{if $d \ge 1$,}
 \\
  \cE(-d,r)[-1] & \text{if $d \le 0$.}
 \end{cases}
 $$
 
 Since almost all the inverse and direct image functors in this proof are the full derived functors 
between various bounded derived categories of coherent sheaves we omit the symbols ${\bf L}$ and ${\bf R}$.

The result is clear when $g=1$ because in that case $k=1$ and $\Phi=\cL^n  \otimes -$.

Assume $g=2$ and write $\frac{n}{k}=[n_1,n_2]$. Thus
$$ 
\cL_{n/k} \; =\;  \pr_1^*\cL^{n_1} \otimes \pr_{12}^*\cP   \otimes   \pr_2^* \cL^{n_2}  \; =\;  \pr_1^*\cL^{n_1} \otimes \cP   \otimes   \pr_2^* \cL^{n_2}.
$$
Let $\cF \in  \sD^b(E)$. By the Projection Formula \cite[(3.11), p.83]{Huy-FM},
\begin{align*}
\Phi(\cF)
& \;=\; 
\pr_{1*}  \big(\pr_1^* \cL^{n_1}  \otimes  \cP  \otimes  \pr_2^*( \cL^{n_2} \otimes \cF) \big)
\\
& \;=\;    \cL^{n_1}  \otimes  \pr_{1*} \big( \cP  \otimes   \pr_2^*( \cL^{n_2} \otimes \cF) \big).
\end{align*}
Certainly $\Phi$ is an auto-equivalence because it is the composition of the auto-equivalences $\cL^{n_2}  \otimes -$,
$ \pr_{1*} \big( \cP  \otimes   \pr_2^*(\, \cdot \,)\big)$, and $ \cL^{n_1}  \otimes -$.
Assume $\cF \in \cE(1,0)$. 
Then $\cL^{n_2} \otimes \cF \in \cE(1,n_2)$ and, because $n_2 \ge 1$, 
$$
\pr_{1*} \big( \cP  \otimes   \pr_2^*( \cL^{n_2} \otimes \cF) \big)  \; \in \;   \cE(n_2,-1).
$$
The proposition holds when $g=2$ because
$ \cL^{n_1}  \otimes  \pr_{1*} \big( \cP  \otimes   \pr_2^*( \cL^{n_2} \otimes \cF) \big)  \in  \cE(n_2,n_1n_2-1) = \cE(k,n)$.

We now assume that $g\ge 3$ and that the proposition is true for $g-1$.
The next diagram helps to keep track of the calculations (most of which involve the Projection Formula and Flat Base Change
\cite[(3.18), p.~85]{Huy-FM}, cf. the proof of \cite[Prop.~5.10]{Huy-FM}): 
$$
\xymatrix{
&& E^g \ar[dl]_{\pr_{12}} \ar[dr]^\pi
\\
&\ar[dl]_\c E^2 \ar[dr]^\d&& \ar[dl]_{\ve}  E^{g-1} \ar[dr]^{\a}
\\
E  &&  E  && E^{g-2} \ar[dr]^\b
 \\
&&&&& E
}
$$
The morphisms are the obvious projections; for example, $\pr_1=\c\circ \pr_{12}$, $\pr_2=\d \circ \pr_{12}=\ve\pi$, $\pr_g=\b\a\pi$,
and $\a$ is the projection onto the right-most factor in the factorization $ E^{g-1}= E \times E^{g-2}$.
The square in the diagram is a Cartesian square in which all four morphisms are flat. 
We will use the factorization 
$$
\cL_{n/k} \;=\;  \pr_1^*\cL^{n_1} \otimes \pr_{12}^*\cP  \otimes \pi^* \cV
$$
where $\cV$ is the invertible $\cO_{E^{g-1}}$-module
$$ 
\big( \pr_2^*\cL^{n_2} \otimes \pr_{23}^*\cP \big) \otimes \cdots  \otimes  \big( \pr_{g-1}^* \cL^{n_{g-1}}  \otimes  
\pr_{g-1,g}^*\cP\big)  \otimes   \pr_g^* \cL^{n_g}
$$
(we label the components of $E^{g-1}$ with $2,\ldots,g$).
Write $\frac{n}{k}=n_1 - \frac{m}{k}$. 
If $\cF \in \sD^b(E)$, then  
\begin{align*}
\Phi (\cF) & \;=\;  \cL^{n_1} \otimes   \pr_{1*}\big( \pr_{12}^*\cP  \otimes \pi^*\cV  \otimes \pr_g^*\cF\big)
\\
 & \;=\;  \cL^{n_1} \otimes \c_{*}  \pr_{12*}\big(\pr_{12}^*\cP  \otimes \pi^*\cV \otimes  \pi^*\a^*\b^*\cF\big)
 \\
 & \;=\;  \cL^{n_1} \otimes \c_{*}  \big(\cP  \otimes    \pr_{12*}\big( \pi^*\cV  \otimes  \pi^* \a^*\b^* \cF\big) \big)
  \\
 & \;=\;  \cL^{n_1} \otimes   \c_{*}  \big(\cP  \otimes   \pr_{12*} \pi^* ( \cV  \otimes \a^*\b^* \cF) \big)
   \\
 & \; = \;  \cL^{n_1} \otimes  \c_{*}  \big(\cP  \otimes  \d^*  \ve_* ( \cV  \otimes \a^*\b^* \cF)\big).
\end{align*}
If $\cF \in \cE(1,0)$, then $ \ve_* ( \cV  \otimes \a^*\b^* \cF) \in \cE(m,k)$ by the induction hypothesis.
Since $\c_*(\cP \otimes \d^*(\, \cdot\, )\big)$ sends $\cE(m,k)$ to $\cE(k,-m)$, 
$\Phi (\cF) \in  \cL^{n_1} \otimes \cE(k,-m) = \cE(k,n_1k-m)= \cE(k,n)$.

Finally, an induction argument on $g$ shows that $\Phi$ is an auto-equivalence: the $g=2$ case was addressed above;
by the induction hypothesis,  $\ve_* ( \cV  \otimes \a^*\b^*(\, \cdot\,))$ is an auto-equivalence; 
and $\c_{*}  \big(\cP  \otimes  \d^*(\, \cdot\,))$ and $\cL^{n_1} \otimes -$ are auto-equivalences; hence the result.
\end{proof}

\subsubsection{}
There is a more direct interpretation of the $n_i$'s in terms of slope. 
Consider the homomorphism from the Grothendieck group $K_0(\coh E)=K_0(\sD^b(E))$ to $\NS(E)=\ZZ^2$
given by taking the first Chern class, $c_1:\coh E \to \ZZ^2$, namely
$$
c_1(\cF) = \begin{pmatrix} \deg\cF \\ \rank\cF \end{pmatrix}.
$$
We use the basis ${0 \choose 1}=c_1(\cO_E)$ and ${1 \choose 0}=c_1(\cO_{\{0\}})$
for $\ZZ^2$ where $\cO_{\{0\}}$ is the skyscraper sheaf at $0 \in E$. 
Define $\mu:\ZZ^2 \to \QQ \sqcup\{\infty\}=\PP^1_\QQ$ by $\mu{d \choose r}=\frac{d}{r}$.
The actions on $K_0(\sD^b(E))$ that are induced by the auto-equivalences $ {\bf R} \pr_{1*}\big( \cP \otimes^{\bf L}\pr_2^*(\, \cdot \,)\big)$ and $-\otimes \cO_E((0))$ induce actions on $\ZZ^2$ that are given by left multiplication by 
$$
S \;=\; \begin{pmatrix} 0 & -1 \\ 1 & 0 \end{pmatrix} \qquad \text{and} \qquad T  \;=\; \begin{pmatrix} 1 & 1 \\ 0 & 1 \end{pmatrix},
$$
respectively.
The induced action of the functor $\Phi$ in \cref{prop.FM.transform} is therefore left multiplication by 
$T^{n_1}S T^{n_2}S \ldots ST^{n_g}$.
A straightforward calculation shows that
\begin{equation*}
\label{SL2Z-action}
\mu\bigg(   T^{n_1}S T^{n_2}S \ldots ST^{n_g}S \begin{pmatrix} 1 \\ 0  \end{pmatrix} \! \bigg) \;=\; [n_1,\ldots,n_g] \;=\; \frac{n}{k}.
\end{equation*}

\subsubsection{Comparison with Feigin and Odesskii}
The sheaf $\cL_{n/k}$ is isomorphic to the sheaf of holomorphic sections of the line bundle denoted $\overline{\xi}$ in \cite{FO89}.
For $g\ge 2$, the latter sheaf is  
$$
\big( \cL^{n_1} \boxtimes \cdots \boxtimes \cL^{n _g} \big) \otimes \left( \bigotimes_{j=1}^{g-1} \pr^*_{j,j+1} \big((\cL^1
\boxtimes \cL^1)(-\D')\big) \right)
$$
where $\D'=\{(z,-z) \; | \; z \in E\}$.
It follows from the See-Saw Theorem that $2((0) \times E)+ 2(E\times (0) )$ and  $\D+\D'$ 
are linearly equivalent divisors on $E \times E$: 
for all $p \in E-\{0\}$ the restrictions of these divisors to $\{p\} \times E$ are linearly equivalent and so 
are their restrictions to  $E \times \{p\}$. 
Hence, on $E \times E$, 
$$
(n_1+1)((0)\!\times \! E) + (n_2+1)(E \! \times \! (0)) -\D' \, \sim \, (n_1-1)((0)\!\times\! E) + (n_2-1)(E \! \times \! (0)) +\D
$$
The claimed isomorphism for $g=2$ is immediate, and the general case follows easily.

\subsubsection{} 
The sheaf $\cL_{n/k}$ is not as special as it might first appear.

It has the form $(\cL_1 \boxtimes \cdots \boxtimes \cL_g)(\D_{1,2}+\cdots+\D_{g-1,g})$ where each $\cL_i$ is 
an invertible $\cO_E$-module of degree $n_i-2+\delta_{i,1}+\delta_{i,g}$. 
However, if $\cL$ is {\it any} invertible $\cO_{E^g}$-module of this form, then \cref{prop.FM.transform} holds with $\cL$ in place of 
$\cL_{n/k}$. Similarly, if $\sfx \in E^g$, then  \cref{prop.FM.transform} holds with 
$T_\sfx^*\cL_{n/k}$ in place of $\cL_{n/k}$.\footnote{The second observation explains the first. 
Both  $(\cL_1 \boxtimes \cdots \boxtimes \cL_g)(\D_{1,2}+\cdots+\D_{g-1,g}) \otimes \cL_{n/k}^{-1}$ and
$T_\sfx^*\cL_{n/k} \otimes \cL_{n/k}^{-1}$ belong to $\Pic^0(E^g)$ and, since $\cL_{n/k}$ is ample (\cref{pr.ample}), 
every $\cF$ in $\Pic^0(E^g)$ is isomorphic $T_\sfx^*\cL_{n/k} \otimes \cL_{n/k}^{-1}$ for some $\sfx \in E^g$ (see
 \cite[\S8, Thm.~1]{Mum08}). Thus, we could have
replaced $\cL_{n/k}$ by any algebraically equivalent invertible $\cO_{E^g}$-module.}
  Furthermore, if we replace $\cL_{n/k}$ in the definition of 
 the characteristic variety $X_{n/k}$ at the beginning of \cref{se.char-var-2} by any algebraically equivalent invertible 
 $\cO_{E^g}$-module, then the resulting variety would remain the same; i.e., there is a great deal of flexibility in how one chooses
to define $\cL_{n/k}$ and hence $X_{n/k}$. This explains why, in \cite[\S3.3]{FO89}, Feigin and Odesskii allow the line bundles they call 
$\xi_i$ to be any line bundles on $E$ of the appropriate degrees. 

Finally, looking ahead, if $D= D_{\fd_i,z_j}$ is one of the standard divisors of type $(n_{1},\ldots,n_{g})$ in \cref{ssec.ample.Lnk} where $n_{i}\geq 2$ for all $i$,
and $\cL_{n/k}$ is replaced by $\cO_{E^g}(D)$ in   \cref{prop.FM.transform}, then  \cref{prop.FM.transform} is true for {\it some} 
relatively prime integers $n>k \ge 1$.

\subsection{Ampleness of $\cL_{n/k}$}
\label{ssec.ample.Lnk}

Although $D_{n/k}$ and $\cL_{n/k}$ play a central role in the study of $Q_{n,k}(E,\tau)$, it is convenient to introduce a larger class of divisors and invertible $\cO_{E^g}$-modules.

For a point $z\in E$, we define the divisors\index{Delta^z_j,j+1@$\Delta^z_{j,j+1}$}
\begin{equation}
      \Delta^z_{j,j+1} \; := \; \{(z_1,\ldots,z_g)\in E^g\; |\; z_{j+1}=z_j+z\}, \quad 1 \le j \le g-1.
\end{equation}
For divisors $\fd_1,\ldots,\fd_g$ on $E$ 
and points $z_1,\ldots,z_{g-1} \in E$,  we define\index{D_d_i,z_j@$D_{\fd_i,z_j}$}
\begin{equation}\label{eq.st.div}
  D_{\fd_i,z_j}  \; := \; \sum_{i=1}^g \pr_i^*(\fd_i) \,+\, \sum_{j=1}^{g-1}\Delta^{z_j}_{j,j+1}.
\end{equation}
If all $\fd_i$ are effective we call $D_{\fd_i,z_j}$ a {\sf standard divisor} on $E^g$. We say $D_{\fd_i,z_j}$ is {\sf of type $(n_1,\cdots,n_g)$} where
\begin{equation*}
	n_{i}=\deg\fd_{i}+2-\delta_{i,1}-\delta_{i,g},
\end{equation*}
that is, $n_{1}=\deg\fd_1$ when $g=1$, and
\begin{equation*}
  n_i=
  \begin{cases}
    \deg\fd_i+1,& i=1,g\\
    \deg\fd_i+2,&\text{otherwise}
  \end{cases}
\end{equation*}
when $g\geq 2$.

In conjunction with \cref{prop.AV.Appl.1}, the next lemma establishes the ampleness of $\cL_{n/k}$, and $\cO_{E^{g}}(D)$ for standard divisors $D$.

\begin{lemma}\label{le.kl}
Let $(n_1,\ldots,n_g)$ be any point in $\ZZ^g$, and let $\sD=\sD(n_1,\ldots,n_g)$ be the matrix defined in \cref{defn.D.matrix}. Let $D=D_{\fd_{i},z_{j}}$ be a standard divisor of type $(n_{1},\ldots,n_{g})$.
If $d:=\det\sD \ne 0$, then $K(D)\cong E[d]$. 
 In particular, 
\begin{enumerate}
\item\label{item.le.kl.main}
$K(D_{n/k})\cong E[n]$, and
\item\label{item.le.kl.cor}
if $n_{i}\geq 2$ for all $i\neq 1$ (resp.,\ all $i\neq g$) and $n_{1}\geq 1$ (resp., $n_{g}\geq 1$), then $K(D)$
is finite.
\end{enumerate}
\end{lemma}
\begin{proof}
Let $\sfx=(x_{1},\ldots,x_{g})\in E^{g}$. Let $\cL=\cO_{E^{g}}(D)$. By the Seesaw Theorem \cite[\S5, Cor.~6]{Mum08}, 
$T^{*}_{\sfx}\cL$ and $\cL$ are isomorphic if and only if, for each $i=1,\ldots,g$ and for generic $(y_{1},\ldots,y_{i-1},y_{i+1},\ldots,y_{g})\in E^{g-1}$, their restrictions to
\begin{equation}
\label{eq.slice}
Y \;=\; \{(y_1,\ldots,y_{i-1})\} \times E \times \{(y_{i+1},\ldots,y_{g})\} \; \subseteq E^g
\end{equation}
are isomorphic to each other.

The restriction of $T^{*}_{\sfx}D$ to $Y$, regarded as a divisor on $E$ by the natural identification, is 
\begin{equation}\label{eq.rest.trans.div}
T_{x_{i}}^{*}\fd_{i}+(x_{i-1}-x_{i}+y_{i-1}+z_{i-1})+(x_{i+1}-x_{i}+y_{i+1}-z_{i})
\end{equation}
for $2\le i\le g-1$; if $i=1$ (resp.,\ $i=g$), delete the second (resp.,\ the third) term.
When $i \notin \{1,g\}$, $\deg\fd_{i}=n_{i}-2$, and thus
the restrictions of $T^{*}_{\sfz}\cL$ and $\cL$ to $Y$ are isomorphic if  and only if \cref{eq.rest.trans.div} is linearly equivalent to
\begin{equation*}
	\fd_{i}+(y_{i-1}+z_{i-1})+(y_{i+1}-z_{i}),
\end{equation*}
that is, if and only if $-x_{i-1}+n_{i}x_{i}-x_{i+1}=0$. The same formula holds for $i=1$ and $i=g$ if we set  $x_{0}=x_{g+1}=0$. 
Thus, $T^*_\sfx\cL \cong \cL$ if and only if $\sfx$ is in the kernel of  the multiplication map $\sD:E^g \to E^g$.
That kernel is isomorphic to $E[d]$ by \cref{lem.Zn}.

\cref{item.le.kl.main}
By \cref{prop.nkl}\cref{item.prop.det.D}, $\det\sD=d(n_{1},\ldots,n_{g})=n\ne 0$.

\cref{item.le.kl.cor}
Suppose $n_{i}\geq 2$ for all $i$. Write $[n_{1},\ldots,n_{g}]=n/k$ for relatively prime integers $n>k \ge 1$. 
To prove the first claim we will show that  $d(1,n_{1},\ldots,n_{g})\ne 0$.
By \cref{4293857},
\begin{equation*}
	d(1,n_{1},\ldots,n_{g})=d(n_{1},\ldots,n_{g})-d(n_{2},\ldots,n_{g})=n-k>0.
\end{equation*}
So the first claim holds.  The second follows because $d(n_{1},\ldots,n_{g},1)=d(1,n_{1},\ldots,n_{g})\ne 0$. 
\end{proof}

\begin{proposition}\label{pr.ample}
If $n_{i}\geq 2$ for all $i\neq 1$ (resp.,\ all $i\neq g$) and $n_{1}\geq 1$ (resp., $n_{g}\geq 1$), then $\cO_{E^{g}}(D)$ is ample for all standard divisors of type $(n_{1},\ldots,n_{g})$. In particular, $\cL_{n/k}$ is ample.
\end{proposition}
\begin{proof}
By  \cref{le.kl}\cref{item.le.kl.cor}, $K(D)$ is finite, so $D$ is ample by  \cref{prop.AV.Appl.1}.
\end{proof}

\begin{corollary}\label{cor.xi-h0}
If $D$ is a standard divisor of type $(n_{1},\ldots,n_{g})$ with $n_i \ge 2$ for all $i$,
then $H^q(E^g,\cO(D))=0$ for all $q>0$.
\end{corollary}
\begin{proof}
Let $X$ be a smooth projective complex algebraic variety, $\omega_X$  its canonical bundle, and $\cL$ an ample invertible $\cO_X$-module.
The Kodaira Vanishing Theorem \cite[Rmk.~III.7.15]{Hart} says that $H^q(X,\cL \otimes \omega_X)=0$ for all $q>0$. 
But $\omega_X \cong \cO_X$ when $X$ is an abelian variety so the result follows from \Cref{pr.ample}.
\end{proof}

\begin{corollary}\label{cor.dp}
If $D$ is a standard divisor of type $(n_{1},\ldots,n_{g})$ with $n_i \ge 2$ for all $i$, then
$$
\dim_\CC H^0(E^g,\cO(D))=\frac{D^g}{g!}
$$
where $D^g$ denotes the $g$-fold self-intersection number.
\end{corollary}
\begin{proof}
  The Riemann-Roch theorem in \cite[\S16]{Mum08} says that  $D^g/g!$ equals the Euler
  characteristic
  \begin{equation*}
    \chi(\cO_{E^{g}}(D)) \;=\;  \sum_{i=0}^g (-1)^i \dim_\CC H^i(E^g,\cO(D)).
  \end{equation*}
But the higher cohomology groups of $\cO(D)$ vanish by 
  \cref{cor.xi-h0} so  $\chi(\cO(D))=\dim H^0(\cO(D))$.    
\end{proof}

\begin{proposition}
\label{pr.sect}
Let $z_1,\ldots,z_g \in E$ and let $X_i=E^{i-1} \times \{z_i\} \times E^{g-i}$ for $i=1,\ldots,g$. 
If $a_1,\ldots,a_g$ and $b_1,\ldots,b_{g-1}$ are arbitrary integers and $D$ is the divisor
$$
a_1X_1\,+\,  \cdots \,+\, a_gX_g \,+\,  b_1\D_{12} \,+\, \cdots \,+\,b_{g-1} \D_{g-1,g}
$$
on $E^g$, then 
$$
\frac{D^g}{g!} \;=\; \det\sA 
$$
where $\sA=\sA(a_1,\ldots,a_g;b_1,\ldots,b_{g-1})$ is the $g \times g$ matrix 
\begin{equation}
\label{defn.A.matrix}
\begin{pmatrix}
		a_{1} +b_1	& b_1				& 				& 			&		\\
		b_1			& a_{2}+b_1+b_2		& b_2			& 			&		\\
					& b_2				&a_{3}+b_2+b_3	& b_3		&		\\
					&  					& \ddots			& 			& 	& 		\\
					&  					& 				&\ddots 		& 	& 		\\
					&					& 				& b_{g-2}			& a_{g-1}+b_{g-2}+b_{g-1} & b_{g-1}  \\
					&					&				&			& b_{g-1}		& a_{g} + b_{g-1}
	\end{pmatrix}
\end{equation} 
\end{proposition}
\begin{proof}
We compute $D^{g}/g!$ in the Chow ring $\CH(E^g)$ but we identify algebraically equivalent divisors in this proof.
We will use the fact that $X_i^2=\Delta_{i,i+1}^2=0$ and 
$X_i\D_{i,i+1}= X_iX_{i+1}=X_{i+1}\D_{i,i+1}$.

We argue by induction on $g$.  The result is certainly true when $g=1$.

If $g=2$, then $D=a_1X_1\,+\, a_2X_2 \,+\,b_1 \Delta$ so
  \begin{equation*}
    \frac{D^2}{2!} \;= \; (a_1a_2+a_1b_1+a_2b_1)X_1X_2  \;=\;  
    \det 
    \begin{pmatrix}
    a_{1} +b_1	& b_1  \\
    b_1		& a_{2}+b_1
    	\end{pmatrix} X_1X_2.
  \end{equation*}
Thus the proposition is true when $g=2$.

From now on we assume $g \ge 3$ and that the result is true for smaller $g$'s.

 We will make use of the following elements in  $\CH(E^g)$:  
\begin{align*}
X &  \;=\;  a_1X_1\,-\,  b_1 X_2 \,+\,  b_1\D_{12},
\\
Y & \;= \;  D-X
\\
& \;=\; (a_2+b_1)X_2 \, + \, a_3X_3 \, +  \cdots \,+\, a_gX_g \,+\,  b_2\D_{23} \,+\, \cdots \,+\,b_{g-1} \D_{g-1,g} \quad \text{and}
\\
Z &  \;=\;  (a_3+b_2)X_3\,+\,  a_4X_4 \, +  \cdots \,+\, a_gX_g \,+\,  b_3\D_{34} \,+\, \cdots \,+\,b_{g-1} \D_{g-1,g}.
\end{align*}

Since $D=X+Y$,
$$
\frac{D^{g}}{g!} \;=\; \frac{1}{g!} \sum_{j=0}^g \binom{g}{j}X^jY^{g-j} \;=\; \frac{Y^g}{g!} \,+\, X \frac{Y^{g-1}}{(g-1)!} \,+\, \frac{1}{2} X^2  \frac{Y^{g-2}}{(g-2)!}\,+\, \cdots .
$$
The terms involving $X^3, X^4,\ldots,$ are zero   because $X=X'\times E^{g-2}$ for a divisor $X'$ on $E \times E$ and $(X')^3=(X')^4=\cdots=0$. 
For a similar reason, $Y^g=0$. Thus  
\begin{equation}
\label{eq.D^g.over.g!}
\frac{D^{g}}{g!}  \;=\; \ X \frac{Y^{g-1}}{(g-1)!} \,+\, \frac{1}{2} X^2  \frac{Y^{g-2}}{(g-2)!}.
\end{equation}
However, $X^2=-2b_1^2 X_1X_2$, and $X_1X_2Y=X_1X_2Z$ because
$$
X_1X_2(Y-Z) \;=\; X_1X_2((a_2+b_1)X_2 - b_2 X_3 + b_2\D_{23}) \;=\; 0,
$$
so 
$$
\frac{D^{g}}{g!}  \;=\; X \frac{Y^{g-1}}{(g-1)!} \,-\, b_1^2 X_1X_2  \frac{Z^{g-2}}{(g-2)!}.
$$
Applying the induction hypothesis to $Y$, viewed as a divisor on $E^{g-1}$,   and to $Z$, viewed as a divisor on $E^{g-2}$,   
\begin{align*}
 \frac{Y^{g-1}}{(g-1)!}  & \;  =\; \det\sA(a_2+b_1,a_3,\ldots,a_g;b_2,\ldots, b_{g-1}) \qquad \text{and}
\\
 \frac{Z^{g-2}}{(g-2)!}  & \;  =\; \det\sA(a_3+b_2,a_4,\ldots,a_g;b_3,\ldots, b_{g-1}).
\end{align*}

Using the cofactor expansion along the first row of $\sA(a_1,\ldots,a_g;b_1,\ldots,b_{g-1})$ then the cofactor expansion down the 
first column of the appropriate minor, one sees that
\begin{align*}
\det\sA(a_1,\ldots,a_g;b_1,\ldots, b_{g-1})  &\; = \;  (a_1+b_1)  \det\sA(a_2+b_1,a_3,\ldots,a_g;b_2,\ldots, b_{g-1})
\\
& \phantom{xxxxxx}\,-\, b_1^2\det\sA(a_3+b_2,a_4,\ldots,a_g;b_3,\ldots, b_{g-1}).
\end{align*}
The result follows.
\end{proof}

\begin{corollary}
\label{cor.sect}
If $D$ is a standard divisor of type $(n_{1},\ldots,n_{g})$ with $n_i \ge 2$ for all $i$, then
\begin{equation*}
	\dim_\CC H^0(E^g,\cO(D))=d(n_{1},\ldots,n_{g}).
\end{equation*}
In particular, $\dim_\CC H^0(E^g,\cL_{n/k})=n$.
\end{corollary}
\begin{proof}
By \cref{cor.dp}, 
$\dim H^0(E^g,\cO(D))=D^g/g!$. To compute the $g$-fold intersection number $D^{g}$ for the standard divisor $D=D_{\fd_{i},z_{j}}$, we can assume that $z_{j}=0$ by applying the translation by a point in $E^{g}$, and that each $\fd_{i}$ is a sum of copies of a single degree-$1$ divisor on $E$ since $\fd_{i}$ is linearly equivalent to such a sum. Thus,
by \cref{pr.sect}, 
\begin{equation} 
\label{eq:D^g/g!}
\frac{D^g}{g!} \;=\;  
\det \begin{pmatrix}
		n_{1}	& 1	& 			& 			&		\\
		1		& n_{2}	& 1		& 			&		\\
				& 1	& \ddots	& \ddots	& 		\\
				&		& \ddots	& n_{g-1}	& 1	\\
				&		&			& 1		& n_{g}
	\end{pmatrix}.
\end{equation} 
If we write $f(n_1,\ldots,n_g)$ for the determinant in the right-hand side, one sees that
$$
f(n_1,\ldots,n_g) \;=\; n_1f(n_2,\ldots,n_g) \,-\, f(n_3,\ldots,n_g).
$$
Viewing $f(x_1,\ldots,x_g)$ as a polynomial in the $x_i$'s, an induction argument shows that every monomial appearing in it is a 
product of an even number of $x_i$'s when $g$ is even and a product of an odd number of $x_i$'s when $g$ is odd. It follows that 
$f(-x_1,\ldots,-x_g)=(-1)^g f(x_1,\ldots,x_g)$.  Hence
$$
\det \begin{pmatrix}
		n_{1}	& 1	& 			& 			&		\\
		1		& n_{2}	& 1		& 			&		\\
				& 1	& \ddots	& \ddots	& 		\\
				&		& \ddots	& n_{g-1}	& 1	\\
				&		&			& 1		& n_{g}
	\end{pmatrix}
	\;=\; 
	\det \begin{pmatrix}
		n_{1}	& -1	& 			& 			&		\\
		-1		& n_{2}	& -1		& 			&		\\
				& -1	& \ddots	& \ddots	& 		\\
				&		& \ddots	& n_{g-1}	& -1	\\
				&		&			& -1		& n_{g}
	\end{pmatrix}
$$	
which is $d(n_{1},\ldots,n_{g})$.

If $D=D_{n/k}$, then $d(n_{1},\ldots,n_{g})=n$ by \cref{prop.nkl}\cref{item.prop.det.D}.
\end{proof}

%%%%%%%%%%%%%%%%%%%%%%%%%%%%%%%%%%%%%%%%%%%%%%%%%%%%%%%%%%%%%%%%%%%%%%%%%%%%%%%%
\subsection{Divisors on $E^g$ associated to labeled graphs}
\label{sss.grph}
%%%%%%%%%%%%%%%%%%%%%%%%%%%%%%%%%%%%%%%%%%%%%%%%%%%%%%%%%%%%%%%%%%%%%%%%%%%%%%%%

In this section we use labeled graphs as bookkeeping devices for certain divisors on  $E^g$. As a consequence we provide another proof for \cref{le.kl} when $D=D_{n/k}$.

\begin{definition}\label{def.wt-grph}
  A {\sf labeled graph} is an unoriented graph with an integer label on each of its edges (some of which can be loops; there is at most one edge for each pair of vertices). 
  The label $0$ means that the respective edge is not present.

  We use $G$ to denote both the labeled and unlabeled graph when no confusion arises. We write $G^0$\index{G^0@$G^{0}$} and $G^1$\index{G^1@$G^{1}$} for the sets of vertices and edges respectively. For $e\in G^1$ we write $\ell_e\in \ZZ$\index{l_e@$\ell_e$} for its label. We introduce the following objects.

  \begin{itemize}
  \item The {\sf degree matrix} of $G$ is the diagonal $|G^0|\times|G^0|$ matrix  $\mathrm{Deg}_G$\index{Deg_G@$\mathrm{Deg}_G$} whose $i^{th}$ diagonal entry is 
  $\sum_e \ell_e$, where the sum ranges over the edges incident to $i$ that are {\it not} loops.
  \item   The {\sf adjacency matrix} of $G$ is the $|G^0|\times |G^0|$ matrix $\sfA_G=(a_{ij})$\index{A_G@$\sfA_G$} whose entry $a_{ij}$
 is the label of the edge $(i,j)\in G^1$. 
  \item  $\sfM=\sfM_G:=\mathrm{Deg}_G-\sfA_G$\index{M, M_G@$\sfM, \sfM_G$}.
  \end{itemize}
\end{definition}

Let $G$ be a labeled graph and let $g=|G^0|$. We write $D_G$\index{D_G@$D_{G}$} for the divisor 
\begin{equation*}
  D_G \; :=\; \sum_{e\in G^1}\ell_e D_e 
\end{equation*}
on $E^g$ where
\begin{equation*}
  e=(i,j)  \; \Longrightarrow  \; D_e=
  \begin{cases}
    \Delta_{ij} &\text{ if }i\ne j,\\
    E^{i-1}\times\{pt\}\times E^{g-i} &\text{ otherwise,}
  \end{cases}
\end{equation*}
where the point in $E$ denoted by $pt$ is fixed but arbitrary for each $i$.

The Chern class of $\cL_G:=\cO_{E^{g}}(D_{G})$\index{L_G@$\cL_{G}$} corresponds to a Hermitian form $\sfH_G$\index{H_G@$\sfH_{G}$} on $\CC^g$ as described in \cite[Ch.~2]{bl}.
 We will now identify the associated skew-symmetric form $\sfE=\sfE_G:=\Im \sfH_{G}$\index{E, E_G@$\sfE$, $\sfE_{G}$} with respect to the basis
\begin{equation}\label{eq:4}
  \eta \sfe_1,\  \ldots,\ \eta \sfe_g,\ \sfe_1,\ \ldots,\ \sfe_g
\end{equation}
of $\RR^{2g}\cong \CC^g$, where   $\sfe_i$ is the vector whose $i^{th}$ entry is $1$ and other entries are zero.

In the notation of \cite[Prop.~1.2]{Polishchuk-book}, the factor of automorphy (see \cite[\S1.3]{Kempf-book} and/or \cite[Appendix~B]{bl}) 
for $\cL_G$ is given by
\begin{equation}\label{eq:5}
  \begin{aligned}
    e_{\sfe_i}(z_1,\cdots,z_g) &= 1,\\
    e_{\eta \sfe_i}(z_1,\cdots,z_g)&=e((-\sfM\sfz)_i + C_i)
  \end{aligned}
\end{equation}
for certain constants $C_i$. (Polishchuk calls a factor of automorphy a ``multiplicator''.)

Writing $e_\sfu(\sfz)=e(f_\sfu(\sfz))$, the skew-symmetric bilinear form $\sfE$ can be computed as explained in \cite[\S2, Prop., p.~18]{Mum08}:
  \begin{equation}\label{eq:3}
    \sfE(\sfu_1,\sfu_2)  \; = \; f_{\sfu_2}(\sfz+\sfu_1)+f_{\sfu_1}(\sfz) - f_{\sfu_1}(\sfz+\sfu_2) - f_{\sfu_2}(\sfz).
  \end{equation}
  It follows that the $2g\times 2g$ matrix $\sfE$, with respect to the basis in \Cref{eq:4},  for which 
  $\sfE(\sfu_1,\sfu_2) =\sfu_1^\sT\sfE\sfu_2$ is of the form
\begin{equation*}
  \begin{pmatrix}
    *& \sfM_G\\ -\sfM_G& 0
  \end{pmatrix}.
\end{equation*}
The upper left hand block $*$ is also zero: according to \Cref{eq:5} we have
\begin{equation*}
f_{\eta \sfe_i}(\sfz+\eta \sfe_j) - f_{\eta \sfe_i}(\sfz) \; = \;  -\eta (\sfM \sfe_j)_i = -\eta \sfM_{ij}.
\end{equation*}
Since $\sfM$ is symmetric, interchanging $i$ and $j$ has no effect on this expression and hence \Cref{eq:3} vanishes for $\sfu_1=\sfe_i$ and $\sfu_2=\sfe_j$. In conclusion, the matrix of $\sfE=\Im \sfH$ in the basis \Cref{eq:4} is
\begin{equation}\label{eq:6}
  \begin{pmatrix}
    0& \sfM_G\\ -\sfM_G& 0
  \end{pmatrix}.
\end{equation}

\begin{proposition}\label{pr.dg}
  Let $G$ be a labeled graph. The self-intersection number of the divisor $D=D_G$ defined above can be computed as follows:
  \begin{equation*}
    \frac{D^g}{g!} \;=\; \det\big(\mathrm{Deg}_G-\sfA_G\big). 
  \end{equation*}
\end{proposition}
\begin{proof}
  Recall that $\sfM_G=\mathrm{Deg}_G-\sfA_G$. As seen above, the matrix of the skew bilinear form $\Im \sfH$ with respect to this basis is
\begin{equation*}
  \begin{pmatrix}
    0& \sfM_G\\ -\sfM_G& 0
  \end{pmatrix}.
\end{equation*}
It follows that
\begin{equation*}
  \det \Im \sfH = (\det \sfM_G)^2, 
\end{equation*}
which is equal to the square of the self-intersection number $\frac{D^g}{g!}$ by \cite[Thms.~3.6.1 and 3.6.3]{bl}. We thus have the desired conclusion up to sign:
\begin{equation*}
  \frac{D^g}{g!}=\pm \det \sfM_G. 
\end{equation*}
To determine the sign, notice first that the self-intersection number is a multivariate polynomial in the labels of $G$, as is $\det \sfM_G$. The fact that the sign is $+$ now follows by considering the computationally immediate case when $G^1$ consists of only loops.
\end{proof}

The next result extends  \Cref{le.kl} to  divisors of the form $D_G$.
We identify $\CC^{g}$ with $\RR^{2g}$ using the basis \cref{eq:4}. Thus $\L=\ZZ^{2g}$.

\begin{proposition}\label{pr.det-k}
  Let $D_G$ be the divisor attached to a labeled graph $G$ as above. Then  $K(D_G)$ is the kernel of the matrix
  $\sfM_G =\mathrm{Deg}_G-\sfA_G$ regarded as an operator on $E^{g}$ by multiplication.
\end{proposition}
\begin{proof}
 Let $D=D_G$ and $\sfM=\sfM_G$.
 
As before, we denote by $\sfE=\sfE_{G}$ the skew-symmetric form $\Im \sfH$ associated to the Hermitian form $\sfH$ of the invertible sheaf $\cL=\cO(D)$. According to \cite[discussion preceding Lem.~2.4.7]{bl}, $K(D)$ is $\L(\cL)/\L$, where
  \begin{equation*}
    \L(\cL):=\{v\in V\ |\ \sfE(v,\L)\subseteq \ZZ\}
  \end{equation*}
and $V=\CC^g\cong \RR^{2g}$ is the universal cover of $E^g$. Having identified $E$ with the matrix \Cref{eq:6}, this is the same as
  \begin{equation*}
    \{v\in V\ |\ \sfE v\in \L\}. 
  \end{equation*}
  Now consider the real endomorphism $\sfJ$ of $V$ acting as $-\eta$ on $\bigoplus_{i=1}^g\RR \sfe_i$ and as $\frac 1{\eta}$ on $\bigoplus_{i=1}^g\RR \eta \sfe_i$. Note that $\sfJ$ acts isomorphically on $\L$, and hence induces an automorphism of the torus $E^g$ (not holomorphic, in general).

  The operator on $E^g$ induced by $\sfM$ is $\sfJ\sfE$, where $\sfE$ is similarly regarded as an operator on $E^g$ via the matrix \Cref{eq:6}.
  \begin{equation*}
    K(D) = \L(\cL)/\L
  \end{equation*}
  is the kernel of the torus endomorphism $\sfE:E^g\to E^g$. This in turn coincides with the kernel of the endomorphism
$\sfM = \sfJ\sfE:E^g\to E^g$,   hence the conclusion.
\end{proof}

In particular, we obtain an

\begin{proof}[Alternative proof for \Cref{le.kl} when $D=D_{n/k}$]
  Simply apply \Cref{pr.det-k} to the labeled graph $G$ whose underlying graph has edges $(i,i)$ and $(i,i+1)$ with labels
  \begin{equation*}
    \ell_{(i,i+1)} = -1,\quad \ell_{(i,i)} = n_{i}+2-\delta_{i,1}-\delta_{i,g}.
    \qedhere
\end{equation*}
\end{proof}

%%%%%%%%%%%%%%%%%%%%%%%%%%%%%%%%%%%%%%%%%%%%%%%%%%%%%%%%%%%%%%%%%%%%%%%%%%%%%%%%
\section{The characteristic variety as a quotient of $E^g$}
\label{se.char-var-2}
%%%%%%%%%%%%%%%%%%%%%%%%%%%%%%%%%%%%%%%%%%%%%%%%%%%%%%%%%%%%%%%%%%%%%%%%%%%%%%%%

Because $\cL_{n/k}$ is generated by its global sections (\cref{prop.gen.gl.sec}), there is an associated morphism\index{Phi_n/k@$\Phi_{n/k}$}
\begin{equation}
\label{defn.Phi_D}
\Phi_{n/k}\;:=\;\Phi_{|\cL_{n/k}|} \colon E^g  \, \longrightarrow \, \PP^{n-1}\;=\; \PP\big(H^0(E^g,{\cL_{n/k}})^*\big).
\end{equation}
We write $X_{n/k}$\index{X_n/k@$X_{n/k}$} for the image of $\Phi_{n/k}$ and call it the {\sf characteristic variety} for $Q_{n,k}(E,\tau)$.
The main result in this section is that 
$$
X_{n/k} \; \cong \; E^g/\Sigma_{n/k}.
$$
While proving this we make several observations that we will use in \cite{CKS3}: 
\begin{itemize}
  \item[-] 
  there is a unique $\Sigma_{n/k}$-equivariant structure on $\cL_{n/k}$ such that the induced action of $\Sigma_{n/k}$ on
$H^0(E^g,{\cL_{n/k}})$ is trivial; 
\item[-]
we will always give $\cL_{n/k}$ that equivariant structure;
\item[-]
 writing $\phi:E^g \to E^g/\Sigma_{n/k}$\index{phi@$\phi$} for the quotient morphism, 
 $(\phi_* \cL_{n/k})^{\Sigma_{n/k}}$ is an invertible sheaf on $E^g/\Sigma_{n/k}$;
  \item[-]
  $D_{n/k}$ is stable under the action of $\Sigma_{n/k}$; 
  \item[-] 
$
\cL_{n/k} \; \cong \; \phi^*\big((\phi_* \cL_{n/k})^{\Sigma_{n/k}} \big);
$
\item[-]
if  $\iota: E^g/\Sigma_{n/k} \to \PP^{n-1}$\index{iota@$\iota$} is the unique morphism such that $\Phi_{n/k}=\iota \circ \phi$, then 
$(\phi_* \cL_{n/k})^{\Sigma_{n/k}} \cong \iota^*\cO_{\PP^{n-1}}(1)$.
\end{itemize}

%%%%%%%%%%%%%%%%%%%%%%%%%%%%%%%%%%%%%%%%%%%%%%%%%%%%%%%%%%%%%%%%%%%%%%%%%%%%%%%%
\subsection{Effective divisors linearly equivalent to $D_{n/k}$}
\label{sect.eff.divs}
%%%%%%%%%%%%%%%%%%%%%%%%%%%%%%%%%%%%%%%%%%%%%%%%%%%%%%%%%%%%%%%%%%%%%%%%%%%%%%%%

We examine when two standard divisors are linearly equivalent to each other. This provides a large supply of effective divisors linearly equivalent to $D_{n/k}$.

\begin{proposition}\label{pr.ssw}
Let $\fd_1,\ldots,\fd_g,\fd_1',\ldots,\fd_g'$ be effective divisors on $E$ and $z_1,\ldots,z_{g-1},z_1',\ldots,z_{g-1}' \in E$. 
  \begin{enumerate}
  \item\label{item.prop.ssw.dz}
 The divisors $D_{\fd_i,z_j}$  and $D_{\fd'_i,z'_j}$ are linearly equivalent if and only if
  \begin{equation}
    \fd_i+(-z_i)+(z_{i-1}) \; \sim  \; \fd'_i+(-z'_i)+(z'_{i-1})
  \end{equation}
  for all $1\le i\le g$, with the convention that $z_0=z'_0=z_{g}=z'_{g}=0$. 
  \item\label{item.prop.ssw.d}
If $\deg\fd_i'=\deg\fd_i$ and  $\fd_i'\sim \fd_i$ for all $i$, then 
  $D_{\fd'_i,z_j}$ is linearly equivalent to $D_{\fd_i,z_j}$.
  \item\label{item.prop.ssw.dpr}
 The divisor $D_{\fd_i,z_j}$  is linearly equivalent to $D_{n/k}$ if and only if  
 \begin{enumerate}
  \item 
  $\deg\fd_i=(n_i-2+\delta_{i,1}+\delta_{i,g})(0)$ for all $i$ and
  \item 
  $\sfsum\fd_i=z_i-z_{i-1}$ for all $i$.
\end{enumerate}
\end{enumerate}
\end{proposition}
\begin{proof}
\cref{item.prop.ssw.dz}
 This follows from the Seesaw Theorem which says that two divisors on $E^g$ are linearly equivalent 
 if and only if their restrictions to almost all slices $\{(x_1, \ldots,x_{i-1})\}\times  E\times \{(x_{i+1}, \ldots,x_g)\}$   are linearly equivalent.

\cref{item.prop.ssw.d}
This follows immediately from \cref{item.prop.ssw.dz}.

\cref{item.prop.ssw.dpr}
For all $i=1,\ldots,g$, let $z_i'=0$ and $\fd_i'=D_i$ as in  \cref{defn.Di}.
Then $D_{\fd'_i,z'_j}=D_{n/k}$. Thus,   
$D_{\fd_i,z_j}$  is linearly equivalent to $D_{n/k}$ if and only if   $\fd_i+(-z_i)+(z_{i-1}) \sim D_i+(0)+(0)$ for all $i$; i.e., 
if and only if $\deg\fd_i=\deg D_i$ and $\sfsum\fd_i-z_i+z_{i-1}=0$ for all $i$.
\end{proof}

\begin{proposition}\label{prop.gen.gl.sec}
	Suppose $n_{i}\geq 2$ for all $i$. Every standard divisor $D_{\fd_{i},z_{j}}$ of type $(n_{1},\ldots,n_{g})$ is base-point free or, equivalently, $\cO_{E^{g}}(D_{\fd_{i},z_{j}})$ is generated by its global sections.
	
	In particular, $D_{n/k}$ is base-point free and $\cL_{n/k}$ is generated by its global sections.
\end{proposition}
\begin{proof}
	Fix a point $\sfx=(x_{1},\ldots,x_{g})\in E^{g}$. We will find a standard divisor $D_{\fd'_{i},z'_{j}}$ linearly equivalent to $D_{\fd_{i},z_{j}}$ that does not contain $\sfx$. Choose $\sfw=(w_{1},\ldots,w_{g})\in E^{g}$ so that
	\begin{itemize}
		\item\label{item.gen.gl.sec.zero} $w_{i}=0$ for all $i$ for which $\deg\fd_{i}=0$,
		\item\label{item.gen.gl.sec.one} $w_{i}\neq x_{i}$ for all $i$ for which $\deg\fd_{i}=1$,
		\item\label{item.gen.gl.sec.ineq} $z_{j}+\sum_{i=1}^{j}(w_{i}-\sfsum\fd_{i})\neq x_{j+1}-x_{j}$ for all $1\leq j\leq g-1$, and
		\item\label{item.gen.gl.sec.eq} $\sum_{i=1}^{g}(w_{i}-\sfsum\fd_{i})=0$.
	\end{itemize}
	Such a $\sfw$ does exist: For example, let $w_{i}=0$ whenever $\deg\fd_{i}=0$, and choose $w_{i}\neq x_{i}$ for other $i\neq 1,g$. Once we further fix $w_{1}$, the last condition determines $w_{g}$. Since $\deg\fd_{1},\deg\fd_{g}\geq 1$, the remaining requirements are a {\it finite} 
number of inequalities on $w_{1}$ (or $w_{g}$), so there is a solution.
	
	Due to the properties of $\sfw$, we can choose a divisor $\fd'_{i}$ on $E$ for each $i$ so that
	\begin{itemize}
		\item $\deg\fd'_{i}=\deg\fd_{i}$,
		\item $\sfsum\fd'_{i}=w_{i}$, and
		\item $\fd'_{i}$ does not contain $x_{i}$.
	\end{itemize}
	Let $z'_{j}:=z_{j}+\sum_{i=1}^{j}(w_{i}-\sfsum\fd_{i})$ for each $j$. Then the standard divisor $D_{\fd'_{i},z'_{j}}$ does not contain $\sfx$, and it is linearly equivalent to $D_{\fd_{i},z_{j}}$ by \cref{pr.ssw}\cref{item.prop.ssw.dz}.
\end{proof}

\begin{lemma}\label{lem.sep.pts.ni.geq.three}
	Let $D=D_{\fd_{i},z_{j}}$ be a standard divisor of type $(n_{1},\ldots,n_{g})$ where $n_{i}\geq 2$ for all $i$. Let $\sfx=(x_{1},\ldots,x_{g})$ and $\sfy=(y_{1},\ldots,y_{g})$ be points in $E^{g}$. If $x_{t}\neq y_{t}$ for some $t$ satisfying $n_{t}\geq 3$, then there is a standard divisor $D'=D_{\fd'_{i},z'_{j}}$ linearly equivalent to $D$ such that $\sfx\in D'$ and $\sfy\notin D'$.
\end{lemma}
\begin{proof}
	Choose $\sfw=(w_{1},\ldots,w_{g})\in E^{g}$ so that
	\begin{itemize}
		\item $w_{i}=0$ for all $i$ for which $\deg\fd_{i}=0$,
		\item $w_{t}=x_{t}$ if $\deg\fd_{t}=1$,
		\item $w_{i}\neq y_{i}$ for all $i$ for which $\deg\fd_{i}=1$,
		\item $w_{t}\neq x_{t}+y_{t}$ if $\deg\fd_{t}=2$,
		\item $z_{j}+\sum_{i=1}^{j}(w_{i}-\sfsum\fd_{i})\neq x_{j+1}-x_{j}$ for all $1\leq j\leq g-1$,
		\item $z_{j}+\sum_{i=1}^{j}(w_{i}-\sfsum\fd_{i})\neq y_{j+1}-y_{j}$ for all $1\leq j\leq g-1$, and
		\item $\sum_{i=1}^{g}(w_{i}-\sfsum\fd_{i})=0$.
	\end{itemize}
	If $t=1$ or $t=g$, then $\deg\fd_{t}\geq 2$ and the second condition does not arise. So such $\sfw$ exists by the same reason as in the proof of \cref{prop.gen.gl.sec}. Similarly, choose a divisor $\fd'_{i}$ on $E$ for each $i$ so that
	\begin{itemize}
		\item $\deg\fd'_{i}=\deg\fd_{i}$,
		\item $\sfsum\fd'_{i}=w_{i}$,
		\item $\fd'_{i}$ does not contain $y_{i}$, and
		\item $\fd'_{t}$ contains $x_{t}$,
	\end{itemize}
	and let $z'_{j}:=z_{j}+\sum_{i=1}^{j}(w_{i}-\sfsum\fd_{i})$ for each $j$. The standard divisor $D_{\fd'_{i},z'_{j}}$ satisfies the desired property.
\end{proof}

Now we prove a similar result that will be used in the proof of \cref{le.tg-inj}.

\begin{lemma}\label{lem.dnk.pt.in.only.div}
Let $\sfx=(x_{1},\ldots,x_{g})$ be a point and 
	let $D_{\fd_{i},z_{j}}$ be a standard divisor of type $(n_{1},\ldots,n_{g})$ where $n_{i}\geq 2$ for all $i$.  For each $t$ satisfying $n_{t}\geq 3$, there is a standard divisor $D_{\fd'_{i},z'_{j}}$ linearly equivalent to $D_{\fd_{i},z_{j}}$ such that $D_{\fd'_{i},z'_{j}}$ contains $E^{t-1}\times(x_{t})\times E^{g-t}$ and no other component of $D_{\fd'_{i},z'_{j}}$ contains $\sfx$.
\end{lemma}
\begin{proof}
	Choose $\sfw=(w_{1},\ldots,w_{g})\in E^{g}$ so that
	\begin{itemize}
		\item $w_{i}=0$ for all $i$ for which $\deg\fd_{i}=0$,
		\item $w_{i}\neq x_{i}$ for all $i\neq t$ for which $\deg\fd_{i}=1$,
		\item $w_{t}=x_{t}$ if $\deg\fd_{t}=1$,
		\item $w_{t}\neq 2x_{t}$ if $\deg\fd_{t}=2$,
		\item $z_{j}+\sum_{i=1}^{j}(w_{i}-\sfsum\fd_{i})\neq x_{j+1}-x_{j}$ for all $1\leq j\leq g-1$, and
		\item $\sum_{i=1}^{g}(w_{i}-\sfsum\fd_{i})=0$.
	\end{itemize}
	Choose a divisor $\fd'_{i}$ on $E$ for each $i$ so that
	\begin{itemize}
		\item $\deg\fd'_{i}=\deg\fd_{i}$,
		\item $\sfsum\fd'_{i}=w_{i}$,
		\item $\fd'_{i}$ does not contain $x_{i}$ if $i\neq t$, and
		\item $\fd'_{t}$ contains $x_{t}$ with multiplicity one,
	\end{itemize}
	and let $z'_{j}:=z_{j}+\sum_{i=1}^{j}(w_{i}-\sfsum\fd_{i})$ for each $j$.
\end{proof}

%%%%%%%%%%%%%%%%%%%%%%%%%%%%%%%%%%%%%%%%%%%%%%%%%%%%%%%%%%%%%%%%%%%%%%%%%%%%%%%%
\subsection{Generalities on equivariant sheaves}
\label{subsec.equiv.shvs}
%%%%%%%%%%%%%%%%%%%%%%%%%%%%%%%%%%%%%%%%%%%%%%%%%%%%%%%%%%%%%%%%%%%%%%%%%%%%%%%%
Let $X$ be a quasi-projective algebraic $\Bbbk$-variety\footnote{We follow the convention in \cite[\S II.4, Defn., p.~105]{Hart} that {\it algebraic varieties} are integral separated schemes of finite type over an algebraically closed field $\Bbbk$.}  
endowed with an action by a finite group $\G$.  
The quotient morphism is denoted by
$$
\phi: X \, \longrightarrow \, X/\G.
$$
This subsection records some facts about $\G$-equivariant invertible $\cO_X$-modules 
(i.e., ``linearized sheaves'' in the terminology of  \cite[\S 1.3]{mumf-git}). 
Most are probably well-known but we couldn't find an ideal reference.

The structure sheaf $\cO_X$ has a family of $\Gamma$-equivariant structures parametrized by the characters of $\Gamma$, i.e.,
the group homomorphisms $\chi:\Gamma\to \Bbbk^*$. 
This correspondence is such that the equivariant structure attached to $\chi$ results in an action of $\Gamma$ on 
$H^0(X,\cO_X)$ via $\chi$ (i.e., as a representation of $\Gamma$, $H^0(X,\cO_X)$ is $\chi$-isotypic). 
We write $\cO_X^{\chi}$ for $\cO_X$ with the equivariant structure corresponding to $\chi$.
If  $\cL$ is a $\Gamma$-equivariant $\cO_X$-module, the {\sf $\chi$-twist}  of the equivariant structure on $\cL$ is the tensor product\index{L^chi@$\cL^{\chi}$}
$$
\cL^{\chi} \; :=\;  \cL\otimes \cO_X^{\chi}
$$
of equivariant sheaves.

The next result is used after we prove that $\cL_{n/k}$ 
has a $\Sigma_{n/k}$-equivariant structure (\Cref{prop.lnk.equiv.str}). 

\begin{proposition}\label{pr.gen-equivariant}
Let $X$ be a quasi-projective algebraic variety acted upon by a finite group $\Gamma$ and let $\phi:X\to X/\Gamma$
be the quotient morphism. If $\cL$ is 
 \begin{itemize}
  \item
  a globally generated $\Gamma$-equivariant invertible $\cO_X$-module such that
  \item 
  the action of $\Gamma$ on $H^0(X,\cL)$ resulting from the equivariant structure is trivial,
  \end{itemize}
then $(\phi_*\cL)^{\Gamma}$ is an invertible  $\cO_{X/\Gamma}$-module and the canonical map  $\phi^*\big((\phi_*\cL)^{\Gamma}\big) \to \cL$
is an isomorphism. 
\end{proposition}
\begin{proof}
Every finite subset of $X$ is contained in an open affine subscheme of $X$ (see, e.g., \cite[Prop.~3.3.36]{liu02}). 
In particular this is true for $\Gamma$-orbits. 
Hence \cite[\S7, Thm.]{Mum08} applies to every $\Gamma$-invariant open subscheme 
$U\subseteq X$ to prove the claim in the case $\cL=\cO_X$.

  Now let $x\in X$. Because $\cL$ is generated by its global sections there is an $s\in H^0(X,\cL)$ that does not vanish at
   $x$. Since the action of $\Gamma$ on $H^0(X,\cL)$ is trivial, $s$ vanishes at none of the points in the orbit $\Gamma x$. 
 That orbit is therefore contained in some (dense, by irreducibility) affine open subscheme $U\subseteq X$ where $s$ is non-vanishing. 
Replacing $U$ by $\bigcap_{\gamma\in \Gamma}\, \gamma U$,
  we can also assume $U$ is $\Gamma$-invariant (like $U$, the intersection is affine because the scheme is quasi-projective and hence separated).

  The $\Gamma$-invariant section
  \begin{equation*}
    s|_U\in H^0(U,\cL) 
  \end{equation*}
 implements a $\Gamma$-equivariant isomorphism $\cO_U \to \cL|_U$, so we can apply  \cite[\S7, Thm.]{Mum08} to conclude that
$(\phi_*\cL)^{\Gamma}$ is an invertible sheaf on $U/\Gamma$ and  the canonical morphism
    \begin{equation}\label{eq:9}
  \phi^*(\phi_*\cL)^{\Gamma} \, \longrightarrow \,    \cL
    \end{equation}
    is an isomorphism over $U$. 
  Since $X$ is covered by such open patches $U$, $X/\Gamma$ is covered by the open affines $U/\Gamma$ and
   these canonical isomorphisms glue to give a  global isomorphism $\phi^*(\phi_*\cL)^{\Gamma} \, \to \,    \cL$.
\end{proof}

\begin{lemma}\label{le.act-scale}
  Let $X$ be a projective algebraic variety acted upon by the finite group $\Gamma$ and 
  let $\cL$ be a $\Gamma$-equivariant globally generated invertible $\cO_X$-module. If the map
  \begin{equation*}
    \Phi_{|\cL|}:X\;\longrightarrow\;\PP(H^0(X,\cL)^*)
  \end{equation*}
  factors through $X/\Gamma$, then 
  \begin{enumerate}
  \item\label{item.act.scale.isotypic}
the action of $\Gamma$ on $H^0(X,\cL)$ is $\chi$-isotypic for some  character $\chi$ and
  \item\label{item.act.scale.triv}
the  natural action of $\Gamma$ on $H^0(X,\cL^{\chi^{-1}})$ is trivial.
\end{enumerate} 
\end{lemma}
\begin{proof}
\cref{item.act.scale.isotypic}
Let $\g \in \Gamma$. Let $s$ be a non-zero section of 
$\cL$. By hypothesis, $(s)_0=(\g.s)_0$. 
Hence 
there is an automorphism
of $\cL$ sending $s$ to $\g.s$. Since $X$ is a projective variety, and therefore irreducible according to our conventions, 
$    \Aut(\cL)\cong\Aut(\cO_X)\cong \Bbbk^{\times} $.
Hence $\g.s$ is a scalar multiple of $s$. Thus, every non-zero element in
$H^0(X,\cL)$ is an eigenvector for every $\g\in \Gamma$. 
It follows that $H^0(X,\cL)$ is $\chi$-isotypic for some character $\chi$.

Part \cref{item.act.scale.triv} is obvious.
\end{proof}

\begin{example}\label{ex.z-z}
  The global generation assumption is necessary in  \Cref{pr.gen-equivariant}. 
  Consider, for example, the natural action of $\Gamma:=\{\pm 1\}$ on $E$ and 
   the sheaf $\cL:=\cO_E((0))$.
  Since $(0)$ is stable under the action of $\G$ we can put a $\G$-equivariant structure on $\cL$. The space $H^0(E,\cL)$ consists of 
  only the constant functions $\cO_E \to \cO_E \subseteq \cL$, so the  induced action of $\G$ on $H^0(E,\cL)$ is trivial. However,  $E/\Gamma \cong  \PP^1$ and, with the notation above, $ (\phi_*\cL)^{\Gamma}\cong \cO_{\PP^1}$. 
 It follows that the homomorphism \Cref{eq:9} is not an isomorphism in this situation.
\end{example}

\begin{lemma}\label{le.inv-div}
Let $X$ be a projective algebraic variety acted upon by the finite group $\Gamma$.
If  $D$ is an effective Cartier divisor on $X$ that is stable under the action of $\G$,
then there is a $\Gamma$-equivariant structure on $\cO_X(D)$ such that the resulting action
of $\G$  on $H^0(X,\cO_X(D))$ fixes the sections vanishing along $D$.
\end{lemma}
\begin{proof}
Let $\gamma\in \Gamma$ and  let $s\in H^0(X,\cO_X(D))$ be a section whose zero locus is $D$;
$s$ is then unique up to a non-zero scalar multiple (see the proof of \Cref{le.act-scale}).

Effective Cartier divisors are in bijection with isomorphism classes of pairs $(\cL,s)$ where $\cL$ is an invertible sheaf and $s$ is a non-zero section. The divisor corresponding to $(\cL,s)$ is the divisor of zeros $(s)_{0}$.

Since $\gamma D=D$, the pairs $(\cO_X(D),s)$ and $(\gamma^*\cO_X(D),\gamma^*s)$ are isomorphic. 
But since the automorphism group of a pair $(\cL,s)$ over a projective variety is trivial, there is exactly one isomorphism
\begin{equation*}
  \phi_{\gamma}:\cO_X(D)\to \gamma^*\cO_X(D)
\end{equation*}
that sends $s$ to $\gamma^*s$. It follows that the cocycle condition
\begin{equation*}
 \begin{tikzpicture}[auto,baseline=(current  bounding  box.center)]
  \path[anchor=base] (0,0) node (1) {$\cO(D)$} +(3,-.5) node (2) {$\beta^*\cO(D)$} +(6,0) node (3) {$(\alpha\beta)^*\cO(D)$};
  \draw[->] (1) to[bend right=6] node[pos=.5,auto,swap] {$\scriptstyle \phi_{\beta}$} (2);
  \draw[->] (2) to[bend right=6] node[pos=.5,auto,swap] {$\scriptstyle \beta^*\phi_{\alpha}$} (3);
  \draw[->] (1) to[bend left=6] node[pos=.5,auto] {$\scriptstyle \phi_{\alpha\beta}$} (3);  
 \end{tikzpicture}
\end{equation*}
holds and hence that the $\phi_{\gamma}$, $\gamma\in \Gamma$, constitute an equivariant structure. Furthermore, the resulting action on $H^0(X,\cO_X(D))$ fixes $s$ by construction and hence all of its scalar multiples, as desired.
\end{proof}

\begin{remark}
The choice of $s$ in the above proof does not affect the equivariant structure, as $s$ is unique up to scaling which would then be passed on to $\gamma^*s$. The structure is thus canonical.
\end{remark}

\begin{remark}\label{rem.char.triv}
  If, in the context of \Cref{le.inv-div}, the hypotheses of \Cref{le.act-scale} are also satisfied for $\cL=\cO_X(D)$, then the action of $\Gamma$ on $H^0(X,\cO_X(D))$ is trivial: indeed, the latter result says that the action is scaling by a character, whereas the former identifies a fixed invariant vector $s\in H^0(X,\cO_X(D))$.

Moreover, the $\Gamma$-equivariant structure obtained by \Cref{le.inv-div} is canonical in the following sense. Let
\begin{equation*}
	\begin{tikzcd}
		X\ar[r,"\phi"] & X/\Gamma\ar[r,"\iota"] & \bbP(H^{0}(X,\cL)^{*})
	\end{tikzcd}
\end{equation*}
be the factorization of $\Phi_{|\cL|}\colon X\to\bbP(H^{0}(X,\cL)^{*})$. The invertible sheaf $\cL':=\iota^{*}\cO(1)$ has a trivial $\Gamma$-equivariant structure with respect to the trivial action of $\Gamma$ on $X/\Gamma$, and its pullback $\cL=\phi^{*}\cL'$ along the $\Gamma$-equivariant morphism $X\to X/\Gamma$ has the induced equivariant structure. The canonical map
\begin{equation*}
	H^{0}(X/\Gamma,\cL')\;\to\; H^{0}(X,\cL)
\end{equation*}
is non-zero and $\Gamma$-equivariant. Since the action of $\Gamma$ on $H^{0}(X/\Gamma,\cL')$ is trivial and that on $H^{0}(X,\cL)$ is $\chi$-isotypic by \cref{le.act-scale}, that on $H^{0}(X,\cL)$ should also be trivial. This shows that the induced equivariant structure is the same as the equivariant structure obtained by \Cref{le.inv-div}. Indeed, as in the proof of \cref{le.act-scale}, the structure morphism $\cL\to\gamma^{*}\cL$ is unique up to a non-zero scalar multiple, and the multiple affects the action on $H^{0}(X,\cL)$. Since both equivariant structures induce a trivial action on $H^{0}(X,\cL)$, they are the same.
\end{remark}

%%%%%%%%%%%%%%%%%%%%%%%%%%%%%%%%%%%%%%%%%%%%%%%%%%%%%%%%%%%%%%%%%%%%%%%%%%%%%%%%
\subsection{The $\Sigma_{n/k}$-equivariant structure on $\cL_{n/k}$}
\label{subsec.equiv.lnk}
%%%%%%%%%%%%%%%%%%%%%%%%%%%%%%%%%%%%%%%%%%%%%%%%%%%%%%%%%%%%%%%%%%%%%%%%%%%%%%%%

We now introduce a class of standard divisors that are stable under the action of $\Sigma_{n/k}$. 
We call a standard divisor $D_{\fd_{i},z_{j}}$   \textsf{balanced} if $z_{j}=0$ for all $j$ and it contains $(0)\times E^{g-1}+E^{g-1}\times(0)$.
This is equivalent to the condition that $D_{\fd_{i},z_{j}}$ contains
\begin{equation}\label{eq.min.bal.std.div}
	\begin{cases}
		2(0) & \text{if $g=1$;} \\
		(0)\times E^{g-1}+\Delta_{1,2}+\cdots+\Delta_{g-1,g}+E^{g-1}\times(0) & \text{if $g\geq 2$.}
	\end{cases}
\end{equation}
In particular, $D_{n/k}$ is balanced.

\begin{proposition}\label{prop.lnk.equiv.str}
Let $D$ be a balanced standard divisor of type $(n_{1},\ldots,n_{g})$. Then $D$ is stable under the action of $\Sigma_{n/k}$ and, consequently, $\cO_{E^{g}}(D)$ admits a $\Sigma_{n/k}$-equivariant structure.
  In particular, $D_{n/k}$ is stable under the action of $\Sigma_{n/k}$ and $\cL_{n/k}=\cO_{E^{g}}(D_{n/k})$ admits a $\Sigma_{n/k}$-equivariant structure.
\end{proposition}
\begin{proof}
	The case $g=1$ is obvious so we assume $g\geq 2$. By the assumption, $D$ is the sum of
	\begin{equation}\label{eq.lnk.equiv.str.first}
		(0)\times E^{g-1}+\Delta_{1,2}+\cdots+\Delta_{g-1,g}+E^{g-1}\times (0)
	\end{equation}
	and
	\begin{equation}\label{eq.lnk.equiv.str.second}
		\sum_{i=1}^{g}E^{i-1}\times\fd_{i}\times E^{g-i}
	\end{equation}
	where $\deg\fd_{i}=n_{i}-2$.
	For each $2\leq i\leq g-1$, $s_{i}^{*}\Delta_{i,i+1}=\Delta_{i-1,i}$. Indeed, if $\sfx=(x_{1},\ldots,x_{g})\in E^g$, then
  \begin{equation*}
    \sfx\in s_{i}^{*}\Delta_{i,i+1}\iff x_{i-1}-x_{i}+x_{i+1}=x_{i+1}\iff x_{i-1}=x_{i}\iff\sfx\in\Delta_{i-1,i}.
  \end{equation*}
  Since $s_{i}^{2}=1$, $s_{i}^{*}\Delta_{i-1,i}=\Delta_{i,i+1}$. Similarly, $s_{1}^{*}\Delta_{1,2}=(0)\times E^{g-1}$ and $s_{g}^{*}\Delta_{g-1,g}=E^{g-1}\times(0)$. Thus \cref{eq.lnk.equiv.str.first} is stable under all $s_{i}$. If $1\leq i\leq g$ and $n_{i}=2$, then $E^{j-1}\times\fd_{j}\times E^{g-j}$ is stable under the action of $s_{i}$ for all $j\neq i$. So \cref{eq.lnk.equiv.str.second} is stable under $s_{i}$. Therefore $D$ is stable under all $s_{i}$ satisfying $n_{i}=2$ so we can equip $\cO_X(D)$ with the equivariant structure 
  described in \Cref{le.inv-div}. 
  
  In particular, all this applies to $\cL_{n/k}\cong\cO_{E^{g}}(D_{n/k})$.  
\end{proof}

Let $D$ be a balanced standard divisor of type $(n_{1},\ldots,n_{g})$ where $g\geq 2$ and $n_{i}\geq 2$ for all $i$. Later we need to consider the restriction of $D$ to the closed subvariety
\begin{equation*}
	Y\;:=\;E^{t-1}\times\{z\}\times E^{g-t}
\end{equation*}
for a point $z\in E$ and $1\leq t\leq g$ satisfying $n_{t}\geq 3$. Obviously $Y$ is stable under the action of $\Sigma_{n/k}$, and there is a canonical isomorphism
\begin{equation*}
	\Sigma_{n/k}\;\cong\;\Sigma_{[n_{1},\ldots,n_{t-1}]}\times\Sigma_{[n_{t+1},\ldots,n_{g}]}
\end{equation*}
that sends each $s_{i}$ to $s_{i}$ on either $E^{t-1}$ or $E^{g-t}$, which identifies these two groups. Note that the group on the right-hand side is naturally acting on $E^{t-1}\times E^{g-t}$.

\begin{lemma}\label{lem.restr.dnk.equiv}
	In the above setting, there is a $\Sigma_{[n_{1},\ldots,n_{t-1}]}\times\Sigma_{[n_{t+1},\ldots,n_{g}]}$-equivariant isomorphism
	\begin{equation*}
		\psi\colon E^{t-1}\times E^{g-t}\to Y
	\end{equation*}
	such that $\psi^{*}(\cO_{E^{g}}(D)|_{Y})$ is isomorphic to
	\begin{equation*}
		\cO(D_{\ell}\times E^{g-t}+E^{t-1}\times D_{r}),
	\end{equation*}
	where $D_{\ell}$ and $D_{r}$ are some balanced standard divisors of types $(n_{1},\ldots,n_{t-1})$ and $(n_{t+1},\ldots,n_{g})$, respectively.
\end{lemma}
\begin{proof}
	As in the proof of \cref{prop.lnk.equiv.str}, we can write $D=D'+D_{\fd_{i}}$ where $D'$ is \cref{eq.lnk.equiv.str.first} and $D_{\fd_{i}}$ is \cref{eq.lnk.equiv.str.second}. The restriction of $\cO_{E^{g}}(D')$ is
	\begin{equation*}
		\cO_{E^{g}}(D')|_{Y}\;=\;\cO_{Y}(D'_{\ell}\times(z)\times E^{g-t}+E^{t-1}\times(z)\times D'_{r})
	\end{equation*}
	where
	\begin{align*}
		D'_{\ell}&\;=\;(0)\times E^{t-2}+\Delta_{1,2}+\cdots+\Delta_{t-2,t-1}+E^{t-2}\times (z),\\
		D'_{r}&\;=\;(z)\times E^{g-t-1}+\Delta_{1,2}+\cdots+\Delta_{g-t-1,g-t}+E^{g-t-1}\times (0),
	\end{align*}
	with the convention that $D'_{\ell}=0$ when $t=1$ and $D'_{r}=0$ when $t=g$.
	Choose $y_{\ell},y_{r}\in E$ so that $ty_{\ell}=z=(g-t+1)y_{r}$, and define the isomorphism $\psi\colon E^{t-1}\times E^{g-t}\to Y$ by
	\begin{equation*}
		\begin{split}
			&\psi(x_{1},\ldots,x_{t-1},x_{t+1},\ldots,x_{g})\\
			&\;=\;(x_{1}+y_{\ell},x_{2}+2y_{\ell},\ldots,x_{t-1}+(t-1)y_{\ell},z,x_{t+1}+(g-t)y_{r},\ldots,x_{g-1}+2y_{r},x_{g}+y_{r}).
		\end{split}
	\end{equation*}
	It is easy to see $\psi$ commutes with $s_{i}$ for all $i\neq t$, so $\psi$ is $\Sigma_{[n_{1},\ldots,n_{t-1}]}\times\Sigma_{[n_{t+1},\ldots,n_{g}]}$-equivariant. The pullback of $\cO_{E^{g}}(D')|_{Y}$ is
	\begin{equation*}
		\psi^{*}(\cO_{E^{g}}(D')|_{Y})\;=\;\cO(D''_{\ell}\times E^{g-t}+E^{t-1}\times D''_{r})
	\end{equation*}
	where
	\begin{align*}
		D''_{\ell}&\;=\;(-y_{\ell})\times E^{t-2}+\Delta^{-y_{\ell}}_{1,2}+\cdots+\Delta^{-y_{\ell}}_{t-2,t-1}+E^{t-2}\times (y_{\ell}),\\
		D''_{r}&\;=\;(y_{r})\times E^{g-t-1}+\Delta^{y_{r}}_{1,2}+\cdots+\Delta^{y_{r}}_{g-t-1,g-t}+E^{g-t-1}\times (-y_{r}).
	\end{align*}
	By \cref{pr.ssw}\cref{item.prop.ssw.dz}, $D''_{\ell}$ and $D''_{r}$ are linearly equivalent to the divisors of the form \cref{eq.lnk.equiv.str.first} on $E^{t-1}$ and $E^{g-t}$, respectively.
	
	On the other hand,
	\begin{equation*}
		\psi^{*}(\cO_{E^{g}}(D_{\fd_{i}})|_{Y})\;=\;\cO(D_{\fd^{\ell}_{i}}\times E^{g-t}+E^{t-1}\times D_{\fd^{r}_{j}})
	\end{equation*}
	where $D_{\fd^{\ell}_{i}}$ and $D_{\fd^{r}_{j}}$ are the divisors of the form \cref{eq.lnk.equiv.str.second} for respective divisors $\fd^{\ell}_{i}$ ($1\leq i\leq t-1$) and $\fd^{r}_{j}$ ($t+1\leq j\leq g$) on $E$ with $\deg\fd^{\ell}_{i}=n_{i}-2$ and $\deg\fd^{r}_{j}=n_{j}-2$. Therefore the desired result follows.
\end{proof}

We end this discussion with the following result, which we later generalize in \Cref{prop.factor.Phi}.

\begin{lemma}\label{pr.all2}
  When all $n_i=2$, i.e., $k=g=n-1$, the morphism $\Phi_{n/k}$ from \Cref{defn.Phi_D} factors as
  \begin{equation}\label{eq:7}
    \xymatrix{
  E^g \ar@/^2pc/[rr]^-{\Phi_{n/k}}  \ar[r]_>>>>>{\phi} & E^g/\Sigma_{n/k}   \ar[r]_-{\iota} & \PP^{n-1}
}
  \end{equation}
where $\phi$ is the quotient map and $\iota$ is an isomorphism.
\end{lemma}
\begin{proof}
\Cref{prop.emb.ell} shows that $\Phi_{n/k}$ is a surjection of $E^g$ onto $\PP^g$. Since all $n_i=2$, $\Sigma_{n/k}\cong \Sigma_{g+1}$ and $E^g/\Sigma_{n/k}\cong \PP^{g}$. Indeed, by \Cref{pr.equiv}, the morphism $\ve:E^g \to E^{g+1}$ is equivariant with respect to this action and the natural action of $\Sigma_{g+1}$ on $E^{g+1}$. But $E^{g+1}/\Sigma_{g+1}$  is the symmetric power $S^{g+1}E$ and the composition 
$E^g \stackrel{\ve}{\longrightarrow} E^{g+1} \to S^{g+1}E$ sends $E^g$ to the fiber over 0 of the morphism 
$S^{g+1}E \to E$, $(\!(z_1,\ldots,z_{g+1})\!) \mapsto z_1+\cdots + z_{g+1}$.  
It is well-known that all fibers of this morphism are isomorphic to $\PP^g$.

By \Cref{prop.lnk.equiv.str}, $\cL=\cL_{n/k}$ admits a $\Sigma_{n/k}$-equivariant structure. Since $\Phi_{n/k}\colon E^{g}\to \PP(H^0(E^g,\cL)^*)$ factors through the quotient $\phi\colon E^g\to E^g/\Sigma_{n/k}$, \cref{rem.char.triv} implies that the action of $\Sigma_{n/k}$ on $H^0(E^g,\cL)$ is trivial. \Cref{pr.gen-equivariant} thus applies to show that $(\phi_*\cL)^{\Sigma_{n/k}}$ is an invertible $\cO_{E^g/\Sigma_{n/k}}$-module. Note also that
\begin{equation}\label{eq:6-2}
  H^0\left(E^g/\Sigma_{n/k},(\phi_*\cL)^{\Sigma_{n/k}}\right) \; \cong \; H^0(E^g/\Sigma_{n/k},\phi_*\cL)^{\Sigma_{n/k}}  \; \cong \; H^0(E^g,\cL)^{\Sigma_{n/k}}\; =\;  H^0(E^g,\cL).
\end{equation}
The morphism
$$
\xymatrix{ 
  \PP^g \; \cong \; E^g/\Sigma_{n/k}   \ar[rr]^-{\iota} && \PP(H^0(E^g,\cL)^*) \; \cong \; \PP^g
}
$$
is induced by $(\phi_*\cL)^{\Sigma_{n/k}}$ via the identification \Cref{eq:6-2}; since $(\phi_*\cL)^{\Sigma_{n/k}}$ has a
$(g+1)$-dimensional space of global sections it must be isomorphic to $\cO_{\PP^g}(1)$. This proves the claim that $\iota$ is an isomorphism.
\end{proof}

%%%%%%%%%%%%%%%%%%%%%%%%%%%%%%%%%%%%%%%%%%%%%%%%%%%%%%%%%%%%%%%%
\subsection{The restriction of $\cL_{n/k}$ to certain subvarieties of $E^g$}
%%%%%%%%%%%%%%%%%%%%%%%%%%%%%%%%%%%%%%%%%%%%%%%%%%%%%%%%%%%%%%%%
%%%%%%%%%%%%%%%%%%%%%%%%%%%%%%%%%%%%%%%%%%%%%%%%%%%%%%%%%%%%%%%%

\begin{lemma}
\label{prop.surj.sections}
Suppose that $n_{i}\geq 2$ for all $i$ and fix $t$ such that $n_{t}\geq 3$. Let $Y:=E^{t-1}\times \{z\}\times E^{g-t}$, $D$ a standard divisor of type $(n_{1},\ldots,n_{g})$, and $\cL:=\cO_{E^{g}}(D)$. The natural map 
$$
H^0(E^g,\cL) \; \longrightarrow H^0(Y,\cL\vert_{Y})
$$
is onto.
\end{lemma}
\begin{pf}
There is an exact sequence $0\to \cL(-Y)\to\cL\to\cL|_{Y}\to 0$ so the result will follow once
we show that $H^1(E^g,\cL(-Y))=0$. The sheaf $\cL(-Y)$ is again a balanced standard divisor but
$n_i$ is replaced by $n_i-1$. Thus
the cohomology vanishing follows from \cref{cor.xi-h0}.
\end{pf}

%%%%%%%%%%%%%%%%%%%%%%%%%%%%%%%%%%%%%%%%%%%%%%%%%%%%%%%%%%%%%%%%%%%%%%%%%%%%%%%%
\subsection{The characteristic variety is isomorphic to $E^g/\Sigma_{n/k}$}
\label{subse.char.var.is.a.quotient}
%%%%%%%%%%%%%%%%%%%%%%%%%%%%%%%%%%%%%%%%%%%%%%%%%%%%%%%%%%%%%%%%%%%%%%%%%%%%%%%%

In this section we show that $X_{n/k}$ is isomorphic to $E^{g}/\Sigma_{n/k}$. 
We argue by induction on $g$; various divisors that are not necessarily of the form $D_{n/k}$ appear;  we state the result in terms of balanced standard divisors.

\begin{theorem}\label{thm.char.var.quot}
  Let $D$ be a balanced standard divisor of type $(n_{1},\ldots,n_{g})$. Then the morphism $\Phi_{|D|}\colon E^{g}\to\bbP^{n-1}$ factors as
  \begin{equation}\label{eq:factor.Phi}
    \xymatrix{
      E^g \ar@/^2pc/[rr]^-{\Phi_{|D|}}  \ar[r]_>>>>>{\phi} & E^g/\Sigma_{n/k}   \ar[r]_-{\iota} & \PP^{n-1}
    }
  \end{equation}
  in which $\phi$ is the quotient morphism and $\iota$ is a closed immersion.
\end{theorem}

First we consider the case $g=1$, i.e., $k=1$. Since $D$ is a degree-$n$ divisor on $E$, it is very ample if $n\geq 3$. If $n=2$, then the claim follows from \cref{pr.all2}.

The proof for $g\geq 2$ consists of a series of partial results.

\begin{lemma}\label{le.factors}
  The morphism $\Phi_{|D|}$ factors as in \Cref{eq:factor.Phi}.
\end{lemma}
\begin{proof}
Let $\sfz$ be an arbitrary point in $E^g$.  We must show that $\Phi_{|D|}(\sfz)=\Phi_{|D|}(s_i(\sfz))$ for all $i$ such that $n_i=2$;
i.e., we must show that if $D'$ is an effective divisor linearly equivalent to $D$ and $\sfz \in \Supp D'$, then $s_i(\sfz)\in \Supp D'$ for all $i$ such that $n_i=2$.

Let $D'$ be an effective divisor such that $D'\sim D$, and suppose $\sfz=(z_1,\ldots,z_g)$ is in $\Supp D'$. Fix $i$ such that $n_{i}=2$. Let $\overline{D}$ and $\overline{D'}$ be the restrictions of $D$ and $D'$, respectively, to
$$
E_i \; := \; \{(z_1,\ldots,z_{i-1})\} \times E\times \{(z_{i+1},\ldots,z_g)\}.
$$
Since $D\sim D'$, $\overline{D} \sim \overline{D'}$ as divisors on $E_i$. 
Since $D$ is a balanced standard divisor of type $(n_{1},\ldots,n_{g})$ and $n_{i}=2$,
\begin{equation*}
\overline{D} \;=\; (z_1,\ldots,z_{i-1}, z_{i-1}, z_{i+1},\ldots,z_g) + (z_1,\ldots,z_{i-1}, z_{i+1}, z_{i+1},\ldots,z_g)
\end{equation*}
where $z_{0}=z_{g+1}=0$.
Since $\overline{D} \sim \overline{D'}$,
$$
\overline{D'} \;=\; (z_1,\ldots,z_{i-1}, p,z_{i+1},\ldots,z_g)\,+\, (z_1,\ldots,z_{i-1}, q, z_{i+1},\ldots,z_g)
$$
where $p+q=z_{i-1}+z_{i+1}$.
By hypothesis, $(z_1,\ldots,z_g) \in \Supp D'$ so we may assume $p=z_{i}$ and $q=z_{i-1}-z_i+z_{i+1}$. Thus $\overline{D'}=\sfz+s_{i}(\sfz)$, and hence $s_{i}(\sfz)\in D'$.
\end{proof}

\begin{lemma}\label{le.iota-inj}
The morphism $\iota$ in \Cref{eq:factor.Phi} is injective.   
\end{lemma}
\begin{proof}
Suppose $\sfz$ and $\sfw$ are in different $\Sigma_{n/k}$-orbits.  We will find an effective divisor $D'$ that is linearly equivalent to $D$ such that $\sfz \in \Supp D'$ but $\sfw \notin \Supp D'$. We do this by induction on $g$, reducing to the case when all $n_i$ are $2$ (or $g=1$).

{\bf (1) The case when $z_i\ne w_i$ and $n_i\ge 3$ for some $i$.} This follows from \cref{lem.sep.pts.ni.geq.three}.

{\bf (2) The case when  $z_i=w_i$ and $n_i\ge 3$ for some $i$.}
The points $\sfz$ and $\sfw$ both belong to
\begin{equation*}
	Y \; :=\;  E^{i-1}\times \{z_i\} \times E^{g-i}.
\end{equation*}
Applying \cref{lem.restr.dnk.equiv} to $D$ and $Y$, we obtain a $\Sigma_{[n_{1},\ldots,n_{i-1}]}\times\Sigma_{[n_{i+1},\ldots,n_{g}]}$-equivariant isomorphism $\psi\colon E^{i-1}\times E^{g-i}\to Y$ and balanced standard divisors $D_{\ell}$ and $D_{r}$ of types $(n_{1},\ldots,n_{i-1})$ and $(n_{i+1},\ldots,n_{g})$ such that
\begin{equation}\label{eq.iota.inj.lin.eq}
	\psi^{*}(\cO_{E^{g}}(D)|_{Y})\;\cong\;\cO(\overline{D})
\end{equation}
where $\overline{D}=D_{\ell}\times E^{g-i}+E^{i-1}\times D_{r}$.
There is a commutative diagram
\begin{equation}\label{eq.iota.inj.comm.diag}
	\begin{tikzcd}
		E^{i-1}\times E^{g-i}\ar[d,"\Phi_{|\psi^{*}(\cO(D)|_{Y})|}"]\ar[r,"\psi","\cong"'] & Y\ar[d,"\Phi_{|\cO(D)|_{Y}|}"]\ar[r] & E^{g}\ar[d,"\Phi_{|D|}"] \\
		\bbP(H^{0}(\psi^{*}(\cO_{E^{g}}(D)|_{Y}))^{*})\ar[r,"\cong"'] & \bbP(H^{0}(\cO_{E^{g}}(D)|_{Y})^{*})\ar[r] & \bbP(H^{0}(\cO_{E^{g}}(D))^{*})
	\end{tikzcd}
\end{equation}
where the lower right horizontal arrow is a closed immersion since the canonical map
\begin{equation*}
  H^0(E^g,\cO_{E^{g}}(D))\to H^0(Y,\cO_{E^{g}}(D)\vert_Y)
\end{equation*}
is onto by \cref{prop.surj.sections}. So it suffices to prove that $\Phi_{|\psi^{*}(\cO(D)|_{Y})|}$ sends $\psi^{-1}(\sfz)$ and $\psi^{-1}(\sfw)$ to different points. By \cref{eq.iota.inj.lin.eq}, we can replace $\Phi_{|\psi^{*}(\cO(D)|_{Y})|}$ by $\Phi_{|\overline{D}|}$, which is the composition of
\begin{equation*}
	\Phi_{|D_{\ell}|}\times\Phi_{|D_{r}|}\colon E^{i-1}\times E^{g-i}\to\bbP(H^{0}(\cO_{E^{i-1}}(D_{\ell}))^{*})\times\bbP(H^{0}(\cO_{E^{g-i}}(D_{r}))^{*})
\end{equation*}
and the Segre embedding
\begin{equation*}
	\bbP(H^{0}(\cO_{E^{i-1}}(D_{\ell}))^{*})\times\bbP(H^{0}(\cO_{E^{g-i}}(D_{r}))^{*})\to\bbP((H^{0}(\cO_{E^{i-1}}(D_{\ell}))\otimes H^{0}(\cO_{E^{g-i}}(D_{r})))^{*}).
\end{equation*}
By the induction hypothesis, the morphisms $\iota$ for $D_{\ell}$ and $D_{r}$ are injective. Since $\psi^{-1}(\sfz)$ and $\psi^{-1}(\sfw)$ do not belong to the same orbit, they are sent to different points by $\Phi_{|D_{\ell}|}\times\Phi_{|D_{r}|}$, and so by $\Phi_{|\overline{D}|}$. Therefore $\iota$ for $D$ is injective.

{\bf (3) The case when $n_i=2$ for all $i$.} The only balanced standard divisor of type $(2,\ldots,2)$ is of the form \cref{eq.min.bal.std.div}. So the claim follows from \Cref{pr.all2}.
\end{proof}

\begin{lemma}\label{le.tg-inj}
	The morphism $\iota$ in \Cref{eq:factor.Phi} induces embeddings of tangent spaces.
\end{lemma}
\begin{proof}
Write $\Sigma=\Sigma_{n/k}$ and $\cL=\cO_{E^{g}}(D)$. The sheaf $\cL$ admits the $\Sigma$-equivariant structure on $\cL$ described in \Cref{prop.lnk.equiv.str} and, by \cref{rem.char.triv}, the $\Sigma$-action on $H^0(E^g,\cL)$ is trivial. By \Cref{pr.gen-equivariant},
$(\phi_*\cL)^{\Sigma} $  is an invertible sheaf on $E^g/\Sigma$ such that  
$$
\phi^*\big((\phi_*\cL)^{\Sigma} \big) \cong \cL
$$ 
and,  as in \Cref{eq:6-2},
\begin{equation*}
  H^0(E^g/\Sigma,\big(\phi_*\cL)^{\Sigma} \big))\;\cong\;H^0(E^g,\cL).
\end{equation*}
The morphism $  \iota\colon E^g/\Sigma\to \PP^{n-1}$  is $\Phi_{|(\phi_*\cL)^{\Sigma}|}$.

We prove the claim by induction on $g$. For the induction step we consider two cases:

{\bf (1) All $n_i=2$.} This is nothing but \Cref{pr.all2}.

{\bf (2) There is some $i$ with $n_i\ge 3$.} Consider a point
\begin{equation*}
 \sfx=(x_1,\cdots,x_g)\in E^g 
\end{equation*}
with the goal of showing that $\iota$ is one-to-one on $T_{\phi(\sfx)}(E^g/\Sigma)$. Because $n_i\ge 3$, \cref{lem.dnk.pt.in.only.div} implies that there is a standard divisor $D'$ that is linearly equivalent to $D$ such that
\begin{itemize}
\item $D_0:=E^{i-1}\times(x_i)\times E^{g-i}$ is a component of $D'$ with multiplicity one;
\item no other component of $D'$ contains $\sfx$.
\end{itemize}
$D_0$ is invariant under $\Sigma$ so we have an embedding
\begin{equation*}
  D_0/\Sigma\;\subseteq\;E^g/\Sigma.
\end{equation*}
The divisor $D'$ is the zero locus of a section $s\in H^0(E^g,\cL)$. When regarded as an element of $H^0(E^g/\Sigma,(\phi_{*}\cL)^{\Sigma})$, $s$ vanishes on $D_0$ and no other components of its zero locus contains $\sfy:=\phi(\sfx)$.

Now consider a non-zero tangent vector $v\in T_{\sfy}(E^g/\Sigma)$. There are two possibilities:

{\bf (2a) $v$ is not tangent to $D_0/\Sigma$.} Since $D_0/\Sigma$ is the only component of the zero locus of $s\in H^0(E^g/\Sigma,(\phi_{*}\cL)^{\Sigma})$ containing $\sfy$, this means that $v$ is not tangent to that zero locus. But then the section $s$ witnesses tangent vector separation at $v\in T_{\sfy}(E^g/\Sigma)$.

{\bf (2b) $v$ is tangent to $D_0/\Sigma$.} Let $Y:=D_{0}$. By the argument in the proof of \cref{le.iota-inj}, we obtain commutative diagram \cref{eq.iota.inj.comm.diag}. By the induction hypothesis, $\iota$ sends $v$ to a non-zero tangent vector.
\end{proof}

\begin{proof}[Proof of \cref{thm.char.var.quot}]
	To see that $\iota$ is an isomorphism onto its image we can apply \cite[Cor.~14.10]{H92}, stating that this follows from the injectivity of $\iota$ provided it induces embeddings of tangent spaces. The two conditions are provided by \Cref{le.iota-inj,le.tg-inj}.
\end{proof}

\begin{corollary}
\label{prop.factor.Phi}
The morphism $\Phi_{n/k}\colon E^g \to \PP^{n-1}$ factors as
\begin{equation}
\xymatrix{
  E^g \ar@/^2pc/[rr]^-{\Phi_{n/k}}  \ar[r]_>>>>>{\phi} & E^g/\Sigma_{n/k}   \ar[r]_-{\iota} & \PP^{n-1}
}
\end{equation}
in which $\phi$ is the quotient morphism and $\iota$ is a closed immersion. In particular, 
\begin{enumerate}
\item\label{item.char.var.quot}
$X_{n/k} \cong E^g/\Sigma_{n/k}$ and
\item\label{item.autom.desc} the automorphism $\s\colon E^g \to E^g$ defined by $\s(\sfz)=\sfz+(\sfk+\sfl-\sfn)\tau$ descends via $\Phi_{n/k}$ to an automorphism of $X_{n/k}$ that we also denote by $\s$.
\end{enumerate}
\end{corollary}
\begin{proof}
	The first claim is \cref{thm.char.var.quot} for $D=D'_{n/k}$. Statement \cref{item.char.var.quot} now follows from the definition of $X_{n/k}$ as the image of $\Phi_{n/k}$, and \cref{item.autom.desc} is now  \cref{prop.t.sigma.comm}\cref{item.t.sigma.comm.des}.
\end{proof}

\subsubsection{Very ampleness}
\label{ssect.very.ample.1}
By \Cref{thm.mumf.p77}, ampleness or very ampleness of $D_{n/k}$ depends only on the class of $D_{n/k}$
in $\NS(E^g)$. By \cref{pr.ample}, $D_{n/k}$ is ample. It is stated at \cite[p.~212]{FO89} that $D_{n/k}$ is {\it very} ample when 
all $n_i$ are $\ge 3$. We now give a direct geometric proof of the following stronger statement.  (See also \cref{ssect.very.ample.2}.)

\begin{proposition}
\label{pr.va}
A standard divisor $D$ of type $(n_{1},\ldots,n_{g})$ is very ample if and only if $n_i \ge 3$ for all $i$. This applies to $D=D_{n/k}$ in particular.
\end{proposition}
\begin{proof}
($\Rightarrow$)
Assume $D$ is very ample.

Let $x \in E$.  For each $i$, let $C_i\subseteq E^g$ be the curve consisting  of the points whose $j^{\rm th}$ coordinate 
is $x$ for all $j\ne i$. Since $D$ is a very ample divisor on $E^g$, $\cO_{E^{g}}(D)|_{C_{i}}$ is a very ample invertible $\cO_{C_i}$-module. 
Since $C_i\cong E$, $\deg\cO_{E^{g}}(D)|_{C_{i}} \ge 3$; but $ \deg\cO_{E^{g}}(D)|_{C_{i}} = D\cdot C_i = n_i$ so $n_i \ge 3$.

($\Leftarrow$)
Since $D$ is linearly equivalent to a translate of a balanced standard divisor by a point of $E^{g}$, we can assume that $D$ is balanced. 
If $n_i \ge 3$ for all $i$, then the group $\Sigma_{n/k}$ is trivial. So this is a special case of \cref{thm.char.var.quot}.
\end{proof}

%%%%%%%%%%%%%%%%%%%%%%%%%%%%%%%%%%%%%%%%%%%%%%%%%%%%%%%%%%%%%%%%%%%%%%%%%%%%%%%%
\subsection{A description of the characteristic variety as a bundle over a power of $E$}
\label{sect.Xn/k.again}
%%%%%%%%%%%%%%%%%%%%%%%%%%%%%%%%%%%%%%%%%%%%%%%%%%%%%%%%%%%%%%%%%%%%%%%%%%%%%%%%

If $\frac{n}{k}=[n_1,\ldots,n_g]$, we define 
\begin{equation}
\label{defn.Sigma}
\Sigma \; = \;\Sigma_{n/k} \;:=\;\langle (i,i+1) \; | \; n_i=2\rangle \, \subseteq \, \Sigma_{g+1} \subseteq \Aut(E^{g+1}). 
\end{equation}
Thus, the characteristic variety $X_{n/k}$ is isomorphic to $E^g/\Sigma$ where we identify $E^g$ with the subgroup $E^{g+1}_0$ of 
$E^{g+1}$ consisting of those points whose coordinates sum to zero, the identification being given by the morphism $\ve$ in \cref{defn.varepsilon}.

We use the following the notation in this subsection:\index{J@$J$}\index{I_alpha@$I_{\a}$}\index{Sigma_I_alpha@$\Sigma_{I_{\a}}$}
\begin{align*}
J & \; := \; \text{the points in $\{1,\ldots,g+1\}$ fixed by $\Sigma$},
\\
I_1,\ldots,I_s  & \; := \; \text{the $\Sigma$-orbits in $\{1,\ldots,g+1\}$ of size $\ge 2$},
\\
\Sigma_{I_\a} & \; :=\; \text{the group of permutations of $I_\a$}.
\end{align*}
The partition $\{1,\ldots,g+1\}= J \sqcup I_1 \sqcup  \ldots \sqcup I_s$ 
corresponds to factorizations
$E^{g+1}=  E^{J} \times E^{I_1}  \times \cdots \times E^{I_s}$
and $\Sigma=\Sigma_{I_1} \times \cdots \times \Sigma_{I_s}$. (Both these are factorizations as a product of subgroups.)
It follows that
\begin{equation}
\label{E^(g+1)/Sigma}
E^{g+1}/\Sigma \; \cong \;  E^{J} \times S^{I_1}E \times \cdots \times S^{I_s}E 
\end{equation}
where $S^{I_\a}E=E^{I_\a}/\Sigma_{I_\a}$\index{S^I_alphaE@$S^{I_\a}E$} is the symmetric power of $E$ of dimension $|I_\a|$.
As is well-known, the summation function $S^rE \to E$ presents $S^rE$ as a $\PP^{r-1}$ bundle over $E$.
Thus, $E^{g+1}/\Sigma$ is a bundle over $E^{J} \times E^s$
with fibers isomorphic to $\PP^{|I_1|-1} \times  \cdots \times \PP^{|I_s|-1}$.

\begin{proposition} 
cf. \cite[Prop.~2.9]{AA18}
\label{prop.E^g/Sigma}
With the above notation, $E^g/\Sigma$ is a bundle over $E^{|J|+s-1}$ with fibers 
isomorphic to $\PP^{|I_1|-1} \times  \cdots \times \PP^{|I_s|-1}$. In particular,
\begin{equation}
\label{E^g/Sigma}
E^g/\Sigma \; \cong \; 
\begin{cases} 
\big(S^{I_1}E \times  \cdots \times S^{I_s}E\big)_0 & \text{if $J=\varnothing$,}
\\
E^{|J|-1} \times S^{I_1}E \times  \cdots \times S^{I_s}E &\text{if  $J \ne \varnothing$,}
\end{cases}
\end{equation}
where $(S^{I_1}E \times  \cdots \times S^{I_s}E)_0$ denotes the subvariety of $S^{I_1}E \times  \cdots \times S^{I_s}E$ whose 
coordinates sum to $0$.
\end{proposition}
\begin{proof}
The structure map, $\pi:E^{g+1}/\Sigma \to E^{J} \times E^s$, for the bundle described just before
the statement of the proposition is
$$
\pi(\sfz,\sfx_1,\ldots,\sfx_s) \; = \; (\sfz,\sfsum\sfx_1, \ldots, \sfsum\sfx_s)
$$
where $\sfz \in E^J$, $\sfx_\a \in S^{I_\a}E$, and $\sfsum\sfx_\a$ is the sum of the coordinates of $\sfx_\a \in S^{I_\a}E$.
As we said at the outset, we identify $E^g$ with the subgroup $E^{g+1}_0$ of $E^{g+1}$ consisting of those points whose 
coordinates sum to 0. Hence $E^g/\Sigma=E^{g+1}_0/\Sigma = \pi^{-1}(E^{|J|+s}_0)$  is the restriction of the 
bundle to $E^{|J|+s}_0 \cong E^{|J|+s-1}$ which is, of course, a bundle with fibers isomorphic to 
$\PP^{|I_1|-1} \times  \cdots \times \PP^{|I_s|-1}$. 

The description of $E^g/\Sigma$ in \cref{E^g/Sigma} follows from the description of $E^{g+1}/\Sigma$ in \cref{E^(g+1)/Sigma}.
\end{proof}

We thank S\'andor Kov\'acs for the next result.

\begin{proposition}
[Kov\'acs]
\label{prop.Kovacs}
Let $X$ be a projective variety. For $i=1,2$, let $f_i:X \to Z_i$ be fiber bundles with the property that the fibers of both are rationally connected varieties. If there are no rational curves on either $Z_1$ or $Z_2$, then fibers of $f_1$ 
are fibers of $f_2$, and conversely. 
\end{proposition}
\begin{proof}
Let $X_1 \subseteq X$ be a fiber of $f_1$. Since $X_1$ is rationally connected and $Z_2$ has no rational curves $f_2(X_1)$ is a single
point on $Z_2$. Thus $X_1$ is contained in a fiber,  $X_2$ say, of $f_2$. Similarly, $X_2$ is contained in a fiber of $f_1$; but that
fiber contains $X_1$ so we must have $X_1=X_2$. 
\end{proof}

By \cref{prop.E^g/Sigma}, there are, after fixing $g$, 
only finitely many possibilities for the bundle structure of the characteristic variety $X_{n/k}$. 
It is reasonable to fix one of those possible structures then ask for all pairs $(n,k)$ such that $X_{n/k}$ has that structure. 
We have already seen one example of this:  
$X_{n/k}$ is isomorphic to $E^g$ if and only if $n_i \ge 3$ for all $i$. \cref{cor.E^g/Sigma=S^gE} classifies those $(n,k)$ for which 
$X_{n/k}$ is a symmetric power of $E$. 

\begin{proposition}\label{prop.isom.bundle}
	If the characteristic variety $X_{n/k}$ is a bundle over $E^{r}$ with fibers isomorphic to $\PP^{a_{1}-1}\times\cdots\times\PP^{a_{q}-1}$, then, in the notation of \cref{prop.E^g/Sigma}, $r=|J|+s-1$ and there is an equality of multisets
	\begin{equation}\label{eq.multisets}
		\{\!\{a_{1},\ldots,a_{q}\}\!\}\;=\;\{\!\{|I_{1}|,\ldots,|I_{s}|\}\!\}.
	\end{equation}
\end{proposition}
\begin{proof}
	Since products of projective spaces are rationally connected and there are no rational curves on products of elliptic curves, \cref{prop.Kovacs} implies that the fiber $\PP^{a_{1}-1}\times\cdots\times\PP^{a_{q}-1}$ is isomorphic to the fiber $\PP^{|I_1|-1} \times  \cdots \times \PP^{|I_s|-1}$ described in \cref{prop.E^g/Sigma}. Since the Chow ring of $\PP^{a_{1}-1}\times\cdots\times\PP^{a_{q}-1}$ is $\ZZ[t_{1},\ldots,t_{q}]$ with $\deg t_{i}=1$ modulo the relations $t_{i}^{a_{i}}=0$ ($1\leq i\leq q$), the equality \cref{eq.multisets} holds. Therefore
	\begin{equation*}
		r\;=\;\dim X_{n/k}-\sum_{i=1}^{q}(a_{i}-1)\;=\;g-\sum_{j=1}^{s}(|I_{j}|-1)\;=\;|J|+s-1.\qedhere
	\end{equation*}
\end{proof}

\begin{corollary}
\label{cor.E^g/Sigma=S^gE}
The characteristic variety is isomorphic to $S^gE$ if and only if $\frac{n}{k}=[n_1,\ldots,n_g]$ is either $[m,2,\ldots,2]$ or $[2,\ldots,2,m]$ for some integer $m \ge 3$.
\end{corollary}
\begin{proof}
($\Leftarrow$) 
If $\frac{n}{k}=[m,2,\ldots,2]$, then $J=\{1\}$, $s=1$, and $I_1=\{2,\ldots,g+1\}$, so $X_{n/k} \cong S^gE$. 
If $\frac{n}{k}=[2,\ldots,2,m]$, then  $J=\{g+1\}$, $s=1$, and $I_1=\{1,\ldots,g\}$, so $X_{n/k} \cong S^gE$. 

($\Rightarrow$) 
Assume $X_{n/k} \cong S^gE$. Thus $X$ is a bundle over $E$ with fibers isomorphic to $\PP^{g-1}$. By \cref{prop.isom.bundle}, $|J|+s-1=1$ and $\{\!\{|I_{1}|,\ldots,|I_{s}|\}\!\}=\{\!\{g\}\!\}$. Thus $s=1$, $|J|=1$, and $|I_{1}|=g$.

Since $|J|=1$, $\Sigma_{n/k}$ fixes a unique point in $\{1,\ldots,g+1\}$ and, since $s=1$, that point is either $1$ or 
$g+1$. In the first case, $\frac{n}{k}=[m,2,\ldots,2]$ with $m \ge 3$.  In the second case, $\frac{n}{k}=[2,\ldots,2,m]$ with $m \ge 3$. 
 \end{proof}

We now make the isomorphism $E^g/\Sigma_{n/k} \to S^gE$ explicit for those cases.

\begin{proposition}
\label{prop.S^gE}
The morphisms $\rho:E^g \to S^gE$ given by 
\begin{equation}
\label{eq:rho1}
\rho(z_1,\ldots,z_g) \; :=\; (\!(z_2-z_1,z_3-z_2,\ldots,z_g-z_{g-1},-z_g)\!)
\end{equation} 
and
\begin{equation}
\label{eq:rho2}
\rho(z_1,\ldots,z_g) \; :=\; (\!(-z_1,z_1-z_2,z_2-z_3,\ldots,z_{g-1}-z_g)\!)
\end{equation}
are quotients for the action of $\Sigma_{n/k}$ on $E^g$ when $\frac{n}{k}=[m,2,\ldots,2]$ and  $\frac{n}{k}=[2, \ldots,2,m]$,
respectively.
\end{proposition}
\begin{proof}
We will use the following general fact.
Let $\G$ and $\G'$ be groups acting on a scheme $X$ and suppose 
$\phi:X \to Y$ is a (categorical) quotient morphism for the action of $\G$. If $\nu:X \to X$ is an automorphism such that $\G'=\nu^{-1}\G \nu$, then
the morphism $\rho=\phi \circ \nu :X \to Y$ is a quotient for the action of $\G'$.  

Let  $\phi:E^g \to S^gE$ be the usual quotient morphism for the natural action of $\Sigma_g$  on $E^g$. 

Assume $\frac{n}{k}=[m,2,\ldots,2]$. Then $\Sigma_{n/k}$ is generated by $s_2,\ldots,s_g$. 
It is easily checked that if $\nu$ is the automorphism of 
$E^g$ given by $\nu(z_1,\ldots,z_g) = (z_2-z_1,z_3-z_2,\ldots,z_g-z_{g-1},-z_g)$, then
$\nu \circ s_i = (i-1,i) \circ \nu$ for $i=2,\ldots,g$ . Hence $\Sigma_{n/k}=\nu^{-1}\Sigma_g \nu$ and $\phi \circ \nu:E^g \to S^gE$ is a
quotient for the action of $\Sigma_{n/k}$ on $E^g$; the morphism $\phi \circ \nu$ is the morphism in \cref{eq:rho1}.

Assume  $\frac{n}{k}=[2,\ldots,2,m]$. Now $\Sigma_{n/k}$ is generated by $s_1,\ldots,s_{g-1}$. If $\nu$ is the automorphism of 
$E^g$ given by $\nu(z_1,\ldots,z_g) = (-z_1,z_1-z_2,\ldots,z_{g-1}-z_g)$, then
$\nu \circ s_i = (i,i+1) \circ \nu$ for $i=1,\ldots,g-1$. As before, it follows that $\phi \circ \nu$, which is now given by \cref{eq:rho2}, is a 
quotient for the action of $\Sigma_{n/k}$ on $E^g$. 
\end{proof}

In the previous result, there are other morphisms $\rho:E^g \to S^gE$ that are quotients for the action of $\Sigma_{n/k}$. 
This is because the automorphism $\nu$ such that $\Sigma_{n/k}=\nu^{-1}\Sigma_g \nu$ is not unique. For example,
the automorphism $[-1]:E^g \to E^g$, $\sfz \mapsto -\sfz$, commutes with the action of $\Sigma_g$ so we could replace the $\nu$
in the proof by $[-1]\circ \nu$. Doing this, one finds that the $\rho$ in \cref{eq:rho1} can be replaced by
$\rho(z_1,\ldots,z_g) =(\!(z_1-z_2,z_2-z_3,\ldots,z_{g-1}-z_{g},z_g)\!)$. Similarly for \cref{eq:rho2}.

%%%%%%%%%%%%%%%%%%%%%%%%%%%%%%%%%%%%%%%%%%%%%%%%%%%%%%%%%%%%%%%%%%%%%%%%%%%%%%%%
\subsubsection{Very ampleness}
\label{ssect.very.ample.2}
%%%%%%%%%%%%%%%%%%%%%%%%%%%%%%%%%%%%%%%%%%%%%%%%%%%%%%%%%%%%%%%%%%%%%%%%%%%%%%%%

In \cref{ssect.very.ample.1}, we showed that $D_{n/k}$ is {\it very} ample if and only if all $n_i$ are $\ge 3$. 
This also follows from \cref{prop.E^g/Sigma}: $D_{n/k}$ is very ample if and only if $X_{n/k} \cong E^g$, i.e., if and only if 
$E^g/\Sigma \cong E^g$, i.e., if and only if $\Sigma$ is trivial which happens if and only if all $n_i$ are $\ge 3$.
 
%%%%%%%%%%%%%%%%%%%%%%%%%%%%%%%%%%%%%%%%%%%%%%%%%%%%%%%%%%%%%%%%%%%%%%%%%%%%%%%%
\subsubsection{Special case: $X_{n/k}=\PP^g$}
\label{subsec.sp.char.var.1}
%%%%%%%%%%%%%%%%%%%%%%%%%%%%%%%%%%%%%%%%%%%%%%%%%%%%%%%%%%%%%%%%%%%%%%%%%%%%%%%%
The characteristic variety for $Q_{n,n-1}(E,\tau)$ is isomorphic to $\PP^{n-1}$. To see this, first note that 
$\frac{n}{n-1}=[2,\ldots,2]$ where the number of $2$'s is $g=n-1$; thus $\Sigma=\Sigma_{g+1}$ in this case;
therefore $X_{n/k} \cong E^g/\Sigma_{g+1} \cong \PP^g$ (see also \cref{pr.all2}).

%%%%%%%%%%%%%%%%%%%%%%%%%%%%%%%%%%%%%%%%%%%%%%%%%%%%%%%%%%%%%%%%%%%%
%%%%%%%%%%%%%%%%%%%%%%%%%%%%%%%%%%%%%%%%%%%%%%%%%%%%%%%%%%%%%%%%%%%%
\section{Theta function methods}
\label{sect.theta.several.var}
%%%%%%%%%%%%%%%%%%%%%%%%%%%%%%%%%%%%%%%%%%%%%%%%%%%%%%%%%%%%%%%%%%%%
%%%%%%%%%%%%%%%%%%%%%%%%%%%%%%%%%%%%%%%%%%%%%%%%%%%%%%%%%%%%%%%%%%%%

The study of $Q_{n,k}(E,\tau)$ involves both geometric methods and methods involving theta functions. 
This section focuses on methods involving theta functions.

\subsection{The action of the Heisenberg group $H_n$ on $\Theta_{n/k}(\Lambda)$}
\label{ssect.Hn.action.g.vars}
We will now define an action of $H_n$ on $\Theta_{n/k}(\Lambda)$ and then define a basis for  $\Theta_{n/k}(\Lambda)$ that behaves 
well with respect to the $H_n$-action. 

\begin{proposition}
\label{prop.Theta.N}
  Let $g$ be a positive integer,  $\sfN$ a symmetric $g \times g$ integer-valued matrix, $N_{ii}$ its $i^{\th}$ diagonal entry, and write
   $\sfd=\frac{1}{2}(N_{11},\ldots, N_{gg})$.
Fix a point   $\sfc=(c_1,\ldots,c_g) \in \CC^g$.
If $f:\CC^g \to \CC$ is a holomorphic function such that 
 \begin{equation}
 \label{one-coord-0}
 f(\sfz +\sfe_i \eta) \;=\; e(\sfz \sfN \sfe_i^\sT +c_i) f(\sfz ) \qquad \text{for $i=1,\ldots,g$},
 \end{equation}
 then
 \begin{equation}
 \label{quasi-period}
f(\sfz +\sfm\eta) \;=\;   
 e\big(\sfz \sfN \sfm^\sT  \,+\,   \tfrac{1}{2} \sfm \sfN \sfm^\sT \eta \, +\,  (\sfc   \, - \,  \sfd \eta)\sfm^\sT \big) 
\,
 f(\sfz )
\end{equation}
for all $\sfm=(m_1,\ldots,m_g) \in \ZZ^g$.
\end{proposition}
\begin{pf}
Fix an integer $i$ between $1$ and $g$. 
An induction argument shows that
 \begin{equation}
 \label{one-coord}
 f(\sfz +m\sfe_i \eta) \;=\; e\big(m\sfz \sfN \sfe_i^\sT +\tbinom{m}{2}  N_{ii} \eta +mc_i  \big) f(\sfz )
 \end{equation}
 for all positive integers $m$. 
If we set $\sfm=m\sfe_i$, then (\ref{one-coord}) becomes
 \begin{align*}
  f(\sfz +\sfm \eta) 
  & \;=\; e\big(\sfz \sfN\sfm^\sT+\tfrac{1}{2}\sfm \sfN\sfm^\sT\eta-\tfrac{1}{2}m  N_{ii}  \eta +\sfc \sfm^\sT  \big) f(\sfz ).
 \end{align*}
 Thus, the proposition holds when $\sfm=m\sfe_i$.
 
Let $r$ be an  integer between $1$ and $g$ and let $\sfm=(m_1,\ldots,m_r,0,\ldots,0) \in \ZZ^g$. 
We will prove the proposition by induction on $r$. More explicitly, we will show that
 $$
  f(\sfz +\sfm\eta) 
     \;=\; e\Big(\sfz \sfN \sfm^\sT  \,+\,\tfrac{1}{2} \sfm \sfN \sfm^\sT \eta    \, - \,\tfrac{1}{2} \sum_{i=1}^{r}    m_i  N_{ii}  \eta  
     \, + \,   \sfc \sfm^\sT  \Big)
  f(\sfz ).
  $$
  When $r=g$ this is the formula in (\ref{quasi-period}).

We have already shown that  (\ref{quasi-period}) holds 
when $r=1$. We now assume $r \ge 2$ and 
write $\sfm'=(m_1,\ldots,m_{r-1},0,\ldots,0)$ so that $\sfm=\sfm'+m_r\sfe_r$. 
 
 It follows from (\ref{one-coord}) that 
 \begin{align*}
 & f(\sfz +\sfm\eta) \;=\;    f(\sfz +\sfm'  \eta+m_r\sfe_r \eta) 
 \\
 & \phantom{xx}
   \;=\; e\big(m_r(\sfz +\sfm' \eta)\sfN \sfe_r^\sT  \,+\,  \tbinom{m_{r}}{2} N_{rr} \eta +m_rc_r\big) 
  f(\sfz +\sfm'  \eta)
  \\
   &  \phantom{xx}
    \;=\; e\big(m_r(\sfz +\sfm' \eta)\sfN \sfe_r^\sT  \,+\,  \tbinom{m_{r}}{2} N_{rr}  \eta +m_rc_r\big) 
    \\
  & \phantom{xxxxx} 
  \times e\Big( \sfz \sfN\sfm'^\sT   \,  +\, \tfrac{1}{2}  \sfm'\sfN \sfm'^\sT  \eta      \, - \,  \tfrac{1}{2}    \sum_{i=1}^{r-1}   m_i  N_{ii} \eta  
   +   \sfc \sfm'^\sT  \Big)
  f(\sfz )
      \\
   &  \phantom{xx}
    \;=\; e\Big(\sfz \sfN \sfm^\sT \,+\,   \sfm' \sfN (m_r \sfe_r)^\sT \eta +   \tbinom{m_{r}}{2} N_{rr}  \eta 
 \,  +\,  \tfrac{1}{2}   \sfm'\sfN \sfm'^\sT  \eta     \, - \,  \tfrac{1}{2}    \sum_{i=1}^{r-1}    m_i  N_{ii} \eta  
     \, + \,    \sfc \sfm^\sT  \Big)
  f(\sfz )
    \\
   &  \phantom{xx}
    \;=\; e\Big(\sfz \sfN \sfm^\sT  \,+\,   \sfm' \sfN (m_r \sfe_r)^\sT   \eta    \,  +\,  \tfrac{1}{2}  m_r^2 N_{rr} \eta 
 \,  +\,  \tfrac{1}{2}   \sfm'\sfN \sfm'^\sT  \eta     \, - \,  \tfrac{1}{2}    \sum_{i=1}^{r}    m_i N_{ii} \eta  
     \, + \,   \sfc \sfm^\sT  \Big)
  f(\sfz ).
  \end{align*} 
Since $\sfN$ is symmetric, $\sfe_i\sfN\sfe_j^\sT=\sfe_j\sfN\sfe_i^\sT$ for all $i$ and $j$. 
Therefore
$$
 \sfm' \sfN (m_r \sfe_r)^\sT   =  \tfrac{1}{2}  \sfm' \sfN (m_r \sfe_r)^\sT  \, + \,   \tfrac{1}{2} (m_r \sfe_r) \sfN \sfm'^\sT 
$$
and it follows that 
 $$
  \tfrac{1}{2} \sfm \sfN \sfm^\sT
   \;=\; 
  \sfm' \sfN (m_r \sfe_r)^\sT       \,  +\,  \tfrac{1}{2}  m_r^2  N_{rr} 
 \,  +\,  \tfrac{1}{2}   \sfm'\sfN \sfm'^\sT .
 $$ 
 Therefore
 $$
  f(\sfz +\sfm\eta) 
     \;=\; e\Big(\sfz \sfN \sfm^\sT  \,+\,   \tfrac{1}{2} \sfm \sfN \sfm^\sT    \, - \,  \tfrac{1}{2}    \sum_{i=1}^{r}    m_i  N_{ii}  \eta  
     \, + \,   \sfc \sfm^\sT  \Big)
  f(\sfz ).
  $$
Thus the induction proceeds and the final case $r=g$ gives (\ref{quasi-period}).
  \end{pf}

\begin{proposition}\label{prop.theta.func.per.kl}
Fix a point $(c_1,\ldots,c_g) \in \CC^g$.
Let $z_0=z_{g+1}=0$.
If $f:\CC^g \to \CC$ is a holomorphic function such that
\begin{equation}
\label{hypoth.on.f}
f(z_1,\ldots,z_i+\eta,\ldots,z_g) \;=\; e(z_{i-1}-n_iz_i+z_{i+1}+c_i)f(z_1,\ldots,z_g)
\end{equation}
for all $i=1,\ldots,g$, then 
$$
f(z_1+k_1\eta,\ldots,z_g+k_g\eta)  \;=\;   
 e\Big(   -nz_1  \, +\,   \sum_{i=1}^gk_i(c_i+\eta)  \,-\,   \tfrac{1}{2} (nk-n+k+1) \eta   \Big)   \,   f(\sfz )
$$
and
$$
f(z_1+l_1\eta,\ldots,z_g+l_g\eta)  \;=\;   
 e\Big(   -nz_g  \, +\,   \sum_{i=1}^gl_i(c_i+\eta)  \,-\,   \tfrac{1}{2} (nk'-n+k'+1) \eta   \Big)   \,   f(\sfz ).
$$
\end{proposition}
\begin{proof}
Recall that $k_0=n$, $k_1=k$, $k_g=1$, and $k_{g+1}=0$.

Let $\sfN =-\sD (n_1,\ldots,n_g)$ where $\sD$ is the tri-diagonal matrix in (\ref{defn.D.matrix}).
With this choice of $\sfN$, the hypothesis that $f$ satisfies (\ref{hypoth.on.f}) can be written as
$$
 f(\sfz +\sfe_i \eta) \;=\; e(\sfz \sfN \sfe_i^\sT +c_i) f(\sfz )
 $$ 
 which is condition (\ref{one-coord-0}) appearing in \cref{prop.Theta.N}. We now apply that result with 
$\sfk=(k_1,\ldots,k_g)$ playing the role that $\sfm$ played there.

Since
\begin{align*}
 \sfN\sfk^\sT   &  \;=\; (-n_1k_1+k_2\,,k_1-n_2k_2+k_3,\ldots,k_{g-2}-n_{g-1}k_{g-1}+k_g,\, k_{g-1}-n_gk_g)^\sT
 \\
  &  \;=\; (-k_0,0,\ldots, 0)^\sT
   \\
  &  \;=\; (-n,0,\ldots, 0)^\sT,
\end{align*}
the formula in (\ref{quasi-period}) becomes
\begin{align*}
f(\sfz +{\sfk}\eta)
&  \;=\;   
 e\Big(   -nz_1  \,-\,   \tfrac{1}{2} k_1n \eta \, +\,   \sum_{i=1}^gk_i(c_i+     \tfrac{1}{2} n_i \eta) \Big)   \,   f(\sfz )
 \\
 &  \;=\;   
 e\Big(   -nz_1  \, +\,   \sum_{i=1}^gk_ic_i  \,-\,   \tfrac{1}{2} nk \eta  \,+\,     \tfrac{1}{2}  \sum_{i=1}^gk_i n_i \eta  \Big)   \,   f(\sfz )
  \\
 &  \;=\;   
 e\Big(   -nz_1  \, +\,   \sum_{i=1}^gk_i(c_i+\eta)  \,-\,   \tfrac{1}{2} (nk-n+k+1) \eta   \Big)   \,   f(\sfz )
\end{align*}
where the last equality follows from the calculation
\begin{align*}
\,-\,   \tfrac{1}{2} nk   \,+\,     \tfrac{1}{2}  \sum_{i=1}^gk_i n_i 
& \;=\; \,-\,   \tfrac{1}{2} nk   \,+\,     \tfrac{1}{2}  \sum_{i=1}^g(k_{i-1}+k_{i+1})
\\
 & \;=\; \,-\,   \tfrac{1}{2} nk   \,+\,     \tfrac{1}{2}(k_0-k_1-k_g) +  \sum_{i=1}^g k_i.
\end{align*}
This completes the proof of the first quasi-periodicity property.

The second follows from the first by replacing $n/k$ by $n/k'$, replacing $(k_1,\ldots,k_g)$ by $(l_g,\ldots,l_1)$, and 
replacing $(z_1,\ldots,z_g)$ by $(z_g,\ldots,z_1)$. 
\end{proof}

In \cite[Appendix~B]{Od-survey}, Odesskii defined two operators,  that we denote by $S$ and $T$,  on the ring of holomorphic functions on 
$\CC^g$.  We also  define operators $S'$ and $T'$ that are related to $n/k'$ in the same way as $S$ and $T$ are related to $n/k$.

\begin{definition}
\label{4308528}
	Define linear operators $S$\index{S@$S$}, $T$\index{T@$T$}, $S'$\index{S'@$S'$}, and $T'$\index{T'@$T'$}, on the space of holomorphic functions on $\CC^{g}$ by
	\begin{align*}
		(S\cdot f)(z_{1},\ldots,z_{g})&\;=\; f(z_{1}+\tfrac{k_{1}}{n},\ldots,z_{g}+\tfrac{k_{g}}{n}\big),
		\\
		(T \cdot f)(z_{1},\ldots,z_{g})&\;=\; e(z_{1}+C)f\big(z_{1}+\tfrac{k_{1}}{n}\eta,\ldots,z_{g}+\tfrac{k_{g}}{n}\eta\big),
		\\
		(S' \cdot f)(z_{1},\ldots,z_{g})&\;=\; f\big(z_{1}+\tfrac{l_1}{n},\ldots,z_{g}+\tfrac{l_g}{n} \big),
		\\
		(T' \cdot f)(z_{1},\ldots,z_{g})&\;=\; e\big(z_{g}+C')f\big(z_{1}+\tfrac{l_1}{n}\eta,\ldots,z_{g}+
		\tfrac{l_g}{n}\eta\big),
	\end{align*}
	where\index{C@$C$}\index{C'@$C'$}
	\begin{align}
		 C & \;=\; -\, \tfrac{1}{n} \left(\sum_{i=1}^{g}k_{i}(c_{i}+\eta)+\left(\tfrac{n-1}{2}-k\right)\eta\right),  \label{defn.scalar.varphi}
		 \\
		C'  & \;=\; -\,\tfrac{1}{n}  \left(\sum_{i=1}^{g}l_{i}(c_{i}+\eta)+\left(\tfrac{n-1}{2} -k'\right)\eta\right). \label{defn.scalar.varphi'}
	\end{align}
\end{definition}

\begin{proposition}\leavevmode
	\begin{enumerate}
		\item\label{2049852} 
		The four operators satisfy the following relations:
		\begin{align*}
			&ST \;=\; e\big(\tfrac{k}{n}\big)  TS, &
			& S'T' \;=\; e\big(\tfrac{k'}{n}\big)      T'S',&
			&   SS' \;=\; S'S  ,\\
			&ST' \;=\; e\big(\tfrac{1}{n}\big)  T'S,&
			&S'T \;=\; e\big(\tfrac{1}{n}\big)  TS' ,&
			&T T' \;=\; T'T.
		\end{align*}
		\item\label{8492848} 
		The operators $S,T,S',T'$  send  $\Theta_{n/k}(\Lambda)$ to itself.
		\item\label{0089507} 
		$S^n=T^n=S'^n=T'^n=1$ on $\Theta_{n/k}(\Lambda)$.
	\end{enumerate}
\end{proposition}

\begin{proof}
	\cref{2049852}
	It is clear that $SS'=S'S$. 
	
	To see that  $TT'=T'T$  we first compute
	\begin{align*}
		&(TT' f)(z_{1},\ldots,z_{g})\\
		&  \;=\;  e(z_{1}+C)(T'f)\big(z_{1}+\tfrac{k_1}{n}\eta,\ldots,z_{g}+\tfrac{k_g}{n}\eta\big)
		    \\
		&\;=\; e(z_{1}+C)e\big(z_{g}+\tfrac{k_g}{n} \eta+C'\big)
		f\big(z_{1}+\tfrac{k_1+l_1}{n}\eta,\ldots,z_{g}+\tfrac{k_g+l_g}{n}\eta\big)
	\end{align*}
	Interchanging the roles of $k_{i}$ and $l_{i}$ in the previous calculation gives a similar expression for $T'T$. 
	Comparing the two calculations, we see that $TT'f=T'Tf$ since $k_{g}=1=l_{1}$.
	
	The following calculation shows that  $ST'=e(\frac{1}{n})T'S$:
	\begin{align*}
		(ST'f)(z_{1},\ldots,z_{g})
		&=(T'f)\big(z_{1}+ \tfrac{k_{1}}{n},\ldots,z_{g}+\tfrac{k_{g}}{n}\big )\\
		&=e\big(z_{g}+\tfrac{k_{g}}{n}+C' \big)
		f(z_{1}+\tfrac{k_{1}}{n}+\tfrac{l_{1}}{n}\eta,\ldots,z_{g}+\tfrac{k_{g}}{n}+\tfrac{l_{g}}{n}\eta\big)\\
		&=e\big(\tfrac{1}{n} \big)
		e\big(z_{g}+C'\big)
		f\big(z_{1}+\tfrac{k_{1}}{n}+\tfrac{l_{1}}{n}\eta,\ldots,z_{g}+\tfrac{k_{g}}{n}+
		\tfrac{l_{g}}{n} \eta\big)
		\\
		&=e\big(\tfrac{1}{n} \big)
		e\big(z_{g}+C'\big)
		(Sf)(z_{1}+\tfrac{l_{1}}{n} \eta,\ldots,z_{g}+\tfrac{l_{g}}{n} \eta\big)\\
		&=e\big(\tfrac{1}{n} \big)  (T'Sf)(z_{1},\ldots,z_{g} ).
	\end{align*}
	
	Similar calculations prove the other three equalities in \cref{2049852}.
	
	\cref{8492848} 
	We only show that for all $f\in\Theta_{n/k}(\Lambda)$ the function $Sf$ has the appropriate quasi-periodicity property 
	with respect to $\eta$ in each variable: 
	\begin{align*}
		&(Sf)(z_{1},\ldots,z_{i}+\eta,\ldots,z_{g})\\
		&=f\big(z_{1}+\tfrac{k_1}{n},\ldots,z_{i}+\tfrac{k_i}{n}+\eta,\ldots,z_{g}+\tfrac{k_g}{n}\big)
		\\
		&=e\big(-n_{i}(z_{i}+ \tfrac{k_i}{n})+(1-\delta_{i,1})(z_{i-1}+\tfrac{k_{i-1}}{n})+(1-\delta_{i,g})(z_{i+1}+\tfrac{k_{i+1}}{n}\big)+c_{i}\big)f\big(z_{1}+\tfrac{k_1}{n},\ldots,z_{g}+\tfrac{k_g}{n} \big)\\
		&=e\big(-\tfrac{k_{0}}{n}\delta_{i,1}-\tfrac{k_{g+1}}{n}\delta_{i,g}\big)e(-n_{i}z_{i}+(1-\delta_{i,1})z_{i-1}+(1-\delta_{i,g})z_{i+1}+c_{i})  
		f\big(z_{1}+\tfrac{k_1}{n},\ldots,z_{i}+\tfrac{k_i}{n},\ldots,z_{g}+\tfrac{k_g}{n}\big)  
		\\
		&=e(-n_{i}z_{i}+z_{i-1}+z_{i+1}+c_{i})(Sf)(z_{1},\ldots,z_{g})
	\end{align*}
	where the last equality follows because $k_{0}=n$ and $k_{g+1}=0$.
	
	\cref{0089507} 
	Since $f$ is periodic with respect to $1$ in each variable, $S^{n}=S'^{n}=1$. We will now show that $T^{n}=1$. First,
	\begin{align*}
		&(T^{n}f)(z_{1},\ldots,z_{g})\\
		&=e(z_{1}+C)
		(T^{n-1}f)\big(z_{1}+\tfrac{k_1}{n} \eta,\ldots,z_{g}+\tfrac{k_g}{n}\eta\big)
		\\
		&=e(z_{1}+C)
		e\big(z_{1}+\tfrac{k_{1}}{n}\eta+C\big)\\
		&\phantom{=}\times\cdots\times e\big(z_{1}+(n-1)\tfrac{k_{1}}{n}+C\big) f(z_{1}+k_{1}\eta,\ldots,z_{g}+k_{g}\eta)
		\\ 
		&=e\big(nz_{1}+\tfrac{n-1}{2}k_{1}\eta+nC\big)
		f(z_{1}+k_{1}\eta,\ldots,z_{g}+k_{g}\eta).
	\end{align*}
	By applying \cref{prop.theta.func.per.kl}, this is equal to
	\begin{align*}
		&e\big(nz_{1}+\tfrac{n-1}{2}k_{1}\eta+nC\big) 
		\\
		&\phantom{=}\times e\bigg(-nz_{1}+\sum_{i=1}^{g}k_{i}(c_{i}+\eta)-\tfrac{1}{2}(nk-n+k+1) \eta\bigg)f(z_{1},\ldots,z_{g})
		\\
		&=e\bigg(nC+\sum_{i=1}^{g}k_{i}(c_{i}+\eta)+(\tfrac{n-1}{2}-k)\eta\bigg)f(z_{1},\ldots,z_{g})\\
		&=f(z_{1},\ldots,z_{g}).
	\end{align*}
	A similar argument shows that $T'^{n}=1$.
\end{proof}

\subsubsection{The basis $w_0,\ldots,w_{n-1}$ for $\Theta_{n/k}(\Lambda)$}
\label{sssect.wi.basis}
We now exhibit a basis  for $\Theta_{n/k}(\Lambda)$ consisting of $S$-eigenvectors. 

The actions of $S$ and $T$ on $\Theta_{n/k}(\Lambda)$ make it a representation of the Heisenberg group $H_n$ of order $n^3$. The generator
\begin{equation*}
  [S,T] = STS^{-1}T^{-1}
\end{equation*}
of the center of $H_n$ acts as the primitive root of unity $e\big(\frac{k}{n}\big)$, so it follows from the classification of irreducible $H_n$-representations (e.g., \cite[\S 3.1]{schulte}) that $\Theta_{n/k}(\Lambda)$ is the unique irreducible $H_n$-representation with this property. Furthermore, it has a basis $\{w_{\a}\;|\;\a\in\ZZ_n\}$\index{w_alpha(z_1,...,z_g)@$w_{\a}(z_{1},\ldots,z_{g})$}, unique up to multiplication by a common non-zero scalar, such that
\begin{equation*}
	Sw_{\a}\;=\; e\big(\tfrac{k\a}{n}\big)w_{\a}   \qquad \text{and} \qquad  Tw_{\a}\;=\;w_{\a+1}.
\end{equation*}

\begin{proposition}\label{02987452}
	There exist $c_{\frac{1}{n}},c_{\frac{1}{n}\eta}\in\CC^{\times}$ such that for all $\a\in\ZZ_n$,
	\begin{equation*}
		S'w_{\a}\; =\; c_{\frac{1}{n}}e\big(\tfrac{\a}{n}\big)w_{\a}
		\qquad \text{and} \qquad 
		T'w_{\a}\;=\;c_{\frac{1}{n}\eta}w_{\a+k'}
	\end{equation*}
	where $k'$ denotes the inverse of $k$ in $\ZZ_n$. 
\end{proposition}

\begin{proof}
	Since $SS'w_{0}  =S'Sw_{0}=S'w_{0}$,
	$S'w_{0}=c_{\frac{1}{n}}w_{0}$ for some $c_{\frac{1}{n}}\in\CC^{\times}$. For each $\alpha$,
	\begin{equation*}
		S'w_{\a}\;=\; S'T^{\a}w_{0}
		\;=\; e\big(\tfrac{\a}{n}\big)T^{\a}S'w_{0}
		\;=\; e\big(\tfrac{\a}{n}\big)T^{\a}\big(c_{\frac{1}{n}}w_{0}\big)\;=\; c_{\frac{1}{n}}e\big(\tfrac{\a}{n}\big)w_{\a}.
	\end{equation*}
	Therefore $ST'w_{0}= e\big(\tfrac{1}{n}\big)T'Sw_{0}	= e\big(\tfrac{kk'}{n}\big)T' w_{0}$
	implies $T'w_{0}=c_{\frac{1}{n}\eta}w_{k'}$ for some $c_{\frac{1}{n}\eta}\in\CC^{\times}$. For each $\a$,
	\begin{equation*}
		T'w_{\a}\;=\; T'T^{\a}w_{0}\;=\; T^{\a}T'w_{0}\;=\; T^{\a}\big(c_{\frac{1}{n}\eta}w_{k'}\big)\;=\; c_{\frac{1}{n}\eta}w_{\a+k'}.
	\end{equation*}
	The proof is complete.
\end{proof}

%%%%%%%%%%%%%%%%%%%%%%%%%%%%%%%%%%%%%%%%%%%%%%%%%%%%%%%%%%%%%%%%%%%%%%%%%%%%%%%%
\subsection{The identification $\Theta_{n/k}(\L)=H^0(E^g,\cL_{n/k})$}
\label{sect.morphism.Phi.n.k}
%%%%%%%%%%%%%%%%%%%%%%%%%%%%%%%%%%%%%%%%%%%%%%%%%%%%%%%%%%%%%%%%%%%%%%%%%%%%%%%%

In this subsection we fix a point 
\begin{equation}\label{eq.good.c}
	\sfc\;=\;(c_{1},\ldots,c_{g})\; \in \;\tfrac{1}{2}(n_{1},\ldots,n_{g})+\Lambda^g.
\end{equation}
and define  $\Theta_{n/k}(\Lambda)$ as in \Cref{sect.Theta.n/k} using that choice of $\sfc$. 
The $\Theta_{n/k}(\Lambda)$  defined at \cite[p.~1152]{Od-survey} is our $\Theta_{n/k}(\Lambda)$ with  
$\sfc=\tfrac{1}{2}(n_{1},\ldots,n_{g})- (0,1,\ldots,1)\eta$. Let $m_{1},\ldots,m_{g}$ be the unique integers such that
\begin{equation*}
\sfc\in\tfrac{1}{2}(n_{1},\ldots,n_{g})+(m_{1},\ldots,m_{g})\eta+\ZZ^{g}.
\end{equation*}

For a complex algebraic variety $X=(X,\cO)$, there is a corresponding analytic space $X^{\mathrm{an}}=(X^{\mathrm{an}},\cO^{\mathrm{an}})$ called the \textsf{analytification} of $X$, and a canonical morphism $\lambda\colon X^{\mathrm{an}}\to X$ of 
ringed spaces. 
There is an exact functor $(-)^{\mathrm{an}}$ 
from the category of (algebraic) coherent $\cO$-modules to the category of coherent analytic $\cO^{\mathrm{an}}$-modules and
Serre's GAGA theorem says this is an equivalence when $X$ is a complex projective algebraic variety.

We now apply this to $X=E^{g}$. Its analytification $X^{\mathrm{an}}$ is simply $E^{g}$ regarded as a complex manifold in the usual way. 
Define $\cL$ to be the invertible analytic $\cO^{\mathrm{an}}$-module whose sections on an analytic open subset $P\subseteq X^{\mathrm{an}}$ are the holomorphic functions on $\pi^{-1}(P)\subseteq\bbC^{g}$ satisfying the quasi-periodicity properties \cref{eq:od-qp}, where $\pi:\CC^{g}\to E^{g}$ is the quotient map. Thus $H^{0}(E^{g},\cL)=\Theta_{n/k}(\Lambda)$. By GAGA, there is an  {\it algebraic} invertible 
$\cO$-module $\cL'$, 
unique up to isomorphism, such that $(\cL')^{\mathrm{an}}\cong\cL$. For open algebraic  subsets $U\subseteq X$, the maps $\cL'(U) \to \cL(U)$ are injective, and the images form an algebraic coherent sheaf $\cL^{\mathrm{alg}}$ that is isomorphic to $\cL'$ and does not depend on the choice of $\cL'$. Moreover, we obtain a canonical isomorphism $H^{0}(X,\cL^{\mathrm{alg}})\to H^{0}(X^{\mathrm{an}},\cL)=\Theta_{n/k}(\Lambda)$.  The next result relates $H^{0}(X,\cL^{\mathrm{alg}})$ to the divisor $D_{n/k}$.

\begin{lemma}\label{prop.sheaf.theta.func}
Let $h\colon\bbC^{g} \to \CC$ be the meromorphic function 
		\begin{equation*}
			h(z_{1},\ldots,z_{g}) \; :=\; \prod_{i=1}^{g}e(m_{i}z_{i})\theta(z_{i})^{n_{i}}\cdot\prod_{j=1}^{g-1}\frac{e(z_{j+1})\theta(z_{j}-z_{j+1})}{\theta(z_{j})\theta(z_{j+1})}.
		\end{equation*}
	\begin{enumerate}
		\item\label{item.prop.sheaf.theta.func.h} 
		The function $h$ belongs to $H^{0}(E^{g},\cL)=\Theta_{n/k}(\Lambda)$.
		\item\label{item.prop.sheaf.theta.func.isom} 
		If we identify $h$ with the corresponding element of $H^{0}(X,\cL^{\mathrm{alg}})$ via the canonical isomorphism, then $(h)_{0}=D_{n/k}$, and thus there is a unique isomorphism $\cO_{E^{g}}(D_{n/k})\to\cL^{\mathrm{alg}}$ that maps $1$ to $h$.
	\end{enumerate}
\end{lemma}

\begin{proof}
	\cref{item.prop.sheaf.theta.func.h} follows from the quasi-periodicity properties of $\theta$ (see \cref{sect.theta.fns.1.var}).
	
	\cref{item.prop.sheaf.theta.func.isom} Since the only zeros of $\theta$ are at the points of $\L$ and they all have order one,  
	the divisor of zeros $(h)_{0}$ is $D_{n/k}$. 
	The argument used in the proof of \cref{le.act-scale} completes the proof.
\end{proof}

Another argument in the proof of \cref{le.act-scale} shows that every isomorphism $\cO_{E^{g}}(D_{n/k})\to\cL^{\mathrm{alg}}$ is a non-zero scalar multiple of the one in \cref{prop.sheaf.theta.func}\cref{item.prop.sheaf.theta.func.isom}.
Thus we have a vector space isomorphism
\begin{equation*}
	H^{0}(E^{g},\cL_{n/k})\;=\;H^{0}(E^{g},\cO_{E^{g}}(D_{n/k}))\;\cong\;H^{0}(E^{g},\cL^{\mathrm{alg}})\;\cong\;H^{0}(E^{g},\cL)\;=\;\Theta_{n/k}(\Lambda)
\end{equation*}
that is canonical up to non-zero scalar multiple. Hence there is a canonical isomorphism (no scalar multiples required)
	$\bbP(H^{0}(E^{g},\cL_{n/k})^{*})  \cong  \bbP(\Theta_{n/k}(\Lambda)^{*})$
which we will treat  as an identification
$$
\bbP(H^{0}(E^{g},\cL_{n/k})^{*}) \; = \; \bbP(\Theta_{n/k}(\Lambda)^{*}).
$$

\begin{remark}\label{rem.good.c}
	In order to obtain a global section of $\cL^{\mathrm{alg}}$ (or equivalently $\cL$) whose divisor of zeros is $D_{n/k}$, the constant $\sfc$ has to be of the form \cref{eq.good.c}. To see this, take $h$ as in \cref{prop.sheaf.theta.func} for $\sfc$ satisfying \cref{eq.good.c}, and let $h'$ be a global section for an arbitrary $\sfc'=(c'_{1},\ldots,c'_{g})\in\CC^{g}$ whose divisor of zeros is $D_{n/k}$. Since $(h)_{0}=(h')_{0}$, the function $f:=h'/h$ is holomorphic on $\CC^{g}$ and satisfies
	\begin{equation*}
		\begin{cases}
			f(z_{1},\ldots,z_{i}+1,\ldots,z_{g})\;=\;f(z_{1},\ldots,z_{g}),\\
			f(z_{1},\ldots,z_{i}+\eta,\ldots,z_{g})\;=\;e(c'_{i}-c_{i})f(z_{1},\ldots,z_{g})
		\end{cases}
	\end{equation*}
	Thus the argument in the proof of \cref{prop.dim.theta.sp} shows that $c'_{i}-c_{i}$ has to be an element of $\L$ for all $i$. Therefore $\sfc'$ is also of the form \cref{eq.good.c}.
\end{remark}

%%%%%%%%%%%%%%%%%%%%%%%%%%%%%%%%%%%%%%%%%%%%%%%%%%%%%%%%%%%%%%%%%%%
\subsection{Remarks on the degree-one component of $Q_{n,k}(E,\tau)$}
\label{ssect.Q_1}
%%%%%%%%%%%%%%%%%%%%%%%%%%%%%%%%%%%%%%%%%%%%%%%%%%%%%%%%%%%%%%%%%%%

Although  $Q_{n,k}(E,\tau)$ is defined as a quotient of the tensor algebra on an anonymous $n$-dimensional 
vector space, $V$, it is useful to give $V$ various concrete interpretations:
\begin{enumerate}[label=(\alph*)]
  \item 
  as the space $\Theta_n(\L)$ of theta functions in one variable;
  \item 
  as the space $\Theta_{n/k}(\L)$ of theta functions in $g$ variables;
  \item 
  as the space $H^0(E^g,\cL_{n/k})$.
\end{enumerate}
We now make some remarks on how these interpretations of $V$ are used.

(1)
The quadratic relations for $Q_{n,k}(E,\tau)$ are expressed in terms an anonymous basis $\{x_0,\ldots,x_{n-1}\}$ for $V$. 
When we take $V$ to be 
$\Theta_n(\L)$ we make the identification $x_\a=\theta_\a$ where the $\theta_\a$'s are the functions in \Cref{official.theta_alpha}. 
We used this identification in \cite[\S3.2]{CKS1}:  there is an action (related to the translation action of $\frac{1}{n}\L$ on $\CC$)
of the Heisenberg group $H_n$ as automorphisms of the field of meromorphic functions $\CC \to \CC$ and 
$\Theta_n(\L)$ becomes an irreducible representation of $H_n$ with respect to that action \cite[\S2.3]{CKS1}; one transfers this 
action of $H_n$ to $V$ by making the identification $x_\a=\theta_\a$; 
the space of quadratic relations for $Q_{n,k}(E,\tau)$ is then seen to be stable under the action
of $H_n$, whence $H_n$ acts as algebra automorphisms of $Q_{n,k}(E,\tau)$. 
The parameter $\tau$ plays no role in this action of $H_n$ on $\Theta_n(\L)$ or $V$; neither does the integer $k$.

(2)
In \Cref{ssect.Hn.action.g.vars} we showed that $\Theta_{n/k}(\L)$ 
is stable under an action of $H_n$ on the space of meromorphic functions $\CC^g \to \CC^g$. 
(That action was defined in terms of the point $\sfk=(k_1,\ldots,k_g)\in \ZZ^g$ and the translation action of $\frac{1}{n}\L^g$ on $\CC^g$.)
The $\theta_\a$'s in $\Theta_n(\L)$ transformed in a nice way with respect to the standard generators for $H_n$. 
A basis $w_0,\ldots,w_{n-1}$ for $\Theta_{n/k}(\L)$ having similar transformation properties was defined in 
\Cref{sssect.wi.basis}. We identify $V$ with $\Theta_{n/k}(\L)$ by setting $x_\a=w_\a$. 
The parameter $\tau$ plays no role in this. 

(3) 
\Cref{prop.sheaf.theta.func} provides a canonical identification of $\Theta_{n/k}(\L)$ 
with $H^0(E^g,\cL_{n/k})$ (up to a non-zero scalar multiple).  Under this identification the $w_\a$'s 
become a basis for $H^0(E,\cL_{n/k})$ and  the map $\Phi_{n/k}:E^g \to \PP(H^0(E^g,\cL_{n/k})^*)$
can now be written as 
\begin{equation}
\label{eq:Phi.nk.in.terms.of.theta.fns}
\Phi_{n/k}(\sfz) \; := \; (w_0(\sfz),\ldots,w_{n-1}(\sfz)), \qquad \sfz \in E^g.
\end{equation}

(4)
Combining the identifications in remarks (2) and (3) leads to an identification $V=H^0(E^g,\cL_{n/k})$. 
Since $X_{n/k}$ is a subvariety of $\PP(H^0(E^g,\cL_{n/k})^*)$, we obtain $X_{n/k} \subseteq \PP(V^*)$.
In this way, $V$ is realized as linear forms on $X_{n/k}$. Thus, if $x$ is a non-zero homogeneous element of degree one in 
$Q_{n/k}(E,\tau)$ we can speak of its divisor of zeros on $X_{n/k}$. In a similar way, $V \otimes V$ consists of bilinear forms
on $\PP(V^*) \times \PP(V^*)$ so we can speak of the vanishing locus of an element in $V \otimes V$ on $X_{n/k} \times 
X_{n/k}$. \Cref{43258924} uses the explicit description of $\Phi_{n/k}$ in \cref{eq:Phi.nk.in.terms.of.theta.fns} 
and Odesskii's identity in \Cref{09847525} to show that
the quadratic relations for $Q_{n,k}(E,\tau)$ vanish on the graph of the automorphism $\s:X_{n/k} \to X_{n/k}$
defined in \cref{sect.aut.sigma}.

(5)
Make the identifications $V=\Theta_n(\L)$ and $x_\a=\theta_\a$ as in remark (1). 
There is an embedding $E \to \PP(\Theta_n(\L)^*)$ given by 
\begin{equation}
\label{eq:embedding.E.via.theta.fns}
z  \; \mapsto \; (\theta_0(z),\ldots,\theta_{n-1}(z)).
\end{equation}
The action of $H_n$ on $\Theta_n(\Lambda)$ induces an action
of $H_n$ on $\PP^{n-1}=\PP(\Theta_n(\Lambda)^*)$; it follows from the action of $H_n$ on the basis $\{\theta_i \; | \; i \in \ZZ_n\}$
that  the image of $E$ in $\PP(\Theta_n(\L)^*)$ is  stable under this action (see \cref{lem.Hn.action.1}).

(6)
Similarly, the action of $H_n$ on $\Theta_{n/k}(\Lambda)$ induces an action of $H_n$ on 
$\PP^{n-1}=\PP(\Theta_n(\Lambda)^*)$ and $X_{n/k}$ is stable 
under this action (see \cref{sssect.wi.basis}).

(7)
Let $\pr_i:E^g \to E$ be the projection $\pr_i(z_1,\ldots,z_g)=z_i$. Since $\mathbf{R}\pr_{1*} \cL_{n/k}=\Phi(\cO_E)$ 
(where $\Phi$ is the functor defined in \cref{prop.FM.transform}),  $\cV_{n,k}:=\pr_{1*} \cL_{n/k}$ is a locally free indecomposable  $\cO_E$-module of rank $k$ and degree $n$. 
The natural map $H^0(E^g,\cL_{n/k}) \to H^0(E, \pr_{1*} \cL_{n/k})$ is an isomorphism so the identification of $Q_{n,k}(E,\tau)_1$
with $H^0(E^g,\cL_{n/k})$ in (4)  leads to an  identification of $Q_{n,k}(E,\tau)_1$ with $H^0(E,\cV_{n,k})$.

There is a result like \cref{prop.FM.transform} where one first pulls back via $\pr_{1}^*$ then pushes forward via 
$\pr_{g*}$; such a result says, in particular, that $\pr_{g*}\cL_{n/k}$ is an indecomposable locally free  $\cO_E$-module 
$\cV_{n,k'}$ having rank $k'$ and degree $n$. (See \cite[Prop.~8.2]{CKS3} for another proof.)
Thus $Q_{n,k}(E,\tau)_1$ also has an interpretation as $H^0(E,\cV_{n,k'})$.

Serre duality provides a natural isomorphism $H^0(E,\cV_{n,k}) \cong \Ext^1(\cV_{n,k},\cO_E)^*$ so there are also natural 
identifications of $Q_{n,k}(E,\tau)_1$ with $ \Ext^1(\cV_{n,k},\cO_E)^*$ and $ \Ext^1(\cV_{n,k'},\cO_E)^*$.

%%%%%%%%%%%%%%%%%%%%%%%%%%%%%%%%%%%%%%%%%%%%%%%%%%%%%%%%%%%%%%%%%%%
%%%%%%%%%%%%%%%%%%%%%%%%%%%%%%%%%%%%%%%%%%%%%%%%%%%%%%%%%%%%%%%%%%%
\subsection{The vanishing locus of the relations for $Q_{n,k}(E,\tau)$}
%%%%%%%%%%%%%%%%%%%%%%%%%%%%%%%%%%%%%%%%%%%%%%%%%%%%%%%%%%%%%%%%%%%
%%%%%%%%%%%%%%%%%%%%%%%%%%%%%%%%%%%%%%%%%%%%%%%%%%%%%%%%%%%%%%%%%%%

Adopting the point of view in \Cref{ssect.Q_1}(4), we will now show that the common zero locus of the relations for 
$Q_{n,k}(E,\tau)$ contains the graph of the automorphism $\s:X_{n/k} \to X_{n/k}$.  
To do this we use a variation of an identity that appears in Odesskii's survey. 
In \Cref{sec.prf.th.id} we give a proof of it that follows  Odesskii's with a little more detail.

In this subsection the points $\sfk=(k_1,\ldots,k_g)$ and $\sfl=(l_1,\ldots,l_g)$ in $\ZZ^g$ are as in \cref{sect.aut.sigma}, 
and the functions  $\theta\in\Theta_{1}(\Lambda)$, $\theta_{\alpha}\in\Theta_{n}(\Lambda)$, and 
$w_{\alpha}\in\Theta_{n/k}(\Lambda)$ are as before, but $\sfc\in\CC^{g}$ is arbitrary here.

\begin{theorem}
{\cite[p.~1153]{Od-survey}}
\label{09847525}
Let $\sfy= (y_{1},\ldots,y_{g})$ and $\sfz=(z_{1},\ldots,z_{g})$ be points in $\CC^g$. 
	For all $\alpha,\beta\in\bbZ_n$ and all $(u,v)\in\bbC^2$, the following equality  holds whenever all denominators  are non-zero:
	\begin{align*}
		\frac{\theta(-nu+y_{1}-z_{1})}{\theta(-nu)\theta(y_{1}-z_{1})} &  w_{\alpha}(\sfy +\sfk u)w_{\beta}(\sfz +\sfl v)\\
		&+\sum_{t=1}^{g-1}\frac{\theta(z_{t}-y_{t}+y_{t+1}-z_{t+1})}{\theta(z_{t}-y_{t})\theta(y_{t+1}-z_{t+1})}
		w_{\alpha}(z_{1}+k_{1}u,\ldots,z_{t}+k_{t}u,y_{t+1}+k_{t+1}u,\ldots,y_{g}+k_{g}u)\\
		&\qquad\times w_{\beta}(y_{1}+l_{1}v,\ldots,y_{t}+l_{t}v,z_{t+1}+l_{t+1}v,\ldots,z_{g}+l_{g}v)\\
		\phantom{xxx} + \; \frac{\theta(z_{g}-y_{g}+nv)}{\theta(z_{g}-y_{g})\theta(nv)} & w_{\alpha}(\sfz +\sfk u)w_{\beta}(\sfy +\sfl v)\\
		\; =\; \tfrac{1}{n}\theta(\tfrac{1}{n})\ldots\theta(\tfrac{n-1}{n}) &
	     \sum_{r\in\ZZ_n}\frac{\theta_{\beta-\alpha+r(k-1)}(-u+v)}{\theta_{\beta-\alpha-r}(-u)\theta_{rk}(v)}w_{\beta-r}(\sfy)w_{\alpha+r}(\sfz + \sfk u+\sfl v).
	\end{align*}
\end{theorem}

The identity in \cite[p.~1153]{Od-survey} can be obtained from \cref{09847525} by making the following substitutions:
replace our $\tau$ by $\eta$; 
replace our $k_i$ by $m_i$;\footnote{In \cite[p.~1153]{Od-survey}, $m_{\alpha}$ was defined to be $d(n_{\alpha+1},\ldots,n_{g})\eta$ 
but it should have been $d(n_{\alpha+1},\ldots,n_{g})=k_\a$.}
 replace our $u$ by $u-v$;  replace our $v$ by $\eta$;  replace our $y_i$ by $y_{i}+m_{i}v$;
replace our $z_i$ by  $z_{i}+m_{i}v$. Note, too, that the $l_i$ in  \cite[p.~1153]{Od-survey}, which is defined to be 
$d(n_1,\ldots,n_{i-1})\eta$, equals our $l_i \tau=d(n_{i-1}, \ldots, n_1)\tau$ (see the remark in \cref{sect.negative.contd.fracs}
explaining why $d(a,b,\ldots,c)=d(c,\ldots,b,a)$).

\begin{corollary}\label{43258924}
\label{cor.relns.vanish.on.graph}
Let $\sfk$, $\sfl$, $\sfn$, and $\s$ be as in \cref{sect.aut.sigma}.
	For all $\alpha,\beta\in\ZZ_n$, all $\sfy=(y_{1},\ldots,y_{g})\in\bbC^g$, and all $\tau\in\CC-\frac{1}{n}\Lambda$,
	\begin{equation}\label{eq.id.rel.w}
		\sum_{r\in\ZZ_n}  \frac{\theta_{\beta-\alpha+r(k-1)}(0)}{\theta_{\beta-\alpha-r}(-\tau) \theta_{rk}(\tau)} 
		\, w_{\beta-r}(\sfy)\,w_{\alpha+r}(\s(\sfy)) 
		\;=\; 0.
	\end{equation}
The relations for $Q_{n,k}(E,\tau)$ therefore vanish on the graph of the automorphism $\s:X_{n/k} \to X_{n/k}$.	
\end{corollary}
\begin{proof}
Since $\s(\sfy)=\sfy+(\sfk+\sfl-\sfn)\tau$, \cref{eq.id.rel.w} follows from the identity in \cref{09847525} when we set $u=v=\tau$ and $z_{i}=y_{i}-n\tau$. Since $X_{n/k}$ is the image of the composition $\CC^g \to E^g \to \PP(V^*)$, the second assertion follows.
\end{proof}

\begin{corollary}
\cite[(30), p.~1151]{Od-survey}
\label{43258934}
For  all $\alpha,\beta\in\ZZ_n$ and all $y,z \in \CC$,
\begin{align}
	&\frac{\theta(-n\tau+y-z)}{\theta(-n\tau)\theta(y-z)} \, \big(\theta_{\alpha}(y+\tau)\theta_{\beta}(z+\tau)-\theta_{\alpha}(z+\tau) \theta_{\beta}(y+\tau)\big)
	\label{42905872}\\
	&\; =\; \tfrac{1}{n}\theta(\tfrac{1}{n})\ldots\theta(\tfrac{n-1}{n}) \sum_{r\in\ZZ_n}\frac{\theta_{\beta-\alpha}(0)}{\theta_{\beta-\alpha-r}(-\tau)\theta_{r}(\tau)}\theta_{\beta-r}(y)\theta_{\alpha+r}(z+2\tau).
	\label{42905873}
\end{align}
\end{corollary}
\begin{proof}
Let $k=1$. Then  $g=1$, $n_{1}=n$, $k_{1}=k=1$, and $l_{1}=1$. 
We choose the scalar $c_1$ in the definition of $\Theta_{n/1}(\Lambda)$ so that $\Theta_{n/1}(\Lambda)=\Theta_{n}(\Lambda)$; 
for example, we can take $c_1=\frac{1}{2}$.  
If we set $y=y_{1}$ and $z=z_{1}$, then the identity in \cref{09847525} becomes the identity in the statement of this corollary.
\end{proof}

\begin{corollary}
\label{cor.8935}
For  all $\alpha,\beta\in\ZZ_n$ and all $y \in \CC$,
\begin{equation}
\label{identity.for.(n,k)=(n,1)}
	\sum_{r\in\ZZ_n}\frac{\theta_{\beta-\alpha}(0)}{\theta_{\beta-\alpha-r}(-\tau)\theta_{r}(\tau)}\theta_{\beta-r}(y)\theta_{\alpha+r}(y+(2-n)\tau)=0.
\end{equation}
\end{corollary}
\begin{proof}
Since $\theta(0)=0$, the term \cref{42905872} is zero when $z=y-n\tau$. The result therefore follows by setting $z=y-n\tau$ in 
\cref{42905873}.
\end{proof}

%%%%%%%%%%%%%%%%%%%%%%%%%%%%%%%%%%%%%%%%%%%%%%%%%%%%%%%%%%%%%%%%%%%%
\subsection{The point modules and point variety for $Q_{n,k}(E,\tau)$}
\label{sect.pt.mods}
%%%%%%%%%%%%%%%%%%%%%%%%%%%%%%%%%%%%%%%%%%%%%%%%%%%%%%%%%%%%%%%%%%%% 
%%%%%%%%%%%%%%%%%%%%%%%%%%%%%%%%%%%%%%%%%%%%%%%%%%%%%%%%%%%%%%%%%%%% 

The ``simplest'' representations of an associative algebra are its one dimensional modules. The simplest representations of a
connected graded $\Bbbk$-algebra are its point modules (defined in \Cref{ssect.pt.mods}).

\begin{proposition}
\label{prop.pt.mods}
Let $V=\Theta_{n/k}(\L)$ and view $Q_{n,k}(E,\tau)$ as a quotient of  the tensor algebra $TV$.
Let $\sfz \in X_{n/k}$. Make the vector space $M(\sfz)=\CC v_0 \oplus \CC v_1 \oplus \cdots$ a left $TV$-module via
$$
x \cdot v_i \; :=\; x(\s^{-i}(\sfz)) v_{i+1}
$$
for each $x \in V$.\footnote{Strictly speaking,  $x(\s^{-i}(\sfz))$ makes no sense so one should interpret it as the value of the 
theta function $x$ at a preimage of $\s^{-i}(\sfz)$ in $\CC^g$. Thus, if $\sfz \in \CC^g$, then $M(\sfz)$ is the graded module with basis $v_0,v_1,\ldots$ and action $$x_\a\cdot v_i=w_\a(\sfz-i(\sfk+\sfl-\sfn)\tau)v_{i+1},$$
and $M(\sfz) \cong M(\sfz')$ if and only if $\sfz$ and $\sfz'$ have the same image in $X_{n/k}$.} 
Then $M(\sfz)$ is a left $Q_{n,k}(E,\tau)$-module and is a point module with respect to the grading 
$\deg v_i=i$. 
\end{proposition}
\begin{pf}
If $\sfy \in \CC^p$, then $f(\sfy) \ne 0$ for some  $f \in \Theta_{n/k}(\L)$. It follows that
$M(\sfz)$ is generated by $v_0$ as a $TV$-module. 
Hence $M(\sfz)$ is a point module for $TV$. To prove the result it suffices to show that every quadratic relation for $Q_{n,k}(E,\tau)$ 
annihilates all $v_i$'s. Since
$$
(x \otimes x')\cdot v_i \;=\;  x(\s^{-i-1}(\sfz)) \, x'(\s^{-i}(\sfz))\, v_{i+2}
$$
it suffices to show that the quadratic relations for $Q_{n,k}(E,\tau)$ vanish at $(\s^{-i-1}(\sfz),\s^{-i}(\sfz))$ for all $i\ge 0$.
But \cref{43258924} shows exactly that.
\end{pf}

\begin{corollary}
\label{cor.commutative}
If $Q_{n,k}(E,\tau)$ is commutative, then $(n-k_i-l_i)\tau=0$ for all $i$. In particular, $(n-k-1)\tau=0$. 
\end{corollary}
\begin{proof}
Assume $Q_{n,k}(E,\tau)$ is commutative. 

Let $\sfz \in \CC^g$ and define the point module $M(\sfz)$ as in the previous footnote.
Every point module for $Q_{n,k}(E,\tau)$ is annihilated by a two-sided
ideal $I$ with the property that $Q_{n,k}(E,\tau)/I$ is a polynomial ring on one variable. 
Let $p \in \PP(V^*)=\PP(\Theta_n(\L)^*)$ denote the vanishing locus of $I$. 
Since $I$ annihilates the basis vectors $v_i$  for $M(\sfz)$ for all $i \ge 0$,
$$
\big(w_0(\sfz-i(\sfk+\sfl-\sfn)\tau), \ldots, w_{n-1}(\sfz-i(\sfk+\sfl-\sfn)\tau)\big) \;=\; p
$$
for all $i\ge 0$. In other words, the image of $\sfz-i(\sfk+\sfl-\sfn)\tau$ under the composition
$\CC^g \to E^g \to \PP(V^*)$ is the same for all $i \ge 0$. Hence the points $\sfz-i(\sfk+\sfl-\sfn)\tau \in E^g$  
belong to the same $\Sigma_{n/k}$-orbit for all $i \ge0 $. Now assume $\sfz=0$. 
Since $\Sigma_{n/k}$ fixes the point $0 \in E^g$ it follows that  $(\sfk+\sfl-\sfn)\tau=0$. 
In particular, $0=(k_1+l_1-n)\tau=(k+1-n)\tau$.  
\end{proof}

\subsubsection{Other point modules}
We are not saying that the point modules in \cref{prop.pt.mods} are all the point modules for $Q_{n,k}(E,\tau)$. 
Indeed, this is not the case for $Q_{4,1}(E,\tau)$, \cite[Rmk.~(ii), p.~270]{SS92}, or $Q_{8,3}(E,\tau)$.
For generic $\tau\in E$, a computation using the relations $R_{ij}$ and $R'_{ij}$ in \cite[Prop.~3.24]{CKS1} shows that
$$
\frac{Q_{8,3}(E,\tau)}{(x_0,x_2,x_4,x_6)} \; = \; \frac{\CC[x_1,x_3,x_5,x_7]}{(x_1-x_5,x_3-x_7)(x_1+x_5,x_3+x_7)}
$$
where $\CC[x_1,x_3,x_5,x_7]$ is a polynomial ring. Similarly
$$
\frac{Q_{8,3}(E,\tau)}{(x_1,x_3,x_5,x_7)} \; = \; \frac{\CC[x_0,x_2,x_4,x_6]}{(x_0-x_4,x_2-x_6)(x_0+x_4,x_2+x_6)}.
$$
Thus, for $Q_{8,3}(E,\tau)$ there are point modules parametrized by 4 lines in $\PP(V^*)$ that do not lie on $X_{8/3} \cong E^2$.

%%%%%%%%%%%%%%%%%%%%%%%%%%%%%%%%%%%%%%%%%%%%%%%%%%%%%%%%%%%%%%%%%%%%%%%%%%%%%%%%
%%%%%%%%%%%%%%%%%%%%%%%%%%%%%%%%%%%%%%%%%%%%%%%%%%%%%%%%%%%%%%%%%%%%%%%%%%%%%%%%
\appendix
%%%%%%%%%%%%%%%%%%%%%%%%%%%%%%%%%%%%%%%%%%%%%%%%%%%%%%%%%%%%%%%%%%%%%%%%%%%%%%%%
%%%%%%%%%%%%%%%%%%%%%%%%%%%%%%%%%%%%%%%%%%%%%%%%%%%%%%%%%%%%%%%%%%%%%%%%%%%%%%%%

%%%%%%%%%%%%%%%%%%%%%%%%%%%%%%%%%%%%%%%%%%%%%%%%%%%%%%%%%%%%%%%%%%%%%%%%%%%%%%%%
%%%%%%%%%%%%%%%%%%%%%%%%%%%%%%%%%%%%%%%%%%%%%%%%%%%%%%%%%%%%%%%%%%%%%%%%%%%%%%%%
\section{A detailed proof of Odesskii's theta identity}
\label{sec.prf.th.id}
%%%%%%%%%%%%%%%%%%%%%%%%%%%%%%%%%%%%%%%%%%%%%%%%%%%%%%%%%%%%%%%%%%%%%%%%%%%%%%%%
%%%%%%%%%%%%%%%%%%%%%%%%%%%%%%%%%%%%%%%%%%%%%%%%%%%%%%%%%%%%%%%%%%%%%%%%%%%%%%%%

We will now prove the formulation in \cref{09847525} of the identity on page 1153 of \cite{Od-survey}. 
We follow Odesskii's argument but with more detail. 

Let $\varphi(u,v,\sfy,\sfz)=\varphi(u,v,y_{1},\ldots,y_{g},z_{1},\ldots,z_{g})$ be the left-hand side of the equation of \cref{09847525} minus the right-hand side.

The next two lemmas are proved by straightforward computation using \cref{prop.theta.func.per.kl}.

\begin{lemma}\label{78542905}
	$\varphi(u,v,\sfy,\sfz)$ is periodic with respect to $y_{1},\ldots,y_{g},z_{1},\ldots,z_{g}$ with period $1$. Moreover
	\begin{align*}
		&\varphi(u,v,y_{1},\ldots,y_{s}+\eta,\ldots,y_{g},z_{1},\ldots,z_{g})\\
		&=e((1-\delta_{s,1})y_{s-1}-n_{s}y_{s}+(1-\delta_{s,g})y_{s+1}+c_{s})\varphi(u,v,y_{1},\ldots,y_{g},z_{1},\ldots,z_{g})\\
		&\varphi(u,v,y_{1},\ldots,y_{g},z_{1},\ldots,z_{s}+\eta,\ldots,z_{g})\\
		&=e((1-\delta_{s,1})z_{s-1}-n_{s}z_{s}+(1-\delta_{s,g})z_{s+1}-\delta_{s,1}nu-\delta_{s,g}nv+c_{s})\varphi(u,v,y_{1},\ldots,y_{g},z_{1},\ldots,z_{g})
	\end{align*}
\end{lemma}

\begin{lemma}\label{409258254}
	$\varphi(u,v,\sfy,\sfz)$ is periodic with respect to $u,v$ with period $1$. Moreover
	\begin{align*}
		&\varphi(u+\eta,v,y_{1},\ldots,y_{g},z_{1},\ldots,z_{g})\\
		&=e\Big(-n(z_{1}+ku)+\sum_{i=1}^{g}k_{i}(c_{i}+\eta)-\tfrac{1}{2}(nk-n+k+1)\eta\Big)\varphi(u,v,y_{1},\ldots,y_{g},z_{1},\ldots,z_{g}),\\
		&\varphi(u,v+\eta,y_{1},\ldots,y_{g},z_{1},\ldots,z_{g})\\
		&=e\Big(-n(z_{g}+k'v)+\sum_{i=1}^{g}l_{i}(c_{i}+\eta)-\tfrac{1}{2}(nk'-n+k'+1)\eta\Big)\varphi(u,v,y_{1},\ldots,y_{g},z_{1},\ldots,z_{g}).
	\end{align*}
\end{lemma}

\begin{lemma}
	$\varphi(u,v,\sfy,\sfz)$ is uniquely extended to a holomorphic function on $\bbC^{g+2}$.
\end{lemma}

\begin{proof}
	First we will show that $\varphi$ can be extended to
	\begin{equation*}
		D\;:=\;\{(u,v,x_{1},\ldots,x_{g},y_{1},\ldots,y_{g})\in\bbC^{2g+2}\;|\;\text{$v\notin\tfrac{1}{n}\Lambda$, $y_{s}-z_{s}\notin\Lambda$ for all $s$}\}.
	\end{equation*}
	Since $Sw_{i}=e(\frac{ki}{n})w_{i}$ and $Tw_{i}=w_{i+1}$, it suffices to show that $\varphi$ can be extended to each point of $D$ with $u=0$. Denote the $(g+1)$ terms of the left-hand side of the equation of \cref{09847525} by $f_{0},\ldots,f_{g}$ and each summand of the right-hand side by $g_{r}$ ($r\in\ZZ_n$) including the multiplier $\frac{1}{n}\theta(\frac{1}{n})\cdots\theta(\frac{n-1}{n})$. Since $f_{i}$ ($i\neq 0$) and $g_{r}$ ($r\neq\beta-\alpha$) are holomorphic even at $u=0$, we only have to see that $f_{0}-g_{\beta-\alpha}$ can be extended to $u=0$ holomorphically.
	Using the identity (27) in \cite[p.~1150]{Od-survey}, we obtain
	\begin{align*}
		(f_{0}-g_{\beta-\alpha})(u,v,\sfx,\sfy)
		&\;=\;\frac{\theta(-nu+y_{1}-z_{1})}{\theta(-nu)\theta(y_{1}-z_{1})}w_{\alpha}(\sfy+\sfk u)w_{\beta}(\sfz+\sfl v)\\
		&\;\phantom{{}={}}\;-\tfrac{1}{n}\theta(\tfrac{1}{n})\cdots\theta(\tfrac{n-1}{n})\frac{\theta_{(\beta-\alpha)k}(-u+v)}{\theta_{0}(-u)\theta_{(\beta-\alpha)k}(v)}w_{\alpha}(\sfy)w_{\beta}(\sfz+\sfk u+\sfl v)\\
		&\;=\;\frac{\theta(-nu+y_{1}-z_{1})}{\theta(-nu)\theta(y_{1}-z_{1})}w_{\alpha}(\sfy+\sfk u)w_{\beta}(\sfz+\sfl v)\\
		&\;\phantom{{}={}}\;-\frac{e(\tfrac{n(n-1)}{2}u)\theta_{0}(u)\cdots\theta_{n-1}(u)}{\theta(nu)\theta_{1}(0)\cdots\theta_{n-1}(0)}\frac{\theta_{(\beta-\alpha)k}(-u+v)}{\theta_{0}(-u)\theta_{(\beta-\alpha)k}(v)}w_{\alpha}(\sfy)w_{\beta}(\sfz+\sfk u+\sfl v)
	\end{align*}
	Applying \cref{prop.official.theta.basis}\cref{item.official.theta.basis.minus} to $\theta(-nu)$ and $\theta_{0}(-u)$,
	\begin{align*}
		(f_{0}-g_{\beta-\alpha})(u,v,\sfx,\sfy)
		&\;=\;\frac{\theta(-nu+y_{1}-z_{1})}{-e(-nu)\theta(nu)\theta(y_{1}-z_{1})}w_{\alpha}(\sfy+\sfk u)w_{\beta}(\sfz+\sfl v)\\
		&\;\phantom{{}={}}\;-\frac{e(\tfrac{n(n-1)}{2}u)\theta_{0}(u)\cdots\theta_{n-1}(u)}{\theta(nu)\theta_{1}(0)\cdots\theta_{n-1}(0)}\frac{\theta_{(\beta-\alpha)k}(-u+v)}{-e(-nu)\theta_{0}(u)\theta_{(\beta-\alpha)k}(v)}w_{\alpha}(\sfy)w_{\beta}(\sfz+\sfk u+\sfl v)\\
		&\;=\;\frac{1}{-e(-nu)\theta(nu)}
		\bigg(
		\frac{\theta(-nu+y_{1}-z_{1})}{\theta(y_{1}-z_{1})}w_{\alpha}(\sfy+\sfk u)w_{\beta}(\sfz+\sfl v)\\
		&\;\phantom{{}={}}\;-\frac{e(\tfrac{n(n-1)}{2}u)\theta_{1}(u)\cdots\theta_{n-1}(u)}{\theta_{1}(0)\cdots\theta_{n-1}(0)}\frac{\theta_{(\beta-\alpha)k}(-u+v)}{\theta_{(\beta-\alpha)k}(v)}w_{\alpha}(\sfy)w_{\beta}(\sfz+\sfk u+\sfl v)
		\bigg).
	\end{align*}
	Thus $(f_{0}-g_{\beta-\alpha})(u,v,\sfx,\sfy)$ is a product of a meromorphic function of $u$ that has a pole of order one at $u=0$ and a holomorphic function of $u$ that has zero at $u=0$. So it is holomorphic at $u=0$.
	
	Similarly, we can show that $\varphi(u,v,\sfy,\sfz)$ can be extended holomorphically to
	\begin{equation*}
		\{(u,v,x_{1},\ldots,x_{g},y_{1},\ldots,y_{g})\in\bbC^{2g+2}\;|\;\text{$u\notin\tfrac{1}{n}\Lambda$, $y_{i}-z_{i}\notin\Lambda$ for all $i$}\}
	\end{equation*}
	and
	\begin{equation*}
		\{(u,v,x_{1},\ldots,x_{g},y_{1},\ldots,y_{g})\in\bbC^{2g+2}\;|\;\text{$u,v\notin\tfrac{1}{n}\Lambda$, $y_{i}-z_{i}\notin\Lambda$ for all $i\neq s$}\}
	\end{equation*}
	for each $1\leq s\leq g$.
	
	To summarize, $\varphi(u,v,\sfy,\sfz)$ can be extended to a holomorphic function on $\CC^{g+2}-Z$, where $Z$ is the set consisting of all $(u,v,\sfy,\sfz)\in\bbC^{2g+2}$ satisfying at least two of the $g+2$ conditions: $u\in\tfrac{1}{n}\Lambda$, $v\in\tfrac{1}{n}\Lambda$, $y_{1}-z_{1}\in\Lambda,\ldots,y_{g}-z_{g}\in\Lambda$. Since $Z$ is an analytic subset of codimension two, $f$ can further be extended holomorphically to $\CC^{g+2}$ by the Second Riemann Theorem (\cite[p.~132]{GrauertRemmert}).
\end{proof}

\begin{lemma}
	There is a unique family of holomorphic functions $\psi_{\lambda,\nu}\colon\bbC^{2}\to\bbC$ ($\lambda,\nu\in\ZZ_n$) such that
	\begin{equation*}
		\varphi(u,v,\sfy,\sfz)\;=\;\sum_{\lambda,\nu\in\ZZ_n}\psi_{\lambda,\nu}(u,v)w_{\lambda}(\sfy)w_{\nu}(\sfz+\sfk u+\sfl v).
	\end{equation*}
\end{lemma}

\begin{proof}
	\cref{78542905} shows that $\varphi(u,v,\sfy,\sfz)$ regarded as a function of $\sfy=(y_{1},\ldots,y_{g})$ belongs to $\Theta_{n/k}(\Lambda)$ (with the same constant $\sfc$). Thus there are unique $\rho_{\lambda}(u,v,\sfz)\in\CC$ such that
	\begin{equation}\label{eq.th.id.fst}
		\varphi(u,v,\sfy,\sfz)\;=\;\sum_{\lambda\in\ZZ_{n}}\rho_{\lambda}(u,v,\sfz)w_{\lambda}(\sfy).
	\end{equation}
	
	Note that $\rho_{\lambda}(u,v,\sfz)$ are holomorphic functions on $u$, $v$, and $\sfz$: we know that the functions $w_{\lambda}$ are linearly independent, so
	\begin{equation*}
		\operatorname{span}\{(w_{\lambda}(\sfy))_{\lambda}\;|\;\sfy\in\CC^{g}\}\;=\;\CC^{n}.
	\end{equation*}
	We can thus find, for any fixed $\lambda_{0}$, arguments $\sfy_{i}$ and scalars $t_i$ such that $\sum_i t_i w_\lambda(\sfy_i)$ is $\delta_{\lambda_0,\lambda}$ (Kronecker delta). This then implies that
	\begin{equation*}
		\rho_{\lambda_0}(u,v,\sfz)\;=\;\sum_i t_i \varphi(u,v,\sfy_{i},\sfz)
	\end{equation*}
    is holomorphic, as claimed.
	
	The quasi-periodicity with respect to $\sfz$ in \cref{78542905} implies that each $\rho_{\lambda}(u,v,\sfz)$ has the same periodicity. It is easy to see that $\rho_{\lambda}(u,v,\sfz-\sfk u-\sfl v)$ belongs to $\Theta_{n/k}(\Lambda)$ with the \emph{same} constant $\sfc$. Thus a similar argument shows that there are unique holomorphic functions $\psi_{\lambda,\nu}(u,v)$ such that
	\begin{equation*}
		\rho_{\lambda}(u,v,\sfz-\sfk u-\sfl v)\;=\;\sum_{\nu\in\ZZ_{n}}\psi_{\lambda,\nu}(u,v)w_{\nu}(\sfz),
	\end{equation*}
	that is,
	\begin{equation}\label{eq.th.id.snd}
		\rho_{\lambda}(u,v,\sfz)\;=\;\sum_{\nu\in\ZZ_{n}}\psi_{\lambda,\nu}(u,v)w_{\nu}(\sfz+\sfk u+\sfl v).
	\end{equation}
	Combining \cref{eq.th.id.fst} and \cref{eq.th.id.snd}, we obtain the desired formula.
\end{proof}

\cref{409258254} implies that $\psi_{\lambda,\nu}$ is periodic in each variable with period $1$ and, moreover,
\begin{equation*}
	\psi_{\lambda,\nu}(u+\eta,v) \;=\; e(nv)\psi_{\lambda,\nu}(u,v)\qquad\text{and}\qquad\psi_{\lambda,\nu}(u,v+\eta)\; =\; 
	e(nu)\psi_{\lambda,\nu}(u,v).
\end{equation*}
Fix $v$. Since $e(nv)$ does not depend on $u$, there exist $a\in\bbC$ and $m\in\bbZ$ such that $\psi_{\lambda,\nu}(u,v)=ae(mu)$. If $a\neq 0$, then $e(m\eta)=e(nv)$ and this implies $v\in \frac{1}{n}\Lambda$. Therefore for all $v\in\bbC - \frac{1}{n}\Lambda$ and $u\in\bbC$, $\psi_{\lambda,\nu}(u,v)=0$. By continuity, $\psi_{\lambda,\nu}(u,v)=0$ 
for all $u,v\in\bbC$. This completes the proof of \cref{09847525}.

\printindex
\clearpage

\bibliography{biblio4}
\bibliographystyle{customamsalpha}

\end{document}